\documentclass[12pt]{article}

\usepackage{amsmath}
\usepackage{amssymb}
\usepackage{amsthm}
\usepackage[a4paper, margin=2.54cm]{geometry}
\usepackage{lmodern}
\usepackage{microtype}
\usepackage{mathtools}
\usepackage{array}
\usepackage{xspace}
\usepackage{tikz}
\usetikzlibrary{arrows.meta,positioning}
\tikzset{
  permdot/.style={circle,fill=black,inner sep=1.1pt},
  covedge/.style={thick,black!75},
  cutdot/.style={circle,draw=black,fill=white,thick,inner sep=1.6pt},
  peelarr/.style={-{Stealth[length=4pt]},thick},
  auxarr/.style={-{Stealth[length=3pt]},densely dashed,black!55},
  corebox/.style={rounded corners=3pt,densely dashed,black!55}
}
\usepackage{comment}
\usepackage[colorlinks=true,linkcolor=blue,citecolor=blue]{hyperref}
\usepackage{cleveref}
\crefformat{equation}{(#2#1#3)}
\Crefformat{equation}{(#2#1#3)}
\crefrangeformat{equation}{(#3#1#4)--(#5#2#6)}
\Crefrangeformat{equation}{(#3#1#4)--(#5#2#6)}
\crefmultiformat{equation}{(#2#1#3)}{ and (#2#1#3)}{, (#2#1#3)}{ and (#2#1#3)}
\Crefmultiformat{equation}{(#2#1#3)}{ and (#2#1#3)}{, (#2#1#3)}{ and (#2#1#3)}
\crefname{enumi}{part}{parts}
\Crefname{enumi}{Part}{Parts}
\creflabelformat{enumi}{(#2#1#3)}
\crefname{section}{Section}{Sections}
\Crefname{section}{Section}{Sections}
\crefname{subsection}{Section}{Sections}
\Crefname{subsection}{Section}{Sections}
\crefname{theorem}{Theorem}{Theorems}
\Crefname{theorem}{Theorem}{Theorems}
\crefname{proposition}{Proposition}{Propositions}
\Crefname{proposition}{Proposition}{Propositions}
\crefname{lemma}{Lemma}{Lemmas}
\Crefname{lemma}{Lemma}{Lemmas}
\crefname{corollary}{Corollary}{Corollaries}
\Crefname{corollary}{Corollary}{Corollaries}
\crefname{conjecture}{Conjecture}{Conjectures}
\Crefname{conjecture}{Conjecture}{Conjectures}
\crefname{definition}{Definition}{Definitions}
\Crefname{definition}{Definition}{Definitions}
\crefname{remark}{Remark}{Remarks}
\Crefname{remark}{Remark}{Remarks}
\crefname{example}{Example}{Examples}
\Crefname{example}{Example}{Examples}
\crefname{table}{Table}{Tables}
\Crefname{table}{Table}{Tables}
\crefname{figure}{Figure}{Figures}
\Crefname{figure}{Figure}{Figures}
\usepackage{amsrefs}

\allowdisplaybreaks

\usepackage{aliascnt}
\newtheorem{theorem}{Theorem}[section]
\newaliascnt{proposition}{theorem}
\newtheorem{proposition}[proposition]{Proposition}
\aliascntresetthe{proposition}
\newaliascnt{lemma}{theorem}
\newtheorem{lemma}[lemma]{Lemma}
\aliascntresetthe{lemma}
\newaliascnt{corollary}{theorem}
\newtheorem{corollary}[corollary]{Corollary}
\aliascntresetthe{corollary}
\newaliascnt{conjecture}{theorem}

\aliascntresetthe{conjecture}
\theoremstyle{definition}
\newaliascnt{definition}{theorem}
\newtheorem{definition}[definition]{Definition}
\aliascntresetthe{definition}
\newaliascnt{remark}{theorem}
\newtheorem{remark}[remark]{Remark}
\aliascntresetthe{remark}
\newaliascnt{example}{theorem}
\newtheorem{example}[example]{Example}
\aliascntresetthe{example}

\newcommand{\prig}{pseudorigid\xspace}
\newcommand{\Prig}{Pseudorigid\xspace}
\newcommand{\prigity}{pseudorigidity\xspace}
\newcommand{\Prigity}{Pseudorigidity\xspace}
\newcommand{\obt}{obtuse\xspace}

\newcommand{\smth}{smooth\xspace}
\newcommand{\Smth}{Smooth\xspace}
\newcommand{\smthness}{smoothness\xspace}
\newcommand{\psm}{pseudosmooth\xspace}
\newcommand{\Psm}{Pseudosmooth\xspace}
\newcommand{\psmness}{pseudosmoothness\xspace}
\newcommand{\good}{good\xspace}
\newcommand{\cb}{corner-blocked\xspace}
\newcommand{\cbness}{corner-blockedness\xspace}

\newcommand{\goodness}{goodness\xspace}
\newcommand{\Goodness}{Goodness\xspace}
\newcommand{\spc}{special\xspace}

\newcommand{\spcness}{specialness\xspace}
\newcommand{\pal}{palette\xspace}      % a color set admitting a peeling datum

\newcommand{\pals}{palettes\xspace}
\newcommand{\fpal}{framed \pal}        % a palette containing the frame

\newcommand{\pset}{A}                  % the symbol for a palette
\newcommand{\pvs}{prehomogeneous\xspace}

\newcommand{\pvsness}{prehomogeneity\xspace}
\newcommand{\wpvs}{pseudo\pvs}

\newcommand{\wpvsness}{pseudo\pvsness}
\newcommand{\dgd}{diagonal deficiency\xspace}

\newcommand{\pwd}{pairwise deficiency\xspace}

\newcommand{\argmin}{\operatorname{argmin}}
\newcommand{\orbs}{\mathcal{O}}
\newcommand{\orb}{\mathfrak{o}}
\newcommand{\pstn}{\mathfrak{p}}
\newcommand{\typ}{\mathfrak{t}}
\newcommand{\topp}{\mathfrak{s}}
\newcommand{\rbm}{\mathfrak{r}}  % the radical of a bimodule
\newcommand{\sms}{\operatorname{ss}}
\newcommand{\cyclem}{\mathcal{C}_-}
\newcommand{\cyclep}{\mathcal{C}_+}

\newcommand{\ZZ}{\mathbb{Z}}
\newcommand{\NN}{\mathbb{N}}
\newcommand{\CC}{\mathbb{C}}
\newcommand{\mix}{\mathrm{mix}}   % subscript: the mixed Peirce part
\newcommand{\diag}{\mathrm{diag}} % subscript: the diagonal Peirce part
\newcommand{\Lec}{\mathrm{Lec}}   % subscript: Leclerc's multisegment
\newcommand{\op}{\mathrm{op}}
\newcommand{\HH}{\mathrm{HH}}
\newcommand{\bd}{\boldsymbol{d}}  % graded dimension vectors
\newcommand{\be}{\boldsymbol{e}}

\newcommand{\GL}{\operatorname{GL}}

\newcommand{\pr}{\operatorname{pr}}
\newcommand{\tr}{\operatorname{tr}}

\newcommand{\Ext}{\operatorname{Ext}}
\newcommand{\ext}{\operatorname{ext}}
\newcommand{\Hom}{\operatorname{Hom}}
\newcommand{\End}{\operatorname{End}}
\newcommand{\Aut}{\operatorname{Aut}}
\newcommand{\ad}{\operatorname{ad}}
\newcommand{\grdim}{\operatorname{grdim}}
\newcommand{\Comp}{\operatorname{Comp}}
\newcommand{\cmp}{C}                       % an irreducible component of Lusztig's nilpotent variety
\newcommand{\Mfr}{\mathfrak{M}}
\newcommand{\mm}{\mathfrak{m}}
\newcommand{\kk}{\Bbbk}
\newcommand{\alg}{\mathcal{A}}   % the associative algebra of a pair
\newcommand{\bmdl}{\mathcal{D}}   % the bimodule of a pair
\newcommand{\XX}{\mathbb{X}}  % the linkage relation of a multisegment
\newcommand{\YY}{\mathbb{Y}}  % the homomorphism relation of a multisegment
\newcommand{\DD}{\mathbb{D}}  % the diagonal part of \YY
\newcommand{\Kk}{K}         % the index set of a multisegment
\newcommand{\qvert}{Q_0}    % the vertex set of the quiver Q
\newcommand{\qarr}{Q_1}     % the arrow set of the quiver Q
\newcommand{\algbss}{\alpha}     % basis vectors of \alg (matrix units)
\newcommand{\bmdlbss}{\beta}    % basis vectors of \bmdl (matrix units)
\newcommand{\Ff}{\mathcal{F}}   % the fringe subbimodule
\newcommand{\Vv}{\mathcal{V}}   % the quotient bimodule \bmdl/\Ff
\newcommand{\pquo}{\theta}      % the quotient map \bmdl \to \Vv
\newcommand{\dfr}{\mathfrak{d}}  % the corank invariants \dfr_\rsub, \dfr_\psub
\newcommand{\Sop}{S}
\newcommand{\Dop}{D}
\newcommand{\rsub}{\sharp}       % subscript: rigidity (equal parameters)
\newcommand{\psub}{\flat}        % subscript: \prigity (free parameters)
\newcommand{\rank}{\operatorname{rank}}
\newcommand{\rad}{\operatorname{rad}}
\newcommand{\gen}{\mathrm{gen}}         % superscript: the set of generators of \bmdl
\newcommand{\cork}{\operatorname{corank}}
\newcommand{\coker}{\operatorname{coker}}
\newcommand{\im}{\operatorname{im}}
\newcommand{\Span}{\operatorname{span}}
\newcommand{\wo}{w_0}                 % the longest element, i \mapsto n+1-i
\newcommand{\rc}{\rho}                % reverse-complement: \rc\sigma = \wo\sigma\wo
\newcommand{\PA}{\mathcal{P}}          % the class of \psm permutations
\newcommand{\Av}{\operatorname{Av}}
\newcommand{\std}{\operatorname{std}}
\newcommand{\ptst}{U}                  % a set of positions of a permutation
\newcommand{\Gc}{\mathcal{G}}          % the colored bipartite graph
\newcommand{\win}{\operatorname{win}}     % the window of a record
\newcommand{\Cov}{\operatorname{Cov}}  % the set of minimal inversions
\newcommand{\Fr}{\operatorname{Fr}}    % the frame of a permutation
\newcommand{\pos}{\operatorname{pos}}  % the position of a point
\newcommand{\Wcr}{W\xspace}    % the first position      = the leftmost point
\newcommand{\Ncr}{N\xspace}    % the position of value n = the highest point
\newcommand{\Scr}{S\xspace}    % the position of value 1 = the lowest point
\newcommand{\Ecr}{E\xspace}    % the last position       = the rightmost point
\newcommand{\SWcr}{SW\xspace}  % the pair \Wcr,\Scr and its \core
\newcommand{\NEcr}{NE\xspace}  % the pair \Ncr,\Ecr and its \core
\newcommand{\blam}{\boldsymbol{\lambda}}   % a color of type \blam
\newcommand{\bmu}{\boldsymbol{\mu}}        % a color of type \bmu
\newcommand{\prj}{\varpi}                  % projection {1,...,n} -> {1,...,k}
\newcommand{\prjY}{\varpi_{\YY}}           % induced map \YY_\sigma -> \YY_\tau
\newcommand{\lmin}{\mathfrak{l}_{\min}}    % first position of a block
\newcommand{\lmax}{\mathfrak{l}_{\max}}    % last position of a block
\newcommand{\cl}{\mathfrak{l}}             % corner lift of an inversion, or a color
\newcommand{\core}{corner core\xspace}

\newcommand{\cores}{corner cores\xspace}
\newcommand{\Cm}{C_{-}}
\newcommand{\Cp}{C_{+}}
\newcommand{\Cpm}{C_{\pm}}
\newcommand{\Rm}{R_{-}}
\newcommand{\Rp}{R_{+}}
\newcommand{\Rpm}{R_{\pm}}
\newcommand{\clr}{\mathfrak{k}}            % a color
\newcommand{\crn}{\mathbf{c}}              % a corner
\newcommand{\cyc}{\mathfrak{c}}            % cycle rank of a graph
\newcommand{\acyc}{t}                     % number of acyclic components of a graph
\newcommand{\cgr}{\Gamma}                 % the cover graph
\newcommand{\igr}{G}                      % the inversion graph

\title{A hereditary theorem for rigid and \prig components of Lusztig's nilpotent varieties in type $A$}
\author{Erez Lapid and Mark Shusterman\\[3pt]
\normalsize Department of Mathematics, Weizmann Institute of Science, Rehovot 7610001, Israel\\[2pt]
\normalsize\texttt{erez.m.lapid@gmail.com}, \texttt{mark.shusterman@weizmann.ac.il}}
\date{}

\begin{document}
\maketitle

\begin{abstract}
\noindent
For irreducible components $\cmp_\sigma$ of Lusztig's nilpotent varieties of type $A$
with graded dimension $(1,2,\dots,n,n-1,\dots,1)$ arising from permutations $\sigma\in S_n$, we 
characterize the property that $\cmp_\sigma\oplus\cmp_\sigma$ is an irreducible component
combinatorially in terms of $\sigma$. The permutations that occur, which we call
\emph{\psm}, are described by a recursion on direct sums and deleting suitable corners, whose terminal
cases are the permutations obtained from
\[
3412,\quad 4231,\quad 35142,\quad 42513,\quad 45312,\quad 426153,\quad 463152,\quad 526413
\]
by inflating the entries into consecutive decreasing blocks. The main new input is a
hereditary property valid
for decompositions of multisegments with disjoint extreme points.
\end{abstract}

\setcounter{tocdepth}{1}
\tableofcontents

\section{Introduction}\label{sec:intro}

Let $Q$ be the equioriented quiver $A_N$.
Let $R_Q(V)$ be the space of representations of $Q$ on a finite-dimensional graded vector space $V$ over an algebraically closed field $\kk$ of
arbitrary characteristic and let $\Lambda(V)\subset T^*(R_Q(V))$ be Lusztig's nilpotent variety \cite{MR1088333}.
It parametrizes modules over the preprojective algebra $\Pi=\Pi(Q)$ and its irreducible components are the closures of the conormal bundles of the $G_V$-orbits on $R_Q(V)$.
As we vary $V$, they are indexed by multisegments, which are formal sums of segments $[a,b]$, $1\le a\le b\le N$.
An irreducible component $\cmp_\mm$ (and by extension, the multisegment $\mm$ itself) of $\Lambda(V)$ is
\emph{rigid} if it contains an open orbit, equivalently a point $x$ with $\Ext^1_\Pi(x,x) = 0$.
These are important for the representation theory of the general linear groups over a non-archimedean local field \cites{MR584084,MR3866895,MR4847251},
and likewise for quiver Hecke (KLR) algebras, where the analogous notion is that of a \emph{real} simple module \cites{MR3314831,MR4016058}.
Rigid modules over preprojective algebras are also the representation-theoretic counterpart of cluster monomials in the work of Geiss--Leclerc--Schr\"oer
\cites{MR2242628,MR2822235}, a point of view that developed into the monoidal categorification of cluster algebras \cites{MR2682185,MR3758148,MR4094378}.
The first example of a non-rigid multisegment was given by Leclerc: $[1,2]+[2,4]+[3,3]+[4,5]$ \cite{MR1959765}.
In the special case where the multisegment $\mm=\sum_{i=1}^n[a_i,b_i]$ is regular in the sense that the $a_i$'s are distinct and the $b_i$'s are distinct
there is a simple geometric/combinatorial criterion for the rigidity of $\mm$ \cites{MR3866895,MR4757326}.\footnote{The corrigendum \cite{MR4757326}
concerns only \cite{MR3866895}*{Corollary 9.6} and the basic cases of \cite{MR3866895}*{\S8}; all the statements of \cite{MR3866895}
that we use --- \S3.5, \S4.1, Remark~4.11, Example~4.12 and the equivalence of rigidity with smoothness --- are unaffected.}
In the typical case where
\[
\mm=\mm_\sigma=\sum_{i=1}^n[i,2n-\sigma(i)]
\]
for a permutation $\sigma\in S_n$, $\mm$ is rigid if and only if $\sigma$ is $4231$- and $3412$-avoiding \cite{MR3866895}, which is exactly the condition that
the Schubert variety in the complete flag variety of $\GL_n$ corresponding to $\sigma$ is smooth \cite{MR1051089}. (Such permutations are called smooth.)

The purpose of the current paper is to study a weaker notion of rigidity, which we call \emph{\prigity}.
By definition, an irreducible component $\cmp$ (or the corresponding multisegment $\mm$) is \emph{\prig} if $\Ext^1_\Pi(x,y) = 0$ for a generic \emph{pair} $x,y\in\cmp$.
Equivalently, $\cmp\oplus\cmp$, the closure of the orbits of direct sums $x\oplus y$, $x,y\in\cmp$, is an irreducible component of $\Lambda(V\oplus V)$ \cite{MR1944812}*{Theorem 1.2},
in which case it is parametrized by $\mm+\mm$.
In the language of \cite{MR4847234}, \prigity is self-strong-commutation.
Over $\CC$ this also implies that the element $\rho_{\cmp}$ of Lusztig's dual semicanonical basis parametrized by $\cmp$ satisfies $\rho_{\cmp}^2=\rho_{\cmp\oplus\cmp}$ \cite{MR2144987}.\footnote{In contrast, this is not true for the corresponding element of the dual canonical basis in Leclerc's example, although the component is \prig (\Cref{ex:leclerc}).}

Both rigidity and \prigity amount to generic surjectivity of certain explicit families $\Sop_\lambda$ and $\Sop_{\lambda,\mu}$ of linear maps that depend linearly on parameters.
More precisely, for any $\lambda,\mu\in\Ext^1_Q(M,M)^*$
\[
\Sop_{\lambda,\mu} : \End_Q(M)\to \Ext^1_Q(M,M)^* ,
\]
is defined by $\Sop_{\lambda,\mu}(a) = a\lambda - \mu a$ and $\Sop_\lambda=\Sop_{\lambda,\lambda}$ is the differential of the orbit map
$\Aut_Q(M)\to \Ext^1_Q(M,M)^*$ of $\lambda$.

The main result of the paper is to characterize combinatorially the permutations $\sigma\in S_n$ for which $\mm_\sigma$ is \prig.

To explain this, we first recall an alternative recursive description of smooth permutations.
Let us say that the \Wcr corner of $\sigma$ is \good if $\sigma(j)<\sigma(i)$ whenever $\sigma(j),\sigma(i)<\sigma(1)$ with $i<j$.
Similarly for the \Scr, \Ecr, \Ncr corners, by passing to $\sigma^{-1}$, $\rc\sigma$, $\rc\sigma^{-1}$ where $\rc$ denotes the reverse-complement.
Then, a nonempty permutation $\sigma$ is smooth if and only if $\sigma$ has a \good $\Wcr$ or $\Scr$ corner and deleting it results in a smooth permutation in $S_{n-1}$.

We can analogously define recursively the class of \emph{\psm} permutations.
The permutation $1\in S_1$ is \psm.
For $n>1$, an indecomposable permutation $\sigma\in S_n$ is \psm if either it has a \Wcr, \Scr, \Ncr or \Ecr \good corner whose deletion yields a \psm permutation, \emph{or} $\sigma$ is in one of eight families, obtained from the permutations
\begin{equation} \label{eq:8quot}
   3412,\quad 4231,\quad 35142,\quad 42513,\quad 45312,\quad 426153,\quad 463152,\quad 526413,
\end{equation}
by inflating the entries by decreasing consecutive blocks of arbitrary length and keeping the relative order.
Finally, a decomposable permutation is \psm if its indecomposable components are.

Our first main result (\Cref{thm:mainA}) is that $\sigma$ is \psm if and only if $\mm_\sigma$ is \prig.

The two directions of \Cref{thm:mainA} are proved by different means.
The forward direction is proved by induction.
It provides support sets of $\lambda$ and $\mu$ that guarantee surjectivity of $\Sop_{\lambda,\mu}$.
The procedure is compatible with \good corners and the eight families above (the \emph{towers}), and eventually boils down to explicit certificates for the eight permutations in \eqref{eq:8quot}.

For the converse direction we prove the following hereditary property, which is our second main result, valid
for rigidity and \prigity alike, for arbitrary multisegments in equioriented type $A$.

\begin{theorem}\label{thm:mainB}
Let $\mm = \sum_{i \in \Kk} [a_i,b_i]$ be a multisegment and let $\Kk = I \sqcup J$ be a partition such that no beginning $a_i$ and no end
$b_i$ occurs in both blocks. If the component $\cmp_\mm$ is rigid (resp.\ \prig),
then so are $\cmp_{\mm_I}$ and $\cmp_{\mm_J}$.
\end{theorem}
For the multisegments $\mm_\sigma$ this immediately implies that \prigity is closed under pattern containment.
Therefore, for the converse direction of \Cref{thm:mainA} it suffices to show that $\mm_\sigma$ is not \prig if $\sigma$ is critical, that is
$\sigma$ is not \psm but every point deletion of $\sigma$ is \psm.
We will not classify the critical permutations here, but instead show that their cover graph $\cgr_\sigma$ (the graph of minimal inversions)
is connected and bicyclic.\footnote{For $n\le8$ this is proved by an exhaustive computation.
For $n\ge9$ we show more: $\cgr_\sigma$ is a dumbbell graph, two four-cycles
joined by a path (\Cref{cor:dumbbell}).}
This suffices since a necessary condition for \prigity is that $\cgr_\sigma$ is a pseudoforest.
(For rigidity, a necessary condition is that $\cgr_\sigma$ is acyclic.)

The operators $\Sop_{\lambda,\mu}$ are well defined for any finite-dimensional bimodule $\bmdl$ over a finite-dimensional $\kk$-algebra $\alg$
and give rise to two deficiency invariants: the \emph{\dgd} $\dfr_\rsub(\alg,\bmdl)$,
the minimal corank of $\Sop_\lambda$, and the \emph{\pwd} $\dfr_\psub(\alg,\bmdl)$,
the minimal corank of $\Sop_{\lambda,\mu}$.
They measure lack of rigidity/\prigity.
The additional data for \Cref{thm:mainB} is an idempotent decomposition $1=e_I+e_J$ in $\alg$, with corresponding Peirce decompositions
of $\alg=\alg_{\diag}\oplus\alg_{\mix}$ and $\bmdl=\bmdl_{\diag}\oplus\bmdl_{\mix}$, together with an $\alg_{\diag}$-bilinear derivation
$f:\alg\rightarrow\bmdl/\Ff$ with image $\bmdl_{\mix}/\Ff$ where $\Ff$ is an $\alg$-subbimodule contained in $\bmdl_{\mix}$.
This is a rather stringent condition. In \Cref{prop:twoblock} we show that it implies that
\[
   \dfr_\ast(\alg,\bmdl) \;\ge\; \dfr_\ast(\alg_I,\bmdl_I) + \dfr_\ast(\alg_J,\bmdl_J), \qquad \ast \in \{\rsub,\psub\}.
\]
In particular, $\dfr_\ast(\alg_I,\bmdl_I) = \dfr_\ast(\alg_J,\bmdl_J) = 0$ whenever $\dfr_\ast(\alg,\bmdl) = 0$.
For the application to \Cref{thm:mainB} the choice of $f$ readily suggests itself from the explicit description of $\alg$ and $\bmdl$ in the equioriented type $A$ case.
We get
\[
   \dfr_\ast(\mm) \;\ge\; \dfr_\ast(\mm_I) + \dfr_\ast(\mm_J), \qquad \ast \in \{\rsub,\psub\}.
\]

\Cref{thm:mainB} is already new for rigidity: for regular multisegments it is deduced a posteriori from the
classification of \cite{MR3866895}*{Theorem 7.1} (see Remark 7.6 there), but here it is proved directly
and is the main tool for the classification.
It provides a simple proof of the (known) equivalence of rigidity and smoothness for permutations.
In the forward direction we actually obtain a precise description of the open orbit in terms of non-vanishing of the minimal inversion coordinates
of $\lambda$. The converse direction, thanks to \Cref{thm:mainB},
reduces to the non-rigidity of $\mm_\sigma$ for $\sigma=3412$ and $4231$ alone,
which follows from the fact that $\cgr_\sigma$ is a four-cycle in these cases.

We do not address in this paper the representation-theoretic counterpart of \prigity, and leave it to a future occasion.

\subsection{Conventions}

Throughout, $\kk$ is an algebraically closed field, of arbitrary characteristic.
Much of what we do works for more general fields -- see \Cref{ssec:arbhered,ssec:arbnilp,ssec:arbesi}.
Unless otherwise specified, all vector spaces, linear maps, ranks, varieties, algebras, etc.~are implicitly over $\kk$;
algebras are finite-dimensional, associative and unital and modules/bimodules over them are finite-dimensional.

\section{Heredity}\label{sec:hered}

\subsection{The setup: bimodules and corank invariants}\label{ssec:pairs}

Let $\alg$ be a finite-dimensional algebra, and let $\bmdl$ be a finite-dimensional $\alg$-bimodule.
The unit group $G=\alg^\times$ is open in $\alg$ and hence a smooth linear algebraic group with Lie algebra $\alg$.
It acts on $\bmdl$ by conjugation.
Conforming to the standard terminology (in a more general setup), we say that the pair $(\alg,\bmdl)$, or simply $\bmdl$ itself, is \emph{\pvs} if $\bmdl$ admits an open $G$-orbit.

Note that for any $\lambda\in\bmdl$ the orbit map $g \mapsto g\lambda g^{-1}$ is separable \emph{in any characteristic}.
Indeed, the differential of the orbit map  at $g = 1$ is
\[
\Sop_\lambda:\alg\longrightarrow\bmdl, \ \ \Sop_\lambda(a)=a\lambda - \lambda a.
\]
The stabilizer of $\lambda$ in $G$ is the unit group of the subalgebra
\[
\alg_\lambda = \{a \in \alg :a\lambda = \lambda a\}.
\]
Being open in $\alg_\lambda$, it is therefore smooth with Lie algebra $\alg_\lambda = \ker \Sop_\lambda$.

In particular, the orbit of $\lambda$ is open in $\bmdl$ if and only if $\Sop_\lambda$ is surjective.
Therefore, $\bmdl$ is \pvs if and only if $\Sop_\lambda$ is surjective for some $\lambda\in\bmdl$.

More generally, for $\lambda,\mu \in \bmdl$ let
\[
\Sop_{\lambda,\mu} : \alg \to \bmdl
\]
be the linear map $\Sop_{\lambda,\mu}(a) = a\lambda - \mu a$, so that $\Sop_\lambda=\Sop_{\lambda,\lambda}$. Set
\begin{subequations}\label{def:pair}
\begin{align}
   \dfr_\rsub(\alg,\bmdl) &= \min_{\lambda\in \bmdl}\cork \Sop_\lambda,
   \label{def:pairdg}\\
   \dfr_\psub(\alg,\bmdl) &= \min_{\lambda,\mu\in \bmdl}\cork \Sop_{\lambda,\mu}.
   \label{def:pairpw}
\end{align}
\end{subequations}
We call $\dfr_\rsub$ the \emph{\dgd} and $\dfr_\psub$ the \emph{\pwd} of
$(\alg,\bmdl)$: the first minimum is taken over the diagonal of
$\bmdl\times\bmdl$ and the second over all of it.
Both minima are attained on a dense open subset of the parameters, and both are the corank at a generic
parameter, that is, the corank of $\Sop_\lambda$, resp.\ $\Sop_{\lambda,\mu}$,
over the field of rational functions in the coordinates of the parameters,
taken as indeterminates. %See \Cref{ssec:arbhered} for a general $\kk$.

By the above, $\bmdl$ is \pvs if and only if $\dfr_\rsub(\alg,\bmdl) = 0$.
We say that $\bmdl$ is \emph{\wpvs} if $\dfr_\psub(\alg,\bmdl) = 0$.
Clearly $\dfr_\rsub \ge \dfr_\psub$, so \pvs implies \wpvs. When considering
both invariants we sometimes write $\dfr_\ast$, where $\ast$ is either
$\rsub$ or $\psub$. We allow the zero algebra ($1 = 0$), whose only bimodule
is $\bmdl = 0$, with $\dfr_\rsub = \dfr_\psub = 0$; it is both \pvs and \wpvs.

The terminology \wpvs does not suggest an open orbit condition for any group action; it is used for want of
a better one.

Our basic example will be the following.
Suppose that $M$ is a finite-dimensional module over a finitely generated algebra $R$. Take
\[
   \alg = \End_R(M), \qquad \bmdl = \Ext^1_R(M,M)^* ,
\]
with the bimodule structure obtained by functoriality. Explicitly,
\begin{equation}\label{eq:functorial}
(a\lambda b)(\xi) = \lambda(\Ext^1_R(a,b)(\xi))\qquad a,b \in \alg,\ \xi \in \Ext^1_R(M,M),\ \lambda \in \bmdl.\footnotemark
\end{equation}
\footnotetext{Here $\Ext^1_R(a,b)$ is the pullback along $a$ followed by the pushout along
$b$: on an extension $\xi : 0 \to M \to E \to M \to 0$, the first argument
acts at the quotient end and the second at the sub end, and the two
operations commute. The variance is reversed by passing to the dual, which is
why the contravariant argument $a$ acts on the left of $\bmdl$ and the
covariant argument $b$ on the right.}

For simplicity, we write in this case
\begin{equation} \label{def:df*M}
\dfr_*^R(M)=\dfr_*(\alg,\bmdl),\ \ \ *=\rsub,\psub
\end{equation}
and omit the superscript $R$ if it is clear from the context.
By abuse of language, we will say that $M$ is \pvs/\wpvs if $\bmdl$ is.

\subsection{Generators, and splitting off an idempotent}\label{ssec:gen}

We keep a pair $(\alg,\bmdl)$ and $G = \alg^\times$.
We call $\lambda \in \bmdl$ a \emph{generator} if $\alg\lambda\alg =\bmdl$.
We write $\bmdl^{\gen}$ for the set of generators.
Since $\im \Sop_\lambda \subseteq \alg\lambda\alg$, being a generator is a
necessary condition for the orbit to be open:
\begin{equation}\label{eq:opengen}
\text{if $\Sop_\lambda$ is surjective, i.e.\ if the orbit of $\lambda$ is
open, then $\lambda \in \bmdl^{\gen}$.}
\end{equation}

%\begin{lemma}\label{lem:opengen}
%If the $G$-orbit of $\lambda$ is open in $\bmdl$, equivalently if
%$\Sop_\lambda$ is surjective, then $\lambda$ is a generator.
%\end{lemma}

Let $\rad(\alg)$ be the radical of $\alg$ and let $\alg_{\sms}$ be the semisimple quotient $\alg/\rad(\alg)$.
Denote by $\rbm(\bmdl)$ the $\alg$-subbimodule
\begin{equation}\label{eq:radbim}
   \rbm(\bmdl) \;=\; \rad(\alg)\,\bmdl \;+\; \bmdl\,\rad(\alg)
\end{equation}
which is the radical of $\bmdl$ as a module over $\alg\otimes\alg^{\op}$.\footnote{$\alg_{\sms}\otimes\alg_{\sms}^{\op}$
is semisimple since $\kk$ is algebraically closed.}
We denote the quotient by
\[
\topp(\bmdl)=\bmdl/\rbm(\bmdl).
\]
It is a bimodule over $\alg_{\sms}=\alg/\rad(\alg)$. By Nakayama's lemma
\begin{equation}\label{eq:nakayama}
\text{$\lambda$ is a generator if and only if its class generates $\topp(\bmdl)$.}
\end{equation}

The following situation will be considered in \Cref{prop:covparam}.
Let $e \in \alg$ be a nonzero idempotent and suppose that
\begin{subequations}\label{eq:cornerax}
\begin{align}
   \alg e &= \kk e , \label{ax:le}\\
   \bmdl e &= 0 . \label{ax:re}
\end{align}
\end{subequations}
Put $e' = 1-e$ and
\begin{equation}\label{eq:cornerpieces}
   \alg' = e'\alg e', \qquad \mathcal N = e\alg e', \qquad \bmdl' = e'\bmdl .
\end{equation}
Let $\pi : \alg\rightarrow\alg'$, $x \mapsto e'xe'$, $p : \bmdl
\rightarrow \bmdl'$, $\lambda\mapsto e'\lambda$ and $G' = (\alg')^\times$.

\begin{lemma}[splitting off $e$]\label{lem:cornersplit}
Assume \eqref{eq:cornerax}. Then:
\begin{enumerate}
\item\label{cs:alg} In the Peirce decomposition of $\alg$ at $e$ one has
$e\alg e = \kk e$ and $e'\alg e = 0$ by \eqref{ax:le}, so that $\alg$ is
the triangular matrix algebra
\begin{equation}\label{eq:triangular}
   \alg \;=\;
   \begin{pmatrix} \kk e & \mathcal N\\ 0 & \alg'\end{pmatrix} ,
   \qquad
   a \;\longmapsto\;
   \begin{pmatrix} eae & eae'\\ 0 & e'ae'\end{pmatrix} ,
\end{equation}
where $\mathcal N$ is a $\kk$-$\alg'$-bimodule and the product is the matrix one.
In particular $\alg = e\alg\oplus\alg'$, where $\alg'$ is a
subalgebra with unit $e'$ and $e\alg = \alg e\alg = \kk e\oplus \mathcal N$ is a
two-sided ideal of $\alg$, and $\pi$ is a surjective algebra homomorphism with
kernel $e\alg$. Moreover,
\begin{equation}\label{eq:radsplit}
   \rad(\alg) \;=\; \mathcal N \oplus \rad(\alg'), \qquad
   \alg_{\sms}\simeq \kk\times \alg'_{\sms} ,
\end{equation}
so that $e$ is central in $\alg_{\sms}$, $\rad(\alg)e = 0$, $e\rad(\alg) = \mathcal N$ and $e'\rad(\alg) =
\rad(\alg')$.
\item\label{cs:group} The map $\pi$ induces a split epimorphism $\tilde\pi :
G \rightarrow G'$ with section $g'\in G'\mapsto e+g'
\in G$, and kernel
\begin{equation}\label{eq:Hdef}
   H = \{ce+x+e' : c \in \kk^\times,\ x \in \mathcal N\} 
\end{equation}
so that $G=H\rtimes G'$.
\item\label{cs:bim} The subspace $e\bmdl$ is an $\alg$-subbimodule, $\bmdl'$ is an
$\alg'$-bimodule, and $p$ is a surjection of bimodules along $\pi$ with
kernel $e\bmdl$. Moreover $e'\rbm(\bmdl) = \rbm(\bmdl')$ and, as a
$\kk\times\alg'_{\sms}$-bimodule,
\begin{equation}\label{eq:topsplit}
   \topp(\bmdl) \;=\; e\bmdl/e\rbm(\bmdl) \;\oplus\;\topp(\bmdl') ,
\end{equation}
the two summands being the images of left multiplication by $e$ and by $e'$.
\item\label{cs:gen} For every $\lambda\in\bmdl$,
\begin{equation}\label{eq:gensplit}
   \lambda \in \bmdl^{\gen}\iff
   p(\lambda) \in (\bmdl')^{\gen} \ \text{ and } \ e\lambda \text{
   generates the bimodule } e\bmdl/e\rbm(\bmdl).
\end{equation}
\item\label{cs:H} We have $\lambda h = \lambda$ for all $\lambda\in\bmdl$ and $h
\in H$; hence $H$ acts linearly on $\bmdl$ and preserves the fibers of $p$:
\begin{equation}\label{eq:Zorbit}
   H\cdot\lambda \;=\; \{\zeta\lambda + e'\lambda :
      \zeta \in e\alg\setminus \mathcal N\} .
\end{equation}
\item\label{cs:transitive} Suppose in addition that $\dim
e\bmdl/e\rbm(\bmdl) = 1$, so that the second condition in
\eqref{eq:gensplit} reads $e\lambda\notin e\rbm(\bmdl)$, and let $\lambda
\in \bmdl^{\gen}$. Then $H$ acts transitively on $p^{-1}(p(\lambda)) \cap
\bmdl^{\gen}$ if and only if $\mathcal N\lambda = e\rbm(\bmdl)$.
\end{enumerate}
\end{lemma}

\begin{proof}
\labelcref{cs:alg}, \labelcref{cs:group}, \labelcref{cs:bim} and \labelcref{cs:H} are straightforward.

For \labelcref{cs:gen}, the two summands of
\eqref{eq:topsplit} are subbimodules, and for every $m\in\topp(\bmdl)$,
\[
\alg_{\sms}m\alg_{\sms}=\alg_{\sms}em\alg_{\sms}\oplus \alg_{\sms}e'm\alg_{\sms}.
\]
A class therefore generates $\topp(\bmdl)$ if and only if each of its two
components generates the corresponding summand. It remains to apply
\eqref{eq:nakayama} to $(\alg,\bmdl)$ and to $(\alg',\bmdl')$, the action
of $\alg$ on $\topp(\bmdl')$ factoring through $\pi$, since $\mathcal N \subseteq
\rad(\alg)$ annihilates $\topp(\bmdl)$, so that generation
over $\alg$ and over $\alg'$ mean the same thing there.

\labelcref{cs:transitive} Note first that $\mathcal N\lambda \subseteq
e\rbm(\bmdl)$, since $\mathcal N = e\mathcal N$ and $\mathcal N \subseteq \rad(\alg)$. By
\eqref{eq:gensplit} the fiber in question is $\{\nu+e'\lambda : \nu \in
e\bmdl\setminus e\rbm(\bmdl)\}$, while by \eqref{eq:Zorbit}
\[
   H\cdot\lambda = \{c\,e\lambda + w + e'\lambda :
      c \in \kk^\times,\ w \in \mathcal N\lambda\} .
\]
As $e\rbm(\bmdl)$ is a hyperplane of $e\bmdl$ not containing $e\lambda$,
we have
\[
e\bmdl\setminus e\rbm(\bmdl) = \{c\,e\lambda + w : c \in\kk^\times,\ w \in e\rbm(\bmdl)\}.
\]
Comparing the two sets gives the assertion.
\end{proof}

\subsection{Deforming the conjugation action by a derivation}\label{ssec:cov}

Let $f : \alg \to \bmdl$ be a \emph{derivation}:
\begin{equation}\label{eq:leibniz}
   f(ab) = a f(b) + f(a) b, \qquad a,b \in \alg .
\end{equation}
In particular, $f(1) = 0$.
In cohomological terms, \eqref{eq:leibniz} says that $f$ is a
Hochschild $1$-cocycle of $\alg$ with values in $\bmdl$.
Equivalently, $f$ gives rise to an extension of $\alg$-bimodules of
$\alg$ by $\bmdl$: let $E = \alg\oplus \bmdl$ as a right $\alg$-module, and define the left $\alg$-action by
\[
   b\,(a,v) = (ba,\; bv + f(b)a), \qquad a,b \in \alg,\ v \in \bmdl .
\]
Thus $E$ is an $\alg$-bimodule containing $\bmdl = \{0\}\oplus \bmdl$ as a
subbimodule with quotient $\alg$, and the class of the extension
\[
   0 \longrightarrow \bmdl \longrightarrow E \longrightarrow \alg
   \longrightarrow 0
\]
in $\Ext^1_{\alg\otimes\alg^{\op}}(\alg,\bmdl) = \HH^1(\alg,\bmdl)$
is the class of the cocycle $f$; the extension splits if and only if $f$ is inner.

The fiber of $E \to \alg$ over $1$ is an affine subspace $\{(1,v) : v \in \bmdl\}$ of $E$ which is stable
under conjugation by $G$. Thus we get an action of $G$ on $\bmdl$ by affine transformations
\begin{equation}\label{eq:tau}
   g\,(1,v)\,g^{-1} = \bigl(1,\tau^f_g(v)\bigr), \qquad
   \tau^f_g(v) = gvg^{-1} + f(g)g^{-1} ,\ \ \ g\in G, v\in \bmdl,
\end{equation}
that is, the affine twist of the linear conjugation action on $\bmdl$ by the cocycle $g\mapsto f(g)g^{-1}$.
Similarly, for $v,w \in \bmdl$ the linear map $a \mapsto a\,(1,v) - (1,w)\,a$, $a \in
\alg$, takes values in $\bmdl$:
\begin{equation}\label{eq:D}
   a\,(1,v) - (1,w)\,a = \bigl(0, \Dop^f_{v,w}(a)\bigr), \qquad
   \Dop^f_{v,w}(a) = av - wa + f(a) ,
\end{equation}
and, as with $\Sop_\lambda$ above, we write $\Dop^f_v = \Dop^f_{v,v}$.
The identity
\[
   \bigl(hag^{-1}\bigr)\bigl(g(1,v)g^{-1}\bigr) -
   \bigl(h(1,w)h^{-1}\bigr)\bigl(hag^{-1}\bigr)
   = h\bigl(a\,(1,v) - (1,w)\,a\bigr)g^{-1}
\]
translates to the two-sided covariance
\begin{equation}\label{eq:covar}
   \Dop^f_{\tau^f_g(v),\,\tau^f_h(w)}(hag^{-1}) = h\,\Dop^f_{v,w}(a)\,g^{-1},
   \qquad a \in \alg,\ g,h \in G .
\end{equation}
In particular the maps $\Dop^f_{\tau^f_g(v),\tau^f_h(w)}$ and $\Dop^f_{v,w}$ have the same rank.

\subsection{The idempotent decomposition}\label{ssec:split}

From now on we fix, as part of the setup, an (automatically orthogonal) idempotent decomposition
\[
   1 = e_I + e_J
\]
in $\alg$. For $P, Q \in \{I,J\}$ put $\alg_{PQ} =
e_P\,\alg\,e_{Q}$ and $\bmdl_{PQ} = e_P\,\bmdl\,e_{Q}$, so that we have the
\emph{Peirce decompositions}
\begin{equation}\label{eq:decomp}
   \alg = \alg_I \oplus \alg_J \oplus \alg_{\mix},
   \qquad
   \bmdl = \bmdl_I \oplus \bmdl_J \oplus \bmdl_{\mix}
\end{equation}
collecting the diagonal corners $\alg_P = \alg_{PP}$, $\bmdl_P = \bmdl_{PP}$
and the mixed parts $\alg_{\mix} = \alg_{IJ}\oplus\alg_{JI}$,
$\bmdl_{\mix} = \bmdl_{IJ}\oplus \bmdl_{JI}$; put also
$\alg_{\diag} = \alg_I\oplus\alg_J$ and $\bmdl_{\diag} =
\bmdl_I \oplus \bmdl_J$. In the \emph{corners} $(\alg_P, \bmdl_P)$, $P \in
\{I,J\}$, $\alg_P$ is an algebra with unit $e_P$ and $\bmdl_P$ is an
$\alg_P$-bimodule, so the invariants \eqref{def:pair} are defined for them.

In the example of \Cref{ssec:pairs}, a direct sum decomposition $M = M_I
\oplus M_J$ of $R$-modules induces an idempotent decomposition
$1 = e_I + e_J$ in $\alg = \End_R(M)$, with $e_P$ the projection onto
$M_P$. The resulting Peirce decompositions are identified by additivity:

\begin{lemma}[corners in the module case]\label{lem:corners}
In the situation above, composition with the projections and inclusions
identifies
\[
   \alg_P \simeq \End_R(M_P), \qquad
   \alg_{IJ} \simeq \Hom_R(M_J,M_I), \qquad
   \alg_{JI} \simeq \Hom_R(M_I,M_J),
\]
as algebras, respectively as $(\alg_I,\alg_J)$- and
$(\alg_J,\alg_I)$-bimodules. Dually, $\bmdl_{PQ} \simeq \Ext^1_R(M_P,M_Q)^*$
for all $P, Q \in \{I,J\}$: in particular $\bmdl_P \simeq
\Ext^1_R(M_P,M_P)^*$ as $\alg_P$-bimodules --- so that
$\dfr_\ast(\alg_P,\bmdl_P)$ and the notions of \Cref{ssec:pairs} for $\bmdl_P$ are
those attached to $M_P$ --- and $\bmdl_{IJ} \simeq \Ext^1_R(M_I,M_J)^*$,
$\bmdl_{JI} \simeq \Ext^1_R(M_J,M_I)^*$. In particular
\begin{gather*}
   \dim\bmdl_{\mix} = \dim\Ext^1_R(M_I,M_J) +
   \dim\Ext^1_R(M_J,M_I),\\
   \dim\alg_{\mix} = \dim\Hom_R(M_I,M_J) + \dim\Hom_R(M_J,M_I) .
\end{gather*}
\end{lemma}

We freely use the resulting \emph{corner calculus}, an immediate
consequence of $e_Pe_Q = \delta_{PQ}e_P$: the products
$\alg_{PQ}\,\alg_{P'Q'}$, $\alg_{PQ}\,\bmdl_{P'Q'}$ and $\bmdl_{PQ}\,\alg_{P'Q'}$
vanish unless $Q = P'$, in which case they lie in the $(P,Q')$-corner. In
particular $\alg_{\mix}$, $\bmdl_{\diag}$ and
$\bmdl_{\mix}$ are subbimodules over $\alg_{\diag}$, though
not over $\alg$.

\subsection{Derivations linear over the diagonal, and the main results}\label{ssec:datum}

We keep the decomposition $1=e_I+e_J$ of \Cref{ssec:split}. Consider an $\alg_{\diag}$-bilinear derivation $f : \alg \to \bmdl$.
Equivalently, $f$ vanishes on $\alg_{\diag}$.
These are the cocycles representing the relative Hochschild cohomology
$\HH^1(\alg,\alg_{\diag};\bmdl)$, which classifies the extensions
of $\alg$-bimodules of $\alg$ by $\bmdl$ that split over $\alg_{\diag}$.
Any such derivation maps $\alg_{IJ}$ to $\bmdl_{IJ}$ and $\alg_{JI}$ to $\bmdl_{JI}$. In particular,
\begin{equation}\label{eq:reldiag}
   f(\alg)=f(\alg_{\mix}) \;\subseteq\; \bmdl_{\mix} .
\end{equation}

\begin{subequations}
We can factorize $f$ as
\begin{equation}\label{eq:factor}
   f = \iota\circ\ad(e_I), \qquad \ad(e_I)a = e_Ia - ae_I ,
\end{equation}
where $\iota : \alg_{\mix} \to \bmdl_{\mix}$ agrees with
$f$ on $\alg_{IJ}$ and with $-f$ on $\alg_{JI}$; it is a homomorphism of
$\alg_{\diag}$-bimodules and satisfies
\begin{equation} \label{eq:iotaporp}
   \iota(a)b = a\,\iota(b), \qquad b\,\iota(a) = \iota(b)\,a,
   \qquad a \in \alg_{IJ},\ b \in \alg_{JI} .
\end{equation}
Conversely, any $\iota$ with these two properties gives rise in this way
to an $\alg_{\diag}$-bilinear derivation. Both directions are a
direct check on the Peirce components of \eqref{eq:leibniz}.
\end{subequations}

The crucial special feature that we will use is the existence of an $\alg_{\diag}$-bilinear derivation whose image
is \emph{all} of $\bmdl_{\mix}$, which is the largest possible by \eqref{eq:reldiag}.
More precisely,

\begin{proposition}[two-block corank inequality]\label{prop:twoblock}
Let $\Ff$ be an $\alg$-subbimodule of $\bmdl$ contained in $\bmdl_{\mix}$ and let 
$\Vv = \bmdl/\Ff$ be the quotient bimodule.
Suppose that there exists an $\alg_{\diag}$-bilinear derivation $f : \alg \to \Vv$ with $f(\alg) =\Vv_{\mix}$.
Then, for $\ast\in \{\rsub,\psub\}$,
\begin{equation}\label{eq:twoblock}
   \dfr_\ast(\alg,\bmdl) \;\ge\; \dfr_\ast(\alg_I,\bmdl_I) + \dfr_\ast(\alg_J,\bmdl_J) .
\end{equation}
In particular, if $\bmdl$ is \pvs (resp.\ \wpvs), then so are the corner
bimodules $\bmdl_I$ over $\alg_I$ and $\bmdl_J$ over $\alg_J$.
\end{proposition}

We will prove \Cref{prop:twoblock} in \Cref{sec:proofoftwoblock}.
We first remark that for the proof, one can assume without loss of generality that $\Ff=0$.
Indeed, let $\pquo : \bmdl \to \Vv$ be the quotient map.
The Peirce decomposition of $\Vv$ is
\[
   \Vv = \bmdl_I \oplus \bmdl_J \oplus \Vv_{\mix}, \qquad \Vv_{\mix} = \bmdl_{\mix}/\Ff,
\]
so that the bimodules $\bmdl$ and $\Vv$ have the same diagonal corners.
Since $\Ff$ is a subbimodule, $\pquo\circ \Sop_{\lambda,\mu} = \Sop_{\pquo(\lambda),\pquo(\mu)}$
for $\lambda,\mu \in \bmdl$, and $\coker(\pquo\circ \Sop_{\lambda,\mu})$ is a quotient of $\coker
\Sop_{\lambda,\mu}$; taking minima (with $\mu =
\lambda$ in the case $\ast = \rsub$) gives $\dfr_\ast(\alg,\Vv) \le
\dfr_\ast(\alg,\bmdl)$. Thus, \Cref{prop:twoblock}, applied to
$\Vv$, yields \eqref{eq:twoblock} for $\bmdl$, and with it the inheritance of \pvsness and \wpvsness.

We specialize to the case $\alg =\End_R(M)$ and $\bmdl = \Ext^1_R(M,M)^*$ where $M$ is an $R$-module.

\begin{definition}[\obt decompositions]\label{def:obtuse}
Let $M = M_I \oplus M_J$ be a direct sum decomposition of $R$-modules,
with the induced idempotent decomposition and Peirce corners of
\Cref{ssec:split} (\Cref{lem:corners}). The decomposition is \emph{\obt}
if there exist an $\alg$-subbimodule $\Ff$ of $\bmdl$ contained in
$\bmdl_{\mix}$ and an $\alg_{\diag}$-bilinear derivation
$f : \alg \to \Vv = \bmdl/\Ff$ with $f(\alg) = \Vv_{\mix}$, as in
\Cref{prop:twoblock}, such that moreover $f$ restricts to an
\emph{isomorphism} of $\alg_{\mix}$ onto $\Vv_{\mix}$.
\end{definition}
In this case
\begin{equation}\label{eq:dimF}
   \dim \Ff \;=\; \dim \bmdl_{\mix} - \dim\alg_{\mix}.
\end{equation}
The terminology will be motivated in \Cref{rem:numerical} below.

\begin{corollary}[heredity for \obt decompositions]\label{cor:schema}
Let $M = M_I \oplus M_J$ be an \obt decomposition.
Then,
\[
  \dfr_\ast(M) \;\ge\; \dfr_\ast(M_I) + \dfr_\ast(M_J) .
\]
In particular, if $M$ is \pvs (resp., \wpvs) then so are $M_I$ and $M_J$.
\end{corollary}

\begin{remark}
Only surjectivity of $f$ is needed for \Cref{prop:twoblock}.
We require injectivity in \Cref{def:obtuse} because this will be the situation in our main application.
However, we certainly have cases where a surjective $f$ exists but not an injective one,
for instance when $\Ext^1_R(M_I,M_J)=\Ext^1_R(M_J,M_I)=0$ but $\Hom_R(M_I,M_J)\ne0$ (where we simply take $f=0$).
\end{remark}

\subsection{Proof of \texorpdfstring{\Cref{prop:twoblock}}{Proposition}}\label{sec:proofoftwoblock}
As remarked after the statement of \Cref{prop:twoblock}, we may and shall assume that $\Ff=0$, 
so that $f : \alg \to \bmdl$ is an $\alg_{\diag}$-bilinear derivation  with $f(\alg) =\bmdl_{\mix}$.

For $t \in \kk$ the scalar multiple $tf$ is again a
derivation, and we apply the constructions of \Cref{ssec:cov} to $\bmdl$ and the derivation $tf$, abbreviating
\[
   \tau^t = \tau^{tf}, \qquad \Dop^t = \Dop^{tf} .
\]
Set
\[
   R_\rsub(t) = \min_{v \in \bmdl}\cork \Dop^t_v,
   \qquad
   R_\psub(t) = \min_{v,w\in \bmdl}\cork \Dop^t_{v,w} .
\]
Note that by \eqref{eq:D}, $\Dop^0_{v,w} = \Sop_{v,w}$, so
\begin{equation}\label{eq:step1}
   R_\ast(0) = \dfr_\ast(\alg,\bmdl), \qquad \ast \in \{\rsub,\psub\} .
\end{equation}

Recall from \eqref{eq:decomp} that $\bmdl_{\diag}$ is a subspace
of $\bmdl$ complementary to $\bmdl_{\mix}$. For every $t$ let
\begin{equation}\label{eq:Phi}
   \Phi_t : G\times\bmdl_{\diag} \longrightarrow \bmdl,
   \qquad
   \Phi_t(g,z) = \tau^t_g(z) = gzg^{-1} + tf(g)g^{-1},
\end{equation}
be the restriction to $\bmdl_{\diag}$ of the affine action
\eqref{eq:tau} of $G$ on $\bmdl$.

\begin{lemma}[dominant action map]\label{lem:dom}
For every $t \ne 0$ the action map $\Phi_t$ of \eqref{eq:Phi} is
dominant, and so is the product map
$\Phi_t\times\Phi_t$ with values in $\bmdl\times\bmdl$. Hence, for $t\ne0$
\begin{equation} \label{eq:Rindiag}
R_\rsub(t) =  \min_{\xi \in \bmdl_{\diag}}\cork \Dop^t_\xi,\ \ \ R_\psub(t) = \min_{\xi,\eta \in \bmdl_{\diag}}\cork \Dop^t_{\xi,\eta}.
\end{equation}
\end{lemma}

\begin{proof}
Indeed, identifying the tangent space of $G$ at $1$ with $\alg$,
the differential of $\Phi_t$ at the point $(1,0)$ is given by
\begin{equation}\label{eq:dPhi}
   d\Phi_t|_{(1,0)}(a,z) = z + tf(a), \qquad a \in \alg,\quad z \in
   \bmdl_{\diag} .
\end{equation}
This is surjective: $f(\alg) = \bmdl_{\mix}$ by the standing
hypothesis, and $\bmdl =\bmdl_{\mix}\oplus\bmdl_{\diag}$ by \eqref{eq:decomp}. As
$G\times\bmdl_{\diag}$ and $\bmdl$ are irreducible and smooth, $\Phi_t$ is therefore
dominant (and separable); and a product of two dominant morphisms is dominant.

Since $\argmin_{\bmdl}\cork \Dop^t_v$ and $\argmin_{\bmdl\times\bmdl}\cork \Dop^t_{v,w}$ are open
by upper semicontinuity, they intersect the images of $\Phi_t$ and $\Phi_t\times\Phi_t$ respectively, and by the covariance \eqref{eq:covar}
\[
\cork \Dop^t_{\xi,\eta}=\cork \Dop^t_{\Phi_t(g,\xi),\Phi_t(h,\eta)},\ \ \forall g,h\in G,\ \xi,\eta\in \bmdl_{\diag}.
\]
Hence, the minimum is already attained on $\bmdl_{\diag}$, which is
\eqref{eq:Rindiag}.
\end{proof}

To finish the proof of \Cref{prop:twoblock} it remains to observe the following.

\begin{lemma}[block projection]\label{lem:blockproj}
Let $\pr : \bmdl \to \bmdl_{\diag}$ be the projection along
$\bmdl_{\mix}$. Let $\xi, \eta \in \bmdl_{\diag}$, $t \in \kk$ and
$a \in \alg$. Write $\xi = \xi_I + \xi_J$, $\eta = \eta_I + \eta_J$ and
$a = a_I + a_J + a_{\mix}$ according to \eqref{eq:decomp}. Then,
\begin{equation}\label{eq:blockcomp}
   \pr\,\Dop^t_{\xi,\eta}(a) = (a_I\xi_I - \eta_Ia_I)\oplus
   (a_J\xi_J - \eta_Ja_J) .
\end{equation}
Therefore,
\[
\coker(\pr\circ\,\Dop^t_{\xi,\eta}) = \coker
\Sop^{\bmdl_I}_{\xi_I,\eta_I}\oplus\coker \Sop^{\bmdl_J}_{\xi_J,\eta_J}
\]
where the superscripts on the right-hand side indicate that we take the maps pertaining to $\bmdl_I$ and $\bmdl_J$
(as $\alg_I$- and $\alg_J$-bimodules) and
\begin{equation}\label{eq:prbound}
   \cork \Dop^t_{\xi,\eta}\ge\cork\bigl(\pr\circ\,\Dop^t_{\xi,\eta}\bigr) =
   \cork \Sop^{\bmdl_I}_{\xi_I,\eta_I} + \cork \Sop^{\bmdl_J}_{\xi_J,\eta_J}.
\end{equation}
\end{lemma}

\begin{proof}
By the corner calculus, the diagonal-by-diagonal products contribute
$a_P\xi_P - \eta_Pa_P \in \bmdl_P$ (products between the two different diagonal
corners vanish), while the diagonal-by-mixed products
$a_{\mix}\xi_P$ and $\eta_Pa_{\mix}$ lie in
$\bmdl_{\mix}$. Finally $f(a) = f(a_{\mix}) \in
\bmdl_{\mix}$ by \eqref{eq:reldiag}. Applying $\pr$ kills the
mixed contributions and leaves \eqref{eq:blockcomp}.

Thus, $\pr\circ \Dop^t_{\xi,\eta}$ is the direct sum of the maps
$\Sop^{\bmdl_I}_{\xi_I,\eta_I}$ and $\Sop^{\bmdl_J}_{\xi_J,\eta_J}$ of the two corners, precomposed with
the projection $a \mapsto (a_I,a_J)$. The second part follows since precomposition with a surjection
does not change the cokernel, while a postcomposition with a surjection cannot increase the corank.
\end{proof}

For either $\ast \in\{\rsub,\psub\}$,
the inequality \eqref{eq:prbound} of \Cref{lem:blockproj}, together with \eqref{eq:Rindiag} and the definition \eqref{def:pair}, implies that for $t\ne0$
\begin{equation}\label{eq:step2}
   R_\ast(t) \ge \dfr_\ast(\alg_I,\bmdl_I) + \dfr_\ast(\alg_J,\bmdl_J) .
\end{equation}
On the other hand, for each fixed pair $(v,w)$, $t \mapsto \cork \Dop^t_{v,w}$ is upper
semicontinuous; and the pointwise infimum of any family of upper
semicontinuous functions is upper semicontinuous. Hence so is $R_\ast$, and
therefore, except for finitely many $t$'s,
\[
   R_\ast(0) \ge R_\ast(t)
\]
for both $\ast$.
Combining this with \eqref{eq:step1} and \eqref{eq:step2} proves \eqref{eq:twoblock}.
The proof of \Cref{prop:twoblock} is complete.

\subsection{Arbitrary fields}\label{ssec:arbhered}

Let $\kk_0$ be an arbitrary field, not necessarily algebraically closed.
Given a pair $(\alg_0,\bmdl_0)$ consisting of a finite-dimensional $\kk_0$-algebra $\alg_0$ and a finite-dimensional bimodule $\bmdl_0$ 
we define $\dfr_\rsub(\alg_0,\bmdl_0)$ (resp., $\dfr_\psub(\alg_0,\bmdl_0)$) to be the corank of $\Sop_\lambda$, resp.\
$\Sop_{\lambda,\mu}$, over the field of rational functions in the coordinates
of the parameters, taken as indeterminates.
This is the same as $\dfr_\ast(\alg_0\otimes_{\kk_0}\kk,\bmdl_0\otimes_{\kk_0}\kk)$
where $\kk$ is an algebraic closure  of $\kk_0$,
and $\dfr_\ast(\alg_0,\bmdl_0) =\dfr_\ast(\alg_0\otimes_{\kk_0}\kk',\bmdl_0\otimes_{\kk_0}\kk')$ for every
field extension $\kk'/\kk_0$.

As before, we call $\bmdl_0$ \pvs, resp.\ \wpvs, if $\dfr_\rsub(\alg_0,\bmdl_0)=0$ (resp., $\dfr_\psub(\alg_0,\bmdl_0)=0$).
Thus $\bmdl_0$ is \pvs precisely when the algebraic group $G_0 = \alg_0^\times$ over $\kk_0$ has a dense orbit in
$\bmdl_0$, in the geometric sense.

\Cref{prop:twoblock} and \Cref{cor:schema} hold verbatim over an arbitrary field.

If $\kk_0$ is infinite, then
\[
   \min_{\lambda\in\bmdl_0}\cork \Sop_\lambda \;=\; \dfr_\rsub(\alg_0,\bmdl_0),
   \qquad
   \min_{\lambda,\mu\in\bmdl_0}\cork \Sop_{\lambda,\mu} \;=\; \dfr_\psub(\alg_0,\bmdl_0).
\]
If $\kk_0$ is finite, we can only guarantee the inequalities $\ge$.

\section{Nilpotent varieties, rigidity and \prigity}\label{sec:nilp}

In this section we recall well known facts about quivers, preprojective algebras and their representations
and define \prigity in this context. See \cite{MR4847234} and the references therein for more details.

\subsection{Quivers and nilpotent varieties}\label{ssec:quivers}

Let $Q = (\qvert,\qarr)$ be a finite quiver without loops, with vertex set
$\qvert$, and let $\kk Q$ be its path
algebra. An arrow $\alpha \in \qarr$ goes from its \emph{tail}
$t(\alpha) \in \qvert$ to its \emph{head} $h(\alpha) \in \qvert$, and $Q^\circ$
denotes the opposite quiver, in which the two are interchanged. For a
finite-dimensional $\qvert$-graded vector space $V = \oplus_{i\in \qvert}V_i$ of
graded dimension $\bd = \grdim V$, the representations of $Q$ on $V$ form
the affine space
\[
   R_Q(V) \;=\; \bigoplus_{\alpha\in\qarr}
   \Hom\bigl(V_{t(\alpha)},\,V_{h(\alpha)}\bigr) ,
\]
on which the group
$G_V = \prod_{i\in \qvert}\GL(V_i)$ of grading-preserving automorphisms acts by
conjugation, the orbits being the isomorphism classes of representations of
graded dimension $\bd$. Let $\langle\cdot,\cdot\rangle_Q$ be the
Euler form and $(\cdot,\cdot)$ its symmetrization, so that for representations
$M$, $N$ of graded dimensions $\bd$, $\be$,
\[
   \dim\Hom_Q(M,N) - \dim\Ext^1_Q(M,N) =
   \langle\bd,\be\rangle_Q .
\]

\begin{remark}[numerical constraint]\label{rem:numerical}
Suppose that $M$ is a representation of $Q$ with a direct sum decomposition $M=M_I\oplus M_J$.
By \Cref{lem:corners} with its notation we have
\[
   \dim\bmdl_{\mix} - \dim\alg_{\mix}
   \;=\; -\bigl(\grdim M_I,\, \grdim M_J\bigr).
\]
Thus, in the case of an \obt decomposition (\Cref{def:obtuse}) we have
\[
\dim\Ff=-(\grdim M_I, \grdim M_J)
\]
and in particular  $(\grdim M_I, \grdim M_J)\le 0$. This justifies the terminology.
\end{remark}

Let $\overline{Q}$ be the double quiver and $\Pi = \Pi_Q$ the preprojective
algebra, the quotient of the path algebra of $\overline{Q}$ by the
preprojective relation. The $\Pi$-module structures on $V$ form a closed
subvariety $R_\Pi(V)$ of $R_{\overline{Q}}(V) = R_Q(V)\times R_{Q^\circ}(V) =
T^*(R_Q(V))$, the fiber over $0$ of the moment map. Lusztig's nilpotent
variety is
\[
   \Lambda(V) \;=\; R_\Pi(V)\cap R^{\mathrm{nilp}}_{\overline{Q}}(V),
\]
the locus of nilpotent $\Pi$-modules. It is of pure dimension $\dim R_Q(V)$
by \cite{MR1088333}*{Theorem 12.3}, and in fact Lagrangian in
$T^*(R_Q(V))$ \cite{MR1088333}*{Theorem 12.9}; only the former will be used. We
write $\Comp(\bd)$ for its (finite) set of irreducible components,
each of which is $G_V$-stable, and $\Comp = \bigsqcup_{\bd}
\Comp(\bd)$.

By a result of Crawley-Boevey \cite{MR1781930}*{Lemma 1}, for any
$\Pi$-modules $M$, $N$ of graded dimensions $\bd$,
$\be$,
\begin{equation}\label{eq:CB}
   \dim\Hom_\Pi(M,N) - \dim\Ext^1_\Pi(M,N) + \dim\Hom_\Pi(N,M) =
   (\bd,\be).
\end{equation}
Moreover $\Ext^1_\Pi(N,M) \simeq \Ext^1_\Pi(M,N)^*$ functorially. For $\cmp_1,
\cmp_2 \in \Comp$ set
\[
   \hom_\Pi(\cmp_1,\cmp_2) = \min\{\dim\Hom_\Pi(x_1,x_2)\},
   \qquad
   \ext^1_\Pi(\cmp_1,\cmp_2) = \min\{\dim\Ext^1_\Pi(x_1,x_2)\},
\]
the minima over $(x_1,x_2) \in \cmp_1\times \cmp_2$; by upper semicontinuity they
are attained on dense open subsets.

Finally, we recall the duality. The preprojective algebra is isomorphic to
its opposite algebra: more precisely, $\Pi_Q^{\op} \simeq
\Pi_{Q^\circ} = \Pi_Q$, where $Q^\circ$ is the opposite quiver and the first
isomorphism takes a path to its opposite. Composing with the linear duality
$M \mapsto M^* = \Hom_\kk(M,\kk)$, we obtain a self-duality on
$\Pi$-modules, an isomorphism $\Lambda(V) \to
\Lambda(V^*)$, and hence a bijection $\Comp(\grdim V) \to \Comp(\grdim V^*)$, denoted $\cmp
\mapsto \cmp^*$; identifying $V^*$ with $V$ it is an involution on
$\Comp(\bd)$ for every $\bd$
\cite{MR4847234}*{\S2}.

\subsection{The Dynkin case}\label{ssec:dynkin}

From now on we assume the underlying graph of $Q$ is a simply laced Dynkin diagram. Then
the isomorphism classes of indecomposable representations of $Q$ are in bijection
$\beta\mapsto M_Q(\beta)$ with the set $\Psi$ of positive roots.
The inverse map is the graded dimension of the representation \cites{MR0332887,MR0393065}.
Writing $\Mfr = \NN\Psi$ for the monoid of
formal sums $\mm = \beta_1 + \dots + \beta_k$ of positive roots, the map $\mm
\mapsto M_Q(\mm) = M_Q(\beta_1)\oplus\dots\oplus M_Q(\beta_k)$ is a bijection
onto the isomorphism classes of representations of $Q$. In the Dynkin case
$\Lambda(V) = R_\Pi(V)$ is the union of the conormal bundles of the (finitely
many) $G_V$-orbits in $R_Q(V)$, and taking closures of conormal bundles gives
a bijection\footnote{Over $\CC$ this is a special case of Pjasecki{\u\i}'s duality
\cite{MR0390138}*{Corollary 2}.}
\[
   \lambda_Q : \Mfr \longrightarrow \Comp.
\]

We write $\cmp_\mm = \lambda_Q(\mm)$. Note that $\cmp_{\mm}\oplus \cmp_{\mm'}$, the
closure of $G_V\cdot\{x\oplus x'\}$, is always contained in $\cmp_{\mm+\mm'}$.

\subsection{Rigidity and \texorpdfstring{\Prigity}{Pseudorigidity}}\label{ssec:prig}

A $\Pi$-module $x$ is \emph{rigid} if $\Ext^1_\Pi(x,x) = 0$. For $x \in
\Lambda(V)$, of graded dimension $\bd = \grdim V$, rigidity is equivalent
to the orbit $G_V\cdot x$ being open in $\Lambda(V)$; in that case the
closure of the orbit is an irreducible component. Indeed, the stabilizer of
$x$ in $G_V$ is the unit group of $\End_\Pi(x)$, hence open in it, so that
$\dim G_V\cdot x = \dim G_V - \dim\End_\Pi(x)$; on the other hand
$(\bd,\bd) = 2\langle\bd,\bd\rangle_Q = 2\bigl(\dim G_V - \dim R_Q(V)\bigr)$,
so that \eqref{eq:CB} with $M = N = x$ gives
\begin{equation}\label{eq:extcodim}
   \dim\Ext^1_\Pi(x,x) \;=\; 2\dim\End_\Pi(x) - (\bd,\bd)
   \;=\; 2\bigl(\dim R_Q(V) - \dim G_V\cdot x\bigr) ,
\end{equation}
which is twice the codimension of $G_V\cdot x$ in $\Lambda(V)$, the latter
being of pure dimension $\dim R_Q(V)$. An
irreducible component $\cmp$ of $\Lambda(V)$ is called \emph{rigid} if it
contains a rigid module, or equivalently, a (necessarily unique) open $G_V$-orbit
\cite{MR4847234}*{\S4}. See \cite{MR2242628} for rigid modules and rigid
components over preprojective algebras.

Say that $\cmp_1$ and $\cmp_2$ \emph{strongly commute} if $\ext^1_\Pi(\cmp_1,\cmp_2) = 0$;
by the symmetry $\Ext^1_\Pi(x_1,x_2)^* \simeq \Ext^1_\Pi(x_2,x_1)$ the
condition is symmetric in $\cmp_1$, $\cmp_2$. By \cite{MR1944812}*{Theorem 1.2}, $\cmp_1$ and $\cmp_2$ strongly commute if and only
if $\cmp_1\oplus \cmp_2$ is an irreducible component, in which case $\cmp_1\oplus \cmp_2=\cmp_{\mm_1+\mm_2}$
where $\cmp_i=\cmp_{\mm_i}$, $i=1,2$.

\begin{definition}\label{def:prig}
$\cmp \in \Comp$ is \emph{\prig} if $\ext^1_\Pi(\cmp,\cmp) = 0$, that is, if $\cmp$
strongly commutes with itself.
\end{definition}

\begin{remark}\label{rem:prigbasic}
We record the immediate properties.
\begin{enumerate}
\item\label{prigbasic:rigid} Every rigid component is \prig, but the
converse is not true (see \Cref{ex:leclerc} below).
\item\label{prigbasic:dual} \Prigity is preserved by the duality $\cmp \mapsto \cmp^*$ defined at the
end of \Cref{ssec:quivers}, since $\Ext^1_\Pi(x^*,y^*) \simeq
\Ext^1_\Pi(y,x)$.
\item\label{prigbasic:genext} In terms of the ``generic extension'' binary operation of Aizenbud--Lapid, $\cmp_\mm$ is \prig if and only if $\cmp_\mm * \cmp_\mm =
\cmp_{\mm+\mm}$ \cite{MR4847234}*{(9.5)}. 
\item \label{rem:semican} In terms of the dual semicanonical basis, $\{\rho_{\cmp}\}_{\cmp\in \Comp}$
of the coordinate ring of a maximal unipotent subgroup \cite{MR2144987} over $\CC$,
\prigity implies that 
\[
   \rho_{\cmp}^{\,2} \;=\; \rho_{\cmp\oplus \cmp} ,
\]
so that the square of $\rho_{\cmp}$ is again a dual semicanonical basis element \cite{MR2144987}*{Theorem 1.1}.
This is the analogue, for that basis, of a dual canonical basis vector
being \emph{real}. We do not know whether the converse holds.
\end{enumerate}
\end{remark}

Let $M$ be a representation of $Q$, i.e.\ an $R$-module where $R = \kk Q$.
Let $\bd$ be the graded dimension of $M$.
We use the notation of \Cref{ssec:pairs}, with respect to
\[
   \alg \;=\; \End_Q(M),
   \qquad
   \bmdl \;=\; \Ext^1_Q(M,M)^* .
\]
In particular, the group $G = \Aut_Q(M)=\alg^\times$ acts on $\bmdl$ by conjugation, $g\cdot\lambda =g\lambda g^{-1}$.

It will be useful to reinterpret $\bmdl$ using the conormal incarnation of
\Cref{ssec:dynkin}. Write $V$ for the underlying graded
space of $M$ and $T \in R_Q(V)$ for its structure maps. The differential at
$1$ of the orbit map $g \mapsto gTg^{-1}$ is
\begin{equation}\label{eq:dT}
   d_T : \bigoplus_{i\in \qvert}\End(V_i)\longrightarrow R_Q(V), \qquad
   d_T(\phi)_\alpha = \phi_{h(\alpha)}T_\alpha - T_\alpha\phi_{t(\alpha)} ,
   \qquad \alpha\in\qarr ,
\end{equation}
with kernel $\alg$ and cokernel $\Ext^1_Q(M,M)$. Thus, $\bmdl$ is the kernel of the dual map to $d_T$.
We can identify the dual of $R_Q(V)$ with $R_{Q^\circ}(V)$ via the pairing
\[
\langle B,\xi\rangle =\sum_{\alpha\in\qarr}\tr(B_\alpha\xi_\alpha).
\]
Collecting the terms of $\langle B, d_T\phi\rangle$ that involve $\phi_i$, the annihilator of
$\im d_T$ is cut out by the preprojective relation at every vertex, so that
\begin{equation}\label{eq:conormal}
   \bmdl \;=\; \Ext^1_Q(M,M)^* \;=\;
   \bigl\{B \in R_{Q^\circ}(V) \;:\; (T,B) \in R_\Pi(V)\bigr\} ,
\end{equation}
the fiber at $T$ of the conormal bundle of the orbit $O_M=G_V\cdot T$; and the
bimodule structure \eqref{eq:functorial} becomes composition of linear maps,
$a\cdot B\cdot b = a\circ B\circ b$.

By \eqref{eq:conormal} the elements of $\bmdl$ are exactly the $\Pi$-structures
on $M$ compatible with the $Q$-structure, and a homomorphism of $\Pi$-modules
$(M,\lambda)\to(M,\mu)$ is a morphism $a : M \to M$ of representations of $Q$
that intertwines the two extra structures, that is $a\lambda = \mu a$. Hence,
for $\lambda,\mu\in\bmdl$,
\begin{equation}\label{eq:hompi}
   \Hom_\Pi\bigl((M,\lambda),(M,\mu)\bigr) \;=\; \ker \Sop_{\lambda,\mu}.
\end{equation}

The invariants defined in \eqref{def:df*M} have the following meaning.

Write $M=M_Q(\mm)$ where $\mm \in \Mfr$ and let $\cmp_\mm=\lambda_Q(\mm)$ be the irreducible component in $\Comp(\bd)$.
Recall that $\cmp_\mm$ is the closure of the conormal bundle of the $G_V$-orbit of $M$ in $T^*(R_Q(V))$.
For simplicity, we write
\[
\dfr_*(\mm)=\dfr_*^R(M).
\]
Then, $\dfr_\rsub(\mm)$ is the codimension of a generic $G_V$-orbit in
$\cmp_\mm$. Indeed, the stabilizer of $M$ in $G_V$ is $G=\Aut_Q(M)$, so the $G_V$-orbit of
a conormal point $(M,\lambda)$ has dimension $\dim O_M + \dim G\cdot\lambda =
\dim O_M + \rank \Sop_\lambda$, the second equality by the
separability recorded in \Cref{ssec:pairs}; on the other hand $\dim \cmp_\mm =
\dim R_Q(V) = \dim O_M + \dim\bmdl$, so that the codimension of that orbit is
$\cork \Sop_\lambda$. The conormal points over $O_M$ are dense in
$\cmp_\mm$, and taking $\lambda$ generic gives the assertion.

Since $(\bd,\bd) = 2(\dim\alg - \dim\bmdl)$, by \eqref{eq:CB} we have
\[
   \dim\Ext^1_\Pi((M,\lambda),(M,\mu))   = \cork \Sop_{\lambda,\mu} + \cork \Sop_{\mu,\lambda}.
\]
Since the conormal points over $O_M$ are dense in $\cmp_\mm$ and
$\dim\Ext^1_\Pi$ is upper semicontinuous, we get
\[
\min_{x\in \cmp_\mm}\dim\Ext^1_\Pi(x,x)=2\dfr_\rsub(\mm)
\]
and, since $\argmin_{\bmdl\times\bmdl}\cork \Sop_{\lambda,\mu}$ and $\argmin_{\bmdl\times\bmdl}\cork \Sop_{\mu,\lambda}$
intersect, being dense open subsets of the irreducible variety $\bmdl\times\bmdl$,
\begin{equation}\label{eq:extformula}
   \ext^1_\Pi(\cmp_\mm, \cmp_\mm) \;=\; 2\dfr_\psub(\mm).
\end{equation}

In particular, $\dfr_\rsub(\mm)=0$ (i.e., $M$ is \pvs in the language of \Cref{ssec:pairs}) if and only if $\cmp_\mm$ is rigid.
For simplicity we say that $\mm$ is rigid in this case.
Similarly, $\dfr_\psub(\mm)=0$ (i.e., $M$ is \wpvs in the language of \Cref{ssec:pairs}) if and only if $\cmp_\mm$ is \prig.
We say that $\mm$ is \prig in this case.

\subsection{Equioriented type \texorpdfstring{$A$}{A}: segments and the matrix-unit model}\label{ssec:typeA}

We specialize further to the equioriented quiver of type $A_r$,
\[
   1 \longrightarrow 2 \longrightarrow \cdots \longrightarrow r-1\longrightarrow r.
\]
We index the positive roots as segments supported in $[1,r]$ where
a \emph{segment} is an interval $\Delta = [a,b]$ with $a,b \in \ZZ$ and $a \le b$.
The representation attached to a segment is a \emph{graded Jordan block}:
$M([a,b])$ has a one-dimensional space at each vertex of $[a,b]$, zero
elsewhere, and nonzero along each arrow of $[a,b]$. It will be useful to set $M([a,a-1])=0$ for all $a$. We say that the segments $[a,b]$ and $[c,d]$ are \emph{juxtaposed} if $c=b+1$.

It is easy to see that $\Hom_Q(M([a,b]),M([c,d]))$ is $0$ unless $c \le a \le d \le b$ in which case it is one-dimensional.
Similarly $\Ext^1_Q(M([a,b]),M([c,d]))$ is $0$ unless $a+1\le c\le b+1\le d$ in which case it is one-dimensional.\footnote{This follows from
the fact that the Auslander--Reiten translate takes $M([a,b])$ to $M([a+1,b+1])$ (interpreted as $0$ if $b=r$) and $\Ext^1_Q(M([a,b]),M([c,d]))^* \simeq\Hom_Q(M([c,d]),\tau M([a,b]))$.}

A \emph{multisegment} is a formal sum $\mm = \sum_{i \in \Kk}\Delta_i$ over a finite index set $\Kk$, with $\Delta_i = [a_i,b_i]$.
The multisegments supported in $[1,r]$ make up the monoid $\Mfr$ of \Cref{ssec:dynkin}.

For $i,j \in \Kk$ define two relations,
\[
   i\,\XX\,j \iff a_i + 1 \le a_j \le b_i + 1 \le b_j ,
   \qquad
   i\,\YY\,j \iff a_i \le a_j \le b_i \le b_j ,
\]
and set $\XX = \XX_\mm = \{(i,j) \in \Kk^2 : i \XX j\}$, $\YY = \YY_\mm =
\{(i,j) \in \Kk^2 : i \YY j\}$. Note that $(i,i) \in \YY$ for every $i$,
while $(i,i) \notin \XX$.

Thus,
\begin{align*}
   \dim\Hom_Q\bigl(M(\Delta_j),M(\Delta_i)\bigr) &=
   \begin{cases} 1, & (i,j) \in \YY,\\ 0, & \text{otherwise,}\end{cases}\\
   \dim\Ext^1_Q\bigl(M(\Delta_i),M(\Delta_j)\bigr) &=
   \begin{cases} 1, & (i,j) \in \XX,\\ 0, & \text{otherwise.}\end{cases}
\end{align*}

In particular $\dim\alg = \#\YY$ and $\dim\bmdl = \#\XX$. We write down
bases for $\alg$ and $\bmdl$ as follows. Fix, for each $i \in \Kk$ and each $v \in
\Delta_i$, a basis vector $x_{i,v}$ of the $v$-component of the summand
$M(\Delta_i)$ of $M$, so that the structure map $A_v$ of $M$ carries
$x_{i,v}$ to $x_{i,v+1}$ if $v+1\in\Delta_i$ and to $0$ otherwise. For $(i,j)
\in \YY$ let $\algbss_{ij}$ be the degree-preserving endomorphism of $V$, and for
$(i,j) \in \XX$ let $\bmdlbss_{ij}$ be the degree-lowering one, determined by
\begin{equation}\label{eq:explicitbasis}
   \algbss_{ij}(x_{\ell,v}) = \delta_{j\ell}\,x_{i,v}
   \quad (v \in \Delta_i\cap\Delta_j),
   \qquad
   \bmdlbss_{ij}(x_{\ell,v}) = \delta_{j\ell}\,x_{i,v-1}
   \quad (v \in \Delta_j,\ v-1\in\Delta_i),
\end{equation}
both vanishing on the basis vectors not listed. In formulas below, the symbols
$\algbss_{ij}$ and $\bmdlbss_{ij}$ appearing on the right-hand side are understood to be zero if
$(i,j)\notin\YY$, resp.\ $(i,j)\notin\XX$.

The following is a straightforward computation.
(See \cite{MR863522}*{Lemmes II.4 and II.5 and Remarque II.5.1} for the first part,
as well as \cite{MR3866895}*{\S4.1}.)

\begin{lemma}[the matrix-unit model]\label{lem:alg}
The families $\{\algbss_{ij}\}_{(i,j)\in\YY}$ and
$\{\bmdlbss_{ij}\}_{(i,j)\in\XX}$ are bases of $\alg$ and of $\bmdl$ (the latter, in the conormal description \eqref{eq:conormal}).
The identity of $\alg$ is $1 =\sum_{i \in \Kk}\algbss_{ii}$ and
the multiplication table of $\alg$ and the bimodule structure of $\bmdl$ are given by
\begin{equation}\label{eq:units}
   \algbss_{ij}\algbss_{\ell m} = \delta_{j\ell}\algbss_{im},
   \qquad
   \algbss_{ij}\bmdlbss_{\ell m} = \delta_{j\ell}\bmdlbss_{im},
   \qquad
   \bmdlbss_{ij}\algbss_{\ell m} = \delta_{j\ell}\bmdlbss_{im},
\end{equation}
where the subscripts of the factors on the left-hand sides are in $\YY$, resp.\ $\XX$,
and the symbols on the right-hand sides are subject to the zero convention above.
Thus, $[\algbss_{ij},\bmdlbss_{\ell m}] = \delta_{j\ell}\bmdlbss_{im} -\delta_{im}\bmdlbss_{\ell j}$ for
$(i,j)\in\YY$ and $(\ell,m)\in\XX$.
\end{lemma}

It follows that the algebra $\alg$ and the bimodule $\bmdl$ attached to a
multisegment, hence also the invariants $\dfr_\rsub$ and $\dfr_\psub$,
depend only on the pair of relations $(\XX_\mm, \YY_\mm)$:

\begin{corollary}[relabeling invariance]\label{cor:relinv}
Let $\mm$, $\mm'$ be multisegments with index sets $\Kk$, $\Kk'$ and let
$\varphi : \Kk \to \Kk'$ be a bijection with $(\varphi\times\varphi)(\XX_\mm) =
\XX_{\mm'}$ and $(\varphi\times\varphi)(\YY_\mm) = \YY_{\mm'}$. Then
$\dfr_\ast(\mm) = \dfr_\ast(\mm')$ for both $\ast$; in particular $\mm$
is rigid (resp.\ \prig) if and only if $\mm'$ is. The same conclusion
holds if $\varphi$ reverses the pairs, that is, if
$(\varphi\times\varphi)(\XX_\mm)$ and $(\varphi\times\varphi)(\YY_\mm)$
are the relations opposite to $\XX_{\mm'}$ and $\YY_{\mm'}$.
\end{corollary}

\subsection{The rank conditions in coordinates}

We now write the maps $\Sop_{\lambda,\mu}$ of \Cref{ssec:pairs} in the
bases above. For $\lambda = \sum \lambda_{ij}\bmdlbss_{ij}$ and $\mu = \sum
\mu_{ij}\bmdlbss_{ij}$, \eqref{eq:units} gives
\begin{equation}\label{eq:S}
   \Sop_{\lambda,\mu}(\algbss_{ij})
   \;=\;
   \sum_{\substack{m\,:\,(j,m)\in \XX\\ (i,m)\in \XX}} \lambda_{jm}\,\bmdlbss_{im}
   \;-\;
   \sum_{\substack{l\,:\,(l,i)\in \XX\\ (l,j)\in \XX}} \mu_{li}\,\bmdlbss_{lj},
   \qquad (i,j) \in \YY ,
\end{equation}
the coefficient arrays being extended by zero outside $\XX$. Thus the
matrix of $\Sop_{\lambda,\mu}$ in the bases $\{\algbss_{ij}\}_{(i,j)\in\YY}$
of $\alg$ and $\{\bmdlbss_{ik}\}_{(i,k)\in\XX}$ of $\bmdl$ has its rows indexed by $\XX$
and its columns indexed by $\YY$, and $\mm$ is \prig if and only if
this matrix has full row rank for some $(\lambda,\mu)$, rigid if and only
if the same holds at $\mu = \lambda$.

Note that each nonzero entry of the matrix is either $\lambda_{jm}$ or
$-\mu_{li}$. We record which one, in the following bookkeeping device.

\begin{definition}[colors and the colored graph]\label{def:M}
To each $(a,b) \in \XX$ we attach two formal symbols, its \emph{colors}
$\blam_{ab}$ and $\bmu_{ab}$, the first \emph{of type $\blam$} and the
second \emph{of type $\bmu$}. The colored bipartite graph $\Gc_\mm$ on
$(\XX,\YY)$ has, for $(a,b) \in \XX$,
\begin{equation}\label{eq:type12}
\begin{aligned}
   \bmu_{ab} &\text{ colors } (a,c)\to(b,c)
   &&\text{for every } (a,c)\in\XX, (b,c) \in \YY,\\
   \blam_{ab} &\text{ colors } (c,b)\to(c,a)
   &&\text{for every } (c,b)\in\XX, (c,a) \in \YY ,
\end{aligned}
\end{equation}
and every edge arises exactly once in this way. In particular $\bmu_{ab}$
joins the source $(a,b)$ to the diagonal target $(b,b)$, and $\blam_{ab}$
joins $(a,b)$ to $(a,a)$.
\end{definition}

By \eqref{eq:S}, the nonzero entry of $\Sop_{\lambda,\mu}$ in position
$(u,y)$ corresponds to an edge $u \to y$ of $\Gc_\mm$; it is the value of
the color of that edge if the edge is of type $\blam$, and minus that
value if it is of type $\bmu$. For a fixed source and target the type and
the color are unique. Since $\XX \subseteq \YY$ whenever $\mm$ has no
juxtaposed segments, a vertex may occur on either side of an edge, and we
fix the convention that a \emph{neighbor} of $v \in \YY$ is always on the
source side, i.e.\ a $w \in \XX$ with an edge $w \to v$.

\begin{definition}[minimal linked pairs]\label{def:covmm}
A pair $(i,k) \in \XX$ is \emph{minimal} if there is no $j \in \Kk\setminus
\{i,k\}$ with $(i,j) \in \YY$ and $(j,k) \in \XX$, or with $(i,j) \in \XX$
and $(j,k) \in \YY$. We write $\Cov(\mm) \subseteq \XX$ for the set of minimal pairs.
The \emph{cover graph} $\cgr_\mm$ is the undirected graph with vertex set
$\Kk$ and edge set $\Cov(\mm)$.
\end{definition}

For a finite graph $\cgr$ we write $\cyc(\cgr)$ for its \emph{cycle rank}, the dimension
of its cycle space (i.e., the number of edges minus the number of vertices plus the number of connected components).
It vanishes exactly when $\cgr$ is a forest. If $\cgr$
carries additional structure, such as directions or colors of its edges, we
disregard it in both counts.

Recall that a pseudotree is a connected graph that contains at most one cycle, i.e. such that
$\cyc(\cgr)\le1$. A pseudoforest is a graph whose connected components are pseudotrees.

\begin{lemma}[the counting bound]\label{lem:countmm}
Let $\acyc(\cgr_\mm)$ be the number of connected components of $\cgr_\mm$ that
contain no cycle. Then, for all $\lambda,\mu \in \bmdl$,
\begin{subequations}\label{eq:countbound}
\begin{align}
   \cork \Sop_{\lambda,\mu} &\;\ge\;
      \#\Cov(\mm) - \#\Kk + \acyc(\cgr_\mm) , \label{count:pair}\\
   \cork \Sop_\lambda &\;\ge\; \cyc(\cgr_\mm) . \label{count:sing}
\end{align}
\end{subequations}
The right-hand side of \eqref{count:pair} is the sum of the differences
$\#(\text{edges}) - \#(\text{vertices})$ over the components of $\cgr_\mm$
containing a cycle. In particular $\mm$ is not \prig unless $\cgr_\mm$ is
a \emph{pseudoforest} --- and thus not as soon as
$\#\Cov(\mm) > \#\Kk$, while $\mm$ is not rigid unless $\cgr_\mm$ is a forest.
\end{lemma}

\begin{proof}
For every $(i,k) \in \Cov(\mm)$ and $a\in\alg$
\begin{equation}\label{eq:covcol}
   \text{the $\bmdlbss_{ik}$-coordinate of } \Sop_{\lambda,\mu}(a)
   \;=\; \lambda_{ik}\,a_{ii} \;-\; \mu_{ik}\,a_{kk} ,
\end{equation}
where $a_{ii}$ are the diagonal coordinates of $a$ (in the basis $\algbss_{pq}$).
Thus the composition $\tilde\Sop_{\lambda,\mu}$ of $\Sop_{\lambda,\mu}$ with the projection of $\bmdl$
onto the coordinates indexed by $\Cov(\mm)$ factors through the $\#\Kk$
diagonal coordinates $a_{jj}$, $j \in \Kk$, and is given there by the
$\Cov(\mm)\times\Kk$ matrix $\mathsf{M}$ whose row $(i,k)$ is $\lambda_{ik}e_i -
\mu_{ik}e_k$: the transposed incidence matrix of $\cgr_\mm$, carrying the
two colors of each edge as its entries. Grouping rows and columns
according to the connected components of $\cgr_\mm$ makes $\mathsf{M}$ block
diagonal, and a component with $v$ vertices and $e$ edges contributes a
block of rank at most $\min(e,v)$, which is $v-1$ when the component is a
tree, since then $e = v-1$. Hence $\rank \mathsf{M} \le \#\Kk - \acyc(\cgr_\mm)$ and
\[
   \cork \Sop_{\lambda,\mu} \;\ge\; \cork \tilde\Sop_{\lambda,\mu}
   \;=\; \#\Cov(\mm) - \rank \mathsf{M} \;\ge\;
   \#\Cov(\mm) - \#\Kk + \acyc(\cgr_\mm) ,
\]
which is \eqref{count:pair}. For $\mu = \lambda$ the row $(i,k)$ of $\mathsf{M}$ is
$\lambda_{ik}(e_i - e_k)$, so that $\rank \mathsf{M}$ is at most the rank of the
oriented incidence matrix of $\cgr_\mm$, that is, $\#\Kk$ minus the number of
connected components; this gives \eqref{count:sing}.
\end{proof}

\begin{corollary}\label{cor:gencov}
If the $G$-orbit of $\lambda$ is open in $\bmdl$ --- equivalently, if
$\Sop_\lambda$ is surjective --- then $\lambda_{ik} \ne 0$ for every
$(i,k) \in \Cov(\mm)$. More generally, if $\Sop_{\lambda,\mu}$ is
surjective then $(\lambda_{ik},\mu_{ik}) \ne (0,0)$ for every $(i,k) \in
\Cov(\mm)$.
\end{corollary}
This follows immediately from \eqref{eq:covcol}.

The minimal pairs have an intrinsic meaning: they index the top of the
bimodule, in the sense of \Cref{ssec:gen}.

\begin{lemma}[the top of the bimodule]\label{lem:top}
Assume that the segments of $\mm$ are pairwise distinct, as they are for a
regular multisegment. Then the radical of $\alg$ is spanned by the
$\algbss_{ij}$ with $i \ne j$, and
\begin{equation}\label{eq:top}
   \rbm(\bmdl) \;=\; \Span\{\bmdlbss_{ik} :
      (i,k) \in \XX\setminus\Cov(\mm)\} ,
\end{equation}
so that the classes of $\bmdlbss_{ik}$, $(i,k) \in \Cov(\mm)$, form a basis
of $\topp(\bmdl)$. Consequently, for $\lambda \in \bmdl$,
\begin{equation}\label{eq:gencrit}
   \alg\lambda\alg = \bmdl
   \quad\Longleftrightarrow\quad
   \lambda_{ik} \ne 0 \ \text{ for all } (i,k) \in \Cov(\mm) .
\end{equation}
\end{lemma}

\begin{proof}
The quotient of $\alg$ by the span of the off-diagonal $\algbss_{ij}$ is
$\kk^{\Kk}$, and that span is nilpotent because $i \YY j$ with $i \ne j$
forces $a_i \le a_j$, $b_i \le b_j$ and, the segments being distinct, at
least one of the two inequalities is strict; this gives the first
assertion. By \eqref{eq:units},
$\rad(\alg)\bmdl$ is spanned by the $\algbss_{ij}\bmdlbss_{jk} =
\bmdlbss_{ik}$ with $(i,j) \in \YY$, $i \ne j$, and $(j,k) \in \XX$, and
$\bmdl\rad(\alg)$ by the $\bmdlbss_{ij}\algbss_{jk} = \bmdlbss_{ik}$ with
$(i,j) \in \XX$ and $(j,k) \in \YY$, $j \ne k$. These are precisely the
two alternatives of \Cref{def:covmm}, whence \eqref{eq:top}.

By \eqref{eq:units} the subbimodule of $\topp(\bmdl)$ generated by
the class of $\lambda$ contains the classes of $\algbss_{ii}\lambda
\algbss_{kk} = \lambda_{ik}\bmdlbss_{ik}$ and is spanned by them, so that
$\lambda$ generates $\topp(\bmdl)$ precisely when all its coordinates
at $\Cov(\mm)$ are nonzero. Now apply \eqref{eq:nakayama}.
\end{proof}

\begin{example}[Leclerc's example]\label{ex:leclerc}
Let $\mm_{\Lec} = [1,2]+[2,4]+[3,3]+[4,5]$, the multisegment of
Leclerc's imaginary vector in type $A_5$ \cite{MR1959765}*{\S2.7}. Here,
\[
   \XX = \{(1,2),\,(1,3),\,(2,4),\,(3,4)\},
   \qquad
   \YY = \{(i,i)\}\cup\{(1,2),\,(2,4)\},
\]
and all elements of $\XX$ are minimal. Thus, $\mm_{\Lec}$ is not rigid by \Cref{lem:countmm}.

On the other hand, for $\lambda=\bmdlbss_{13}+\bmdlbss_{34}$, $\mu=-(\bmdlbss_{12}+\bmdlbss_{24})$ we have
\[
\Sop_{\lambda,\mu}(\algbss_{11}, \algbss_{22},\algbss_{33},\algbss_{44})=
(\bmdlbss_{13},\bmdlbss_{12},\bmdlbss_{34},\bmdlbss_{24}).
\]
Hence $\mm_{\Lec}$ is \prig.
\end{example}

\begin{example} \label{exma:nonprigid}
An example of a non-\prig multisegment is
\[
   \mm = [1,4]+[2,3]+[3,6]+[4,5]+[6,7] .
\]
Here
\begin{equation}\label{eq:nonprigrel}
   \XX = \bigl(\{1,2\}\times\{3,4\}\bigr) \cup
   \bigl(\{3,4\}\times\{5\}\bigr), \qquad
   \YY = \{(i,i)\}\cup\{(1,3),\,(1,4),\,(2,3),\,(3,5)\}.
\end{equation}
Once again, all six elements of $\XX$ are minimal, and by \Cref{lem:countmm} $\mm$ is not \prig.
\end{example}

\subsection{Arbitrary fields}\label{ssec:arbnilp}

Let $\kk_0$ be an arbitrary field with algebraic closure $\kk$, and let $Q$ be
of Dynkin type.

In the Dynkin case all the objects of \Cref{ssec:quivers} and
\Cref{ssec:dynkin} are defined over the prime field. Indeed, $R_Q(V)$,
$R_{\overline{Q}}(V)$ and $R_\Pi(V)$ are cut out by equations with integral
coefficients and, Gabriel's theorem being valid over an arbitrary field, by
the species version of \cite{MR447344}*{Proposition 2.6(a)}, the
indecomposable representations $M_Q(\beta)$ admit models over the prime field.
Consequently each orbit $G_V\cdot M_Q(\mm)$ is the image of the orbit map of a
point defined over the prime field, and $G_V$ is geometrically integral; hence
the orbit, and with it its conormal bundle and the closure $\cmp_\mm$ of the
latter, is geometrically irreducible and defined over the prime field.
Therefore every irreducible component of $\Lambda(V)_{\kk}$ is
defined over the prime field, a fortiori over $\kk_0$, and both $\Comp(\bd)$ and its parametrization
$\lambda_Q$ by formal sums of positive roots are insensitive to the field. This is what allows
us to assume throughout that $\kk$ is algebraically closed. Note also that
the purity of $\Lambda(V)$ quoted in \Cref{ssec:quivers} is available in
every characteristic: the standing hypothesis of \cite{MR1088333} is only
that the ground field be algebraically closed, characteristic zero being
imposed there from \S13 on, whereas \S12 --- which contains
\cite{MR1088333}*{Theorem 12.3} --- is free of it.

The algebra $\alg$ and the bimodule $\bmdl$ attached to $\mm \in \Mfr$ are
likewise defined over the prime field and are compatible with base change. This is clear from their description in \Cref{ssec:prig}.
In equioriented type $A$ this is visible in \Cref{lem:alg}: the multiplication
table \eqref{eq:units} has integral structure constants, determined by
$\XX_\mm$ and $\YY_\mm$ alone.
Accordingly, the deficiencies, rigidity and \prigity of $\cmp_\mm$ are well defined. The geometric readings of $\dfr_\ast(\mm)$ obtained in
\Cref{ssec:prig} are then statements about $\cmp_\mm$ over $\kk$.

\section{Endpoint-separated heredity}\label{sec:esi}

The mechanism of \Cref{sec:hered} requires an
$\alg_{\diag}$-bilinear derivation with prescribed image. We now
produce one, in the equioriented type $A$, under an endpoint-separation hypothesis and deduce the heredity
theorem (\Cref{thm:esi}).

\subsection{Endpoint separation produces a relative derivation}\label{ssec:typeAsep}

We specialize to the equioriented type $A$ and the matrix-unit model of
\Cref{ssec:typeA}. Throughout, $\mm = \sum_{i \in \Kk}\Delta_i$, $\Delta_i =
[a_i,b_i]$, is a multisegment with integral endpoints, and
$\alg$, $\bmdl$ are the algebra and bimodule
attached to $\mm$, with bases $\{\algbss_{ij}\}_{(i,j)\in \YY}$ and
$\{\bmdlbss_{ij}\}_{(i,j)\in \XX}$.

Suppose we have a partition $\Kk = I \sqcup J$. Consider the decomposition 
\begin{equation} \label{eq:Mdecomp}
   M_Q(\mm) \;=\; M_Q(\mm_I)\oplus M_Q(\mm_J),\ \ \    \mm_P = \sum_{i \in P}\Delta_i, P\in\{I,J\}.
\end{equation}
For $P \in \{I,J\}$ put
\[
\XX_P = \XX \cap (P\times P), \qquad   \YY_P = \YY \cap (P\times P),
\]
so that $\XX_P = \XX_{\mm_P}$ and $\YY_P = \YY_{\mm_P}$, and let
$\XX_{\mix}$, $\YY_{\mix}$ be the intersections of $\XX$,
$\YY$ with $(I\times J)\cup(J\times I)$. The orthogonal idempotent
decomposition corresponding to \eqref{eq:Mdecomp} is
\[
   1 = e_I + e_J, \qquad   e_I = \sum_{i\in I}\algbss_{ii}, \quad e_J = \sum_{j\in J}\algbss_{jj} .
\]
The Peirce corners are spanned by the correspondingly indexed basis
vectors: $\alg_{PR}$ has basis $\{\algbss_{ij} : (i,j) \in \YY\cap(P\times
R)\}$ and $\bmdl_{PR}$ has basis $\{\bmdlbss_{ij} : (i,j) \in \XX\cap(P\times
R)\}$. In particular the corners are those attached to $\mm_I$ and
$\mm_J$ (consistently with \Cref{lem:corners}), and
\begin{equation}\label{eq:Xcount}
   \dim \bmdl_P = \#\XX_P, \qquad
   \dim \bmdl_{\mix} = \#\XX_{\mix}, \qquad
   \#\XX = \#\XX_I + \#\XX_J + \#\XX_{\mix} .
\end{equation}
The standing hypothesis is \emph{endpoint separation}:
\begin{equation}\label{eq:sep}
   \{a_i : i \in I\} \cap \{a_j : j \in J\} = \varnothing,
   \qquad
   \{b_i : i \in I\} \cap \{b_j : j \in J\} = \varnothing .
\end{equation}
No distinctness of endpoints is assumed \emph{within} either block. \Cref{fig:sep} shows an example.

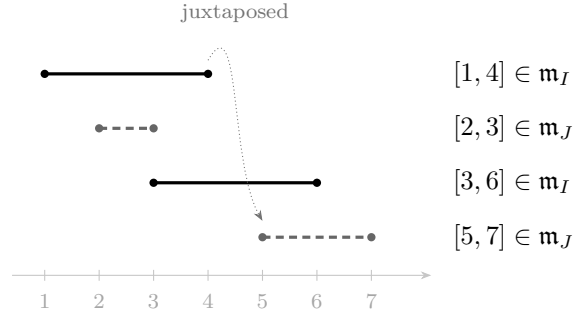
\begin{figure}[!htbp]
\centering
\begin{tikzpicture}[scale=0.72]
\draw[gray!45,-{Stealth[length=4pt]}] (0.4,0.3)--(8.1,0.3);
\foreach \x in {1,...,7} {\draw[gray!45] (\x,0.2)--(\x,0.4); \node[gray!70,font=\scriptsize,below=1pt] at (\x,0.2) {$\x$};}
% block I: solid
\draw[very thick] (1,4)--(4,4); \node[permdot] at (1,4){}; \node[permdot] at (4,4){};
\draw[very thick] (3,2)--(6,2); \node[permdot] at (3,2){}; \node[permdot] at (6,2){};
% block J: dashed
\draw[very thick,densely dashed,black!60] (2,3)--(3,3);
\node[permdot,fill=black!60] at (2,3){}; \node[permdot,fill=black!60] at (3,3){};
\draw[very thick,densely dashed,black!60] (5,1)--(7,1);
\node[permdot,fill=black!60] at (5,1){}; \node[permdot,fill=black!60] at (7,1){};
\node[font=\footnotesize,anchor=west] at (8.3,4) {$[1,4]\in\mm_I$};
\node[font=\footnotesize,anchor=west] at (8.3,3) {$[2,3]\in\mm_J$};
\node[font=\footnotesize,anchor=west] at (8.3,2) {$[3,6]\in\mm_I$};
\node[font=\footnotesize,anchor=west] at (8.3,1) {$[5,7]\in\mm_J$};
\draw[black!55,densely dotted,-{Stealth[length=4pt]}]
   (4,4.25) to[out=60,in=120] (5,1.3);
\node[font=\scriptsize,black!55,anchor=south] at (4.5,4.75) {juxtaposed};
\end{tikzpicture}
\caption{An endpoint-separated partition \eqref{eq:sep} of
$\mm = [1,4]+[3,6]+[2,3]+[5,7]$ into $\mm_I = [1,4]+[3,6]$ (solid) and
$\mm_J = [2,3]+[5,7]$ (dashed). The two blocks share no beginning --- $1,3$
against $2,5$ --- and no end --- $4,6$ against $3,7$. The only mixed
juxtaposition is $[1,4]\,|\,[5,7]$; by \Cref{lem:mixed}\labelcref{lem:fringe}
it contributes the single basis vector of the fringe $\Ff$, and by
\Cref{rem:obtusepairs} it is what makes
$-(\grdim M_I,\grdim M_J) = \dim\Ff = 1$.}
\label{fig:sep}
\end{figure}

The result of this subsection is the following.

\begin{proposition}[endpoint separation gives an \obt decomposition]
\label{prop:sepobtuse}
Assume \eqref{eq:sep}. Then
\[
   M_Q(\mm) \;=\; M_Q(\mm_I)\oplus M_Q(\mm_J)
\]
is an \obt decomposition in the sense of \Cref{def:obtuse}.
\end{proposition}

We describe the fringe subbimodule $\Ff\subseteq \bmdl_{\mix}$ in \Cref{lem:mixed}, and the derivation in \Cref{lem:taut}.
Combined with \Cref{cor:schema}, the proposition gives the heredity theorem (\Cref{thm:esi}) below.

The following elementary result is crucial.
\begin{lemma}[mixed pairs and fringe pairs]\label{lem:mixed}
Under \eqref{eq:sep} the following hold.
\begin{enumerate}
\item\label{mix:incl} We have $\YY_{\mix} \subseteq\XX_{\mix}$.
\item\label{mix:adj} For every $(i,j) \in \XX_{\mix}$,
\[
   (i,j) \notin \YY \iff a_j = b_i + 1 .
\]
\item \label{lem:fringe} Define the \emph{fringe subspace}
\begin{equation}\label{eq:fringe}
   \Ff = \Span_\kk\{\bmdlbss_{ij} : (i,j) \in \XX_{\mix}\setminus
   \YY_{\mix}\} \subseteq \bmdl_{\mix} .
\end{equation}
Then, $\Ff$ is an $\alg$-subbimodule of $\bmdl$, contained in $\bmdl_{\mix}$.
\end{enumerate}
\end{lemma}

\begin{proof}
The first two parts are immediate from the definition of $\YY_{\mix}$ and $\XX_{\mix}$ and the separation assumption.

For the last part, note that if $\bmdlbss_{ij}$, $(i,j)\in\XX_{\mix}$ is a fringe basis vector, so that $a_j = b_i + 1$,
and for some $(u,i) \in \YY$, $\algbss_{ui}\bmdlbss_{ij} =\bmdlbss_{uj}$ is non-zero, i.e.,  $(u,j) \in \XX$,
then necessarily $b_u=b_i$.
By \eqref{eq:sep} this forces $u$ and $i$ to
lie in the same block; hence $u$ and $j$ lie in different blocks and $\bmdlbss_{uj}$ is again a fringe vector.

The case of a nonzero right product is symmetric.
\end{proof}

\begin{remark}[endpoint separation is pairwise obtuseness]\label{rem:obtusepairs}
Let $\gamma_i = \grdim M(\Delta_i)$ be the positive root corresponding to the segment $\Delta_i$.
The symmetrized Euler form of \Cref{ssec:quivers} and the dictionary of
\Cref{ssec:typeA} give, for $i \ne j$ in $\Kk$,
\begin{equation}\label{eq:pairform}
   (\gamma_i,\gamma_j) = \#\bigl(\{(i,j),(j,i)\}\cap \YY\bigr) -
   \#\bigl(\{(i,j),(j,i)\}\cap \XX\bigr) .
\end{equation}
Consequently
\[
   (\gamma_i,\gamma_j) \le 0
   \iff
   a_i \ne a_j \text{ and } b_i \ne b_j .
\]
Thus the endpoint separation hypothesis \eqref{eq:sep} says precisely
that the roots of the segments of $\mm_I$ and $\mm_J$ are \emph{pairwise obtuse}.
Moreover, $(\gamma_i,\gamma_j)\ge-1$ with equality precisely when $[a_i,b_i]$ and $[a_j,b_j]$ are juxtaposed.
The number of mixed juxtaposed pairs is $\#(\XX_{\mix}\setminus \YY_{\mix})$ by \Cref{lem:mixed}.
Thus,
\[
-(\grdim M_I,\grdim M_J) =\#(\XX_{\mix}\setminus \YY_{\mix}),
\]
the dimension of the fringe \eqref{eq:fringe}, in accordance with \eqref{eq:dimF}.
\end{remark}

Put $\Vv = \bmdl/\Ff$ and write $\bar\bmdlbss_{ij}$ for the class of
$\bmdlbss_{ij}$ in $\Vv$. Since $\YY_{\mix} \subseteq
\XX_{\mix}$ (\Cref{lem:mixed}), the mixed basis vectors surviving
in $\Vv_{\mix} = \bmdl_{\mix}/\Ff$ are exactly those
indexed by $\YY_{\mix}$: $\{\bar\bmdlbss_{ij} : (i,j) \in
\YY_{\mix}\}$ is a basis of $\Vv_{\mix}$, while
$\{\algbss_{ij} : (i,j) \in \YY_{\mix}\}$ is a basis of
$\alg_{\mix}$, indexed by the same set.

\begin{lemma}[the tautological matching]\label{lem:taut}
Assume \eqref{eq:sep}. Then the tautological map
\begin{equation}\label{eq:taut}
   \iota : \alg_{\mix} \longrightarrow \Vv_{\mix}, \qquad
   \iota(\algbss_{ij}) = \bar\bmdlbss_{ij},
   \qquad (i,j) \in \YY_{\mix} ,
\end{equation}
is an isomorphism of $\alg_{\diag}$-bimodules and satisfies
\eqref{eq:iotaporp}. Consequently, the associated derivation
$f = \iota\circ\ad(e_I) : \alg \to \Vv$, namely
\[
   f(\algbss_{ij}) = \bigl(\chi(i)-\chi(j)\bigr)\,\bar\bmdlbss_{ij},
   \qquad (i,j) \in \YY ,
\]
where $\chi : \Kk \to \{0,1\}$ is the indicator function of the block $I$,
is an $\alg_{\diag}$-bilinear derivation restricting to an isomorphism of
$\alg_{\mix}$ onto $\Vv_{\mix}$. Thus the fringe $\Ff$ and the derivation
$f$ satisfy the condition of \Cref{prop:twoblock}.
\end{lemma}

\begin{proof}
This is a straightforward verification.
It amounts to \eqref{eq:units} together with \Cref{lem:mixed}: the
multiplication rules for the $\algbss$'s and the $\bmdlbss$'s have the same index
bookkeeping, and the only discrepancy is that $\algbss_{ui}\algbss_{ij}$ may vanish
(because $\YY$ is not transitive) while $\bmdlbss_{uj}\ne0$; in that case
$(u,j)\in\XX_{\mix}\setminus\YY_{\mix}$, so that $\bar\bmdlbss_{uj}=0$ in $\Vv$.
The opposite discrepancy cannot occur, by \Cref{lem:mixed}\labelcref{mix:incl}.
The identities \eqref{eq:iotaporp} are immediate from \eqref{eq:units}, both sides
being $\delta_{j\ell}\bar\bmdlbss_{im}$, resp.\ $\delta_{mi}\bar\bmdlbss_{\ell j}$.
\end{proof}

This proves \Cref{prop:sepobtuse}.

\subsection{The heredity theorem}

By \Cref{prop:sepobtuse} and \Cref{cor:schema} we conclude

\begin{theorem}[endpoint-separated heredity]\label{thm:esi}
Let $\mm = \sum_{i\in \Kk}[a_i,b_i]$ be a multisegment with integral endpoints
and let $\Kk = I\sqcup J$ satisfy \eqref{eq:sep}. If $\mm$ is rigid, then
$\mm_I$ and $\mm_J$ are rigid. If $\mm$ is \prig, then $\mm_I$ and $\mm_J$ are
\prig.
\end{theorem}

\Cref{thm:esi} is \Cref{thm:mainB} of the introduction: for an endpoint-separated
partition, rigidity (resp.\ \prigity) of the irreducible component $\cmp_\mm$ is
inherited by $\cmp_{\mm_I}$ and $\cmp_{\mm_J}$.
Note that when the $a_i$'s are distinct and the $b_i$'s are distinct,
i.e., $\mm$ is regular in the sense of \cite{MR3866895}*{\S3.5},
the hypothesis \eqref{eq:sep} is satisfied by \emph{every} partition.

The hypothesis \eqref{eq:sep} cannot simply be dropped, and neither half of
it can be dropped by itself. In the first example below only one half is
dropped and rigidity already fails to be inherited.\footnote{Upon reflection of the segments $[a,b] \mapsto [r+1-b,r+1-a]$,
the same conclusion holds for the other half.}
In the second example both halves are dropped, and the conclusion fails for both rank conditions at once.

\begin{example}[one half of \eqref{eq:sep} is not enough]\label{ex:half}
Let
\[
   \mm \;=\; \underbrace{[1,2]+[2,5]+[3,4]+[5,6]}_{\textstyle \mm_I}
   \;+\; \underbrace{[4,4]}_{\textstyle \mm_J}.
\]
The beginnings are separated, while the ends are not.
Exactly as in \Cref{ex:leclerc},
\[
   \XX_{\mm_I} = \{(1,2),\,(1,3),\,(2,4),\,(3,4)\}, \qquad
   \YY_{\mm_I} = \{(i,i)\}\cup\{(1,2),\,(2,4)\} ,
\]
and $\mm_I$ is not rigid. For $\mm$ itself
\[
   \XX = \XX_{\mm_I}\cup\{(5,4)\}, \qquad
   \YY = \YY_{\mm_I}\cup\{(5,5),\,(3,5)\} ,
\]
so $\dim\bmdl = 5$ and $\dim\alg = 8$. For $\lambda \;=\; \bmdlbss_{12}+\bmdlbss_{13}+\bmdlbss_{24}+\bmdlbss_{54}$ we have
\[
   \Sop_\lambda(\algbss_{11},\algbss_{22},\algbss_{33},\algbss_{55},\algbss_{35})=
   (\bmdlbss_{12}+\bmdlbss_{13},\bmdlbss_{24}-\bmdlbss_{12},-\bmdlbss_{13},\bmdlbss_{54},\bmdlbss_{34}).
\]
Hence $\Sop_\lambda$ is surjective (over any field) and $\mm$ is rigid. Thus, rigidity is not inherited.

The mechanism of \Cref{ssec:typeAsep} breaks down at its very first step:
here $\YY_{\mix} = \{(3,5)\}$ is not contained in $\XX_{\mix} =
\{(5,4)\}$, against \Cref{lem:mixed}\labelcref{mix:incl}, and the line
$\kk\bmdlbss_{54}$ that \eqref{eq:fringe} would take for the fringe is not
an $\alg$-subbimodule, since $\algbss_{35}\bmdlbss_{54} = \bmdlbss_{34}
\ne 0$ lies in $\bmdl_I$.
\end{example}

\begin{example}[\prigity is not inherited without separation]\label{ex:notsepflat}
Let
\[
   \mm \;=\; \underbrace{[1,4]+[2,3]+[3,6]+[4,5]+[6,7]}_{\textstyle \mm_I}
   \;+\; \underbrace{[3,4]}_{\textstyle \mm_J} .
\]
Both halves of \eqref{eq:sep} fail.
Here $\mm_I$ is the multisegment of \Cref{exma:nonprigid}, which is not \prig. Its relations
$\XX_{\mm_I}$, $\YY_{\mm_I}$ are those displayed in \eqref{eq:nonprigrel}. For $\mm$,
\[
   \XX = \XX_{\mm_I}\cup\{(2,6),\,(6,4)\}, \qquad
   \YY = \YY_{\mm_I}\cup\{(6,6),\,(1,6),\,(2,6),\,(6,3),\,(6,4)\} ,
\]
so $\dim\bmdl = 8$ and $\dim\alg = 14$. Taking $\lambda \;=\; \bmdlbss_{13}+\bmdlbss_{26}+\bmdlbss_{35}+\bmdlbss_{45}+\bmdlbss_{64}$ we get
\begin{align*}
   \Sop_\lambda(\algbss_{11}) &= \bmdlbss_{13}, &
   \Sop_\lambda(\algbss_{16}) &= \bmdlbss_{14}, &
   \Sop_\lambda(\algbss_{22}) &= \bmdlbss_{26}, &
   \Sop_\lambda(\algbss_{26}) &= \bmdlbss_{24},\\
   \Sop_\lambda(\algbss_{63}) &= -\bmdlbss_{23}, &
   \Sop_\lambda(\algbss_{33}) &= \bmdlbss_{35}-\bmdlbss_{13}, &
   \Sop_\lambda(\algbss_{55}) &= -\bmdlbss_{35}-\bmdlbss_{45}, &
   \Sop_\lambda(\algbss_{66}) &= \bmdlbss_{64}-\bmdlbss_{26}.
\end{align*}
Thus, $\Sop_\lambda$ is surjective over every field and hence $\mm$ is rigid, and in particular \prig,
although $\mm_I$ is neither. So \eqref{eq:twoblock} fails here for both rank
conditions at once, and not merely its consequence for $\dfr_\rsub$ as in
\Cref{ex:half}.
\end{example}

\subsection{Arbitrary fields}\label{ssec:arbesi}

\Cref{thm:esi}, and with it \Cref{thm:mainB}, holds over an arbitrary field
$\kk_0$, the invariants being read as in \Cref{ssec:arbnilp}. Indeed, the
whole input is \Cref{prop:sepobtuse}, which asserts that the decomposition
$M_Q(\mm) = M_Q(\mm_I)\oplus M_Q(\mm_J)$ is \obt; the fringe $\Ff$ of
\Cref{lem:mixed} and the tautological matching of \Cref{lem:taut} are defined
by the relations $\XX$ and $\YY$ alone, hence over the prime field, so the
decomposition is \obt already over $\kk_0$. \Cref{cor:schema} over $\kk_0$
(\Cref{ssec:arbhered}) then gives the conclusion.

Both examples above are, likewise, valid over an arbitrary field: the
displayed $\lambda$ has coordinates $0$ and $1$ and the resulting
surjectivity is witnessed by a triangular submatrix with diagonal entries
$\pm1$.

\section{Permutations and the \psm class}\label{sec:perm}

The remaining sections specialize the theory to a distinguished family of
multisegments, those attached to permutations, and prove the
characterization of the \prig members of that family announced as
\Cref{thm:mainA} below. We rely on an exhaustive finite check of permutations of sizes at most eight
(\Cref{prop:critsmall}), and on the smaller check of \Cref{cor:notower};
both are done externally.

\subsection{Conventions on permutations}\label{ssec:permconv}

Permutations $\sigma \in S_n$ act on $\{1,2,\dots,n\}$ and are written
in one-line notation $\sigma = \sigma_1\sigma_2\cdots\sigma_n$; we
write $|\sigma| = n$. We speak of the \emph{points} of $\sigma$, a point
being a position together with its value; for a point $x$ we write
$\pos(x)$ for its position and $\sigma_x$ for its value. A point $y$ lies
\emph{northeast} of a point $x$ if $\pos(y) > \pos(x)$ and $\sigma_y >
\sigma_x$, and \emph{southwest} if the reverse holds. A pair of positions
$i<j$ is an \emph{inversion} if $\sigma_i > \sigma_j$ and an \emph{ascent
pair} if $\sigma_i < \sigma_j$; a subsequence is decreasing exactly when
all its pairs are inversions.

\emph{Patterns and deletions.} For $\ptst \subseteq \{1,\dots,n\}$,
$\std(\sigma|_{\ptst})$ is the permutation of $\{1,\dots,|\ptst|\}$
order-isomorphic to the subsequence of $\sigma$ indexed by $\ptst$. We write
$\tau \preceq \sigma$, and say that $\sigma$ \emph{contains} $\tau$, if
$\tau = \std(\sigma|_{\ptst})$ for some $\ptst$; for a single position $d$ we write
$\sigma - d = \std(\sigma|_{\{1,\dots,n\}\setminus\{d\}})$. For a set
$W$ of permutations, $\Av(W)$ denotes the class of permutations
containing no member of $W$. Such a class is closed under pattern
containment, and if $W$ is an antichain in the pattern order then $W$ is
the set of minimal permutations outside $\Av(W)$, the \emph{basis} of the
class.

\emph{Symmetries.} Let $\wo \in S_n$ be the longest element, $\wo(i) =
n+1-i$, and let $\rc\sigma = \wo\sigma\wo$ be the reverse-complement,
$(\rc\sigma)_i = n+1-\sigma_{n+1-i}$, the rotation of the plot of
$\sigma$ by $180^\circ$. Together with inversion $\sigma \mapsto
\sigma^{-1}$ it generates a Klein four-group acting on $S_n$.
Inversion reflects the plot in the main diagonal, so it preserves
\emph{decreasing} subsequences, as does $\rc$.

\emph{Direct sums.} For $\alpha \in S_a$, $\beta \in S_b$ the
\emph{direct sum} $\alpha\oplus\beta \in S_{a+b}$ is
$(\alpha\oplus\beta)_i = \alpha_i$ for $i \le a$ and
$(\alpha\oplus\beta)_i = \beta_{i-a} + a$ for $i > a$. A permutation is
\emph{decomposable} if it is such a sum with both summands nonempty, and
\emph{indecomposable} otherwise. The decompositions of $\sigma \in S_n$
correspond to the \emph{cuts} $t$, that is, the indices $1 \le t \le n-1$ with
$\sigma(\{1,\dots,t\}) = \{1,\dots,t\}$; every permutation is uniquely a
direct sum of indecomposable summands.

\emph{Records, windows, corners.} A \emph{record} of $\sigma$ is either a
left-to-right maximum, that is a position $d$ with $\sigma_e < \sigma_d$
for all $e < d$, or a right-to-left minimum. For a left-to-right maximum
$d$ the \emph{window} is $\win(d) = \{k > d : \sigma_k < \sigma_d\}$, read
as the subsequence $\sigma|_{\win(d)}$; for a right-to-left minimum the
window is defined through $\rc$. A record is \emph{\good} if its window
is decreasing. The four \emph{corners} of $\sigma$ (which are not necessarily distinct) are the first
position, the position of the value $n$, the position of the value $1$
and the last position, called the \emph{\Wcr, \Ncr, \Scr} and \emph{\Ecr}
corner respectively. The names refer to the plot of $\sigma$, positions
increasing to the east and values to the north: these four points are the
leftmost, the highest, the lowest and the rightmost one. The \Wcr and \Ncr
corners are left-to-right maxima, the \Scr and \Ecr corners right-to-left
minima. A corner is \emph{\good} if it is a \good record and
\emph{blocked} otherwise, and $\sigma$ is \emph{\cb} if all four of its
corners are blocked, that is, if it has no \good corner.

The Klein four-group permutes the four corner \emph{labels} simply
transitively (the corners themselves may coincide, as the \Wcr and \Ncr
corners do in $4231$):
\begin{equation}\label{eq:cornersym}
\begin{aligned}
   \sigma &\mapsto \sigma^{-1} &&:\quad \text{(\Wcr \Scr)(\Ncr \Ecr)},\\
   \sigma &\mapsto \rc\sigma   &&:\quad \text{(\Wcr \Ecr)(\Ncr \Scr)} .
\end{aligned}
\end{equation}
Records, windows and \goodness are likewise preserved by the action, so \cbness is invariant.

Deleting a \good corner does not destroy indecomposability:

\begin{lemma}[deleting a \good corner]\label{lem:cornerindec}
Let $\sigma \in S_n$, $n \ge 2$, be indecomposable and let $\crn$ be a
\good corner of $\sigma$. Then $\sigma - \crn$ is indecomposable.\footnote{\Goodness is needed here: $312$ is indecomposable, but its \Wcr corner,
whose window reads $12$, leaves $12$ upon deletion.}
\end{lemma}

\begin{proof}
Decomposability is invariant under the Klein four-group, since
$(\alpha\oplus\beta)^{-1} = \alpha^{-1}\oplus\beta^{-1}$ and
$\rc(\alpha\oplus\beta) = \rc\beta\oplus\rc\alpha$, so by
\eqref{eq:cornersym} we may assume that $\crn$ is the \Wcr corner. Put $a
= \sigma_1$ and suppose that $\sigma-\crn$ had a cut, that is, that for
some $2 \le t \le n-1$ the values $\sigma_2,\dots,\sigma_t$ are the $t-1$
smallest elements of $\{1,\dots,n\}\setminus\{a\}$.

If $a \le t$, these are $\{1,\dots,t\}\setminus\{a\}$, so that
$\sigma(\{1,\dots,t\}) = \{1,\dots,t\}$ and $t$ is a cut of $\sigma$,
contrary to the assumption.

If $a > t$, they are $\{1,\dots,t-1\}$ and all lie in the window $\win(1)$.
\Goodness makes $\sigma|_{\win(1)}$ decreasing, so that $\sigma_k = t-k+1$
for $2 \le k \le t$, and in particular $\sigma_t = 1$. On the other hand
the value $t$ is none of $a,1,\dots,t-1$, hence occupies a position $k >
t$, and $t < a$ puts $k$ in $\win(1)$. Thus $\sigma|_{\win(1)}$ takes the value
$1$ at $t$ and the larger value $t$ at $k > t$, contradicting \goodness.
\end{proof}

\emph{Minimal inversions and the cover graph.} A \emph{witness} against
an inversion $(i,k)$ is a position $j$ with $i<j<k$ and $\sigma_i >
\sigma_j > \sigma_k$, that is, a point lying strictly inside the
rectangle spanned by the two; the inversion is \emph{minimal} if it has
none. We write $\Cov(\sigma)$ for the set of minimal inversions of
$\sigma$, and $\cgr_\sigma$ for the \emph{cover graph}, the undirected
graph on the $n$ positions of $\sigma$ with one edge for each minimal
inversion. For instance an inversion with consecutive values is minimal. Both notions are visibly invariant under the Klein
four-group, which therefore carries $\Cov(\sigma)$ to $\Cov(\sigma^{-1})$
and $\Cov(\rc\sigma)$, and $\cgr_\sigma$ isomorphically to
$\cgr_{\sigma^{-1}}$ and $\cgr_{\rc\sigma}$.

Deleting a point $p$ can destroy a minimal inversion only through $p$:
if $(i,k) \in \Cov(\sigma)$ and $p \notin \{i,k\}$, then $(i,k)$ is still
an inversion of $\sigma - p$, and a witness against it there would be a
witness in $\sigma$. Identifying throughout the points of $\sigma - p$
with the points of $\sigma$ other than $p$, this gives
\begin{subequations}\label{eq:covdelete}
\begin{align}
   \{(i,k) \in \Cov(\sigma) : p \notin \{i,k\}\}
   &\;\subseteq\; \Cov(\sigma - p) ,
   \label{eq:covdelete-incl}\\
   |\Cov(\sigma)|
   &\;\le\; |\Cov(\sigma - p)| + \deg_{\cgr_\sigma}(p) .
   \label{eq:covdelete-count}
\end{align}
\end{subequations}
The inclusion is in general strict, the deletion of a witness turning a
non-edge into an edge. For a corner it is an equality,
\begin{subequations}\label{eq:covcorner}
\begin{align}
   \Cov(\sigma-\crn) &= \{(i,k) \in \Cov(\sigma) : \crn \notin \{i,k\}\},
   \label{eq:covcorner-set}\\
   \cgr_{\sigma-\crn} &= \cgr_\sigma - \crn ,
   \label{eq:inhrtcrn}
\end{align}
\end{subequations}
the second being the first read on the cover graph: deleting a corner
removes its vertex and nothing else. Indeed, the inclusion $\supseteq$ is
\eqref{eq:covdelete-incl}; for $\subseteq$, let $(i,k) \in
\Cov(\sigma-\crn)$ and let $w$ be a witness against $(i,k)$ in $\sigma$.
Then $w = \crn$, since any other witness would survive in
$\sigma - \crn$. But a corner is extreme in position or in value, hence
never lies strictly inside the rectangle spanned by two other points; so
$\crn$ witnesses against no inversion at all, and $(i,k) \in
\Cov(\sigma)$.

If the corner is \good, more is true. Indeed
\begin{equation}\label{eq:gooddegree}
\text{a \good corner of $\sigma$ lies in at most one minimal inversion,}
\end{equation}
for let $\crn$ be the \Wcr corner, its window $w_1,\dots,w_r$ being
listed in increasing position and, by \goodness, in decreasing value.
A witness against $(\crn,w_k)$ is a point later than $\crn$ of smaller
value, hence lies in the window; so $(\crn,w_1)$ is minimal, while for
$k > 1$ the point $w_1$ witnesses against $(\crn,w_k)$. The other three
corners follow by \eqref{eq:cornersym}. Together with
\eqref{eq:inhrtcrn} this says that for a \good corner $\crn$ the graph
$\cgr_\sigma$ is recovered from $\cgr_{\sigma-\crn}$ by adding a new
vertex which is either isolated or a leaf.

\subsection{The multisegment of a permutation}\label{ssec:dictionary}

As in \cite{MR1458969}*{\S8.1} we consider type $A_{2n-1}$ and graded dimension
\[
   (1,2,\dots,n-1,n,n-1,\dots,2,1).
\]
The open $G_V$-stable subvariety of $R_Q(V)$ where the first $n-1$ arrows are injective
and the last $n-1$ are surjective encodes the flag variety of $\GL(n)$ with the Borel action.
The relevant multisegments are of the form $\mm_\sigma$, $\sigma \in S_n$
with index set $\Kk = \{1,\dots,n\}$ and segments
\begin{equation}\label{def:mm}
   \Delta_i \;=\; [\,a_i,\,b_i\,] \;=\; [\,i,\; 2n - \sigma_i\,],
   \qquad i \in \Kk .
\end{equation}

We abbreviate $\XX_\sigma = \XX_{\mm_\sigma}$, $\YY_\sigma =
\YY_{\mm_\sigma}$ and $\Gc_\sigma = \Gc_{\mm_\sigma}$ for the data of
\Cref{def:M}. These are the combinatorial data of $\sigma$ itself:
\begin{equation}\label{eq:dictUV}
\begin{gathered}
\begin{aligned}
   \XX_\sigma &= \{(i,k) : 1 \le i < k \le n,\ \sigma_i > \sigma_k\},\\
   \YY_\sigma &= \{(i,j) : 1 \le i \le j \le n,\ \sigma_i \ge \sigma_j\},\\
   \DD_\sigma &= \{(c,c) : 1 \le c \le n\},
\end{aligned}\\[2pt]
   \YY_\sigma = \XX_\sigma \sqcup \DD_\sigma ;
\end{gathered}
\end{equation}
that is, $\XX_\sigma$ is the set of inversions of $\sigma$. The minimal
pairs of $\mm_\sigma$ are likewise the minimal inversions of $\sigma$, so
that $\Cov(\sigma)=\Cov(\mm_\sigma)$ and $\cgr_\sigma=\cgr_{\mm_\sigma}$ in the notation of
\Cref{ssec:permconv} and \Cref{def:covmm}.

Note that dualizing and relabeling the vertices of the quiver by $v \mapsto 2n-v$,
takes a segment $[a,b]$ to $[2n-b,2n-a]$ and therefore $\mm_\sigma$ to $\mm_{\sigma^{-1}}$.\footnote{On the other hand,
rotation is not induced by an automorphism or a duality of the quiver: the segments $[i,2n-\sigma_i]$ and $[\wo(i),2n-\wo(\sigma_i)]$
have complementary lengths.}

Inversion and rotation act on pairs by the relabelings
\begin{equation}\label{eq:pairsym}
   \text{inversion:}\quad (i,k) \mapsto (\sigma_k,\sigma_i),
   \qquad
   \text{rotation:}\quad (i,k) \mapsto (\wo(k),\,\wo(i)) ,
\end{equation}
which carry $\XX_\sigma$ and $\YY_\sigma$ bijectively onto
$\XX_{\sigma^{-1}}$ and $\YY_{\sigma^{-1}}$, resp.\ onto
$\XX_{\rc\sigma}$ and $\YY_{\rc\sigma}$.

Note that $\mm_\sigma$ is regular and $a_i\le b_j$ for all $i,j$ and in particular there are no juxtaposed segments, so that $\XX_\sigma \subseteq \YY_\sigma$.
Thus, $\Ff = 0$ and $\Vv =\bmdl$ in the notation of \Cref{sec:esi}.

\begin{remark}\label{rem:reggen}
Conversely, if $\mm=\sum_{i\in \Kk}[a_i,b_i]$ is a regular multisegment such that $a_i\le b_j$ for all $i,j$, then
$\dfr_*(\mm)=\dfr_*(\mm_\sigma)$ where upon indexing $\Kk = \{1,\dots,n\}$ such that $a_1 < \dots < a_n$,
$\sigma\in S_n$ is the permutation such that $\sigma_i > \sigma_j \iff b_i < b_j$.
This follows from \Cref{cor:relinv} since $\XX_\mm = \XX_\sigma$ and $\YY_\mm =\YY_\sigma$.
\end{remark}

\begin{definition}\label{def:prigperm}
A permutation $\sigma \in S_n$ is \emph{\prig} (resp.\ \emph{rigid}) if
the multisegment $\mm_\sigma$ is \prig (resp.\ rigid);
this is the \prigity (resp.\ rigidity) of the irreducible component
$\cmp_\sigma := \cmp_{\mm_\sigma}$.
\end{definition}

By \eqref{eq:pairsym} and \Cref{cor:relinv},
\begin{equation} \label{quot:klein}
\begin{gathered}
\text{\Prigity and rigidity of permutations are invariant}\\\text{under reverse-complement and inversion.}
\end{gathered}
\end{equation}

This definition is geometric in nature; \Cref{thm:smoothrigid,thm:mainA} will characterize both notions purely combinatorially.

We can write the rank problem of \Cref{ssec:typeA} in terms of permutations.
For $\mm = \mm_\sigma$ the matrix of $\Sop_{\lambda,\mu}$ in the
bases of \eqref{eq:S} has its rows indexed by $\XX_\sigma$ and its
columns by $\YY_\sigma$; we denote it by $\Sop^\sigma_{\lambda,\mu}$ and
$\Sop^\sigma_\lambda = \Sop^\sigma_{\lambda,\lambda}$. Its entries are
described by the colored graph $\Gc_\sigma$ of \Cref{def:M}. Spelled out
in terms of $\sigma$, its edges are
\begin{align*}
\text{\emph{of type $\blam$}}\quad & \bigl((i,k),(i,j)\bigr)
   \quad\text{for } i \le j < k \text{ with } \sigma_i \ge \sigma_j >
   \sigma_k, \quad\text{colored } \blam_{jk};\\
\text{\emph{of type $\bmu$}}\quad & \bigl((i,k),(j,k)\bigr)
   \quad\text{for } i < j \le k \text{ with } \sigma_i > \sigma_j \ge
   \sigma_k, \quad\text{colored } \bmu_{ij};
\end{align*}
an edge of type $\blam$ shortens the inversion $(i,k)$ at its right end
and an edge of type $\bmu$ at its left end.

The relabelings \eqref{eq:pairsym} transport the colors as well, interchanging the two types:
\begin{equation}\label{eq:colorsym}
\begin{aligned}
   \text{under inversion,} \quad \blam_{ij} &\mapsto \bmu_{\sigma_j\sigma_i},
   & \bmu_{ij} &\mapsto \blam_{\sigma_j\sigma_i},\\
   \text{under rotation,} \quad \blam_{ij} &\mapsto \bmu_{\wo(j)\wo(i)},
   & \bmu_{ij} &\mapsto \blam_{\wo(j)\wo(i)} .
\end{aligned}
\end{equation}
Hence they are isomorphisms of colored graphs, carrying $\Gc_\sigma$ to
$\Gc_{\sigma^{-1}}$, resp.\ $\Gc_{\rc\sigma}$. Both relabelings reverse the
pairs, so it is the second half of \Cref{cor:relinv} that applies; this is
also why they interchange the two types of colors.

The \emph{value} of $\blam_{ab}$, resp.\ $\bmu_{ab}$, at a
pair $(\lambda,\mu) \in \bmdl\times\bmdl$ is the parameter $\lambda_{ab}$,
resp.\ $\mu_{ab}$, of \eqref{eq:S}; giving a value to every color is the
same as giving the pair $(\lambda,\mu)$.

The nonzero entry $(u,y)$ of $\Sop^\sigma_{\lambda,\mu}$,
i.e., the coefficient of $\bmdlbss_u$ in $\Sop_{\lambda,\mu}(\algbss_y)$,
corresponds to an edge $u \to y$ in $\Gc_\sigma$.
It is the value of the color of that edge if the edge is of type $\blam$ and minus that value if it is of type $\bmu$.

The definition \eqref{def:pair} gives

\begin{corollary}\label{cor:psGS}
Let $\sigma \in S_n$. Then $\sigma$ is \prig if and only if
$\Sop^\sigma_{\lambda,\mu}$ is surjective for some $(\lambda,\mu)$, and rigid
if and only if $\Sop^\sigma_\lambda$ is surjective for some $\lambda$.
Equivalently, $\sigma$ is \prig, resp.\ rigid, if and only if
$\Sop^\sigma_{\lambda,\mu}$, resp.\ $\Sop^\sigma_\lambda$, has row rank
$|\XX_\sigma|$ over the field of rational functions in the $2|\XX_\sigma|$,
resp.\ $|\XX_\sigma|$, parameters, taken as indeterminates.
\end{corollary}

%\subsection{Heredity for permutations}\label{ssec:heredperm}

For the multisegments $\mm_\sigma$, \emph{every} partition of the index set is endpoint-separated.
From \Cref{thm:esi} and \Cref{rem:reggen} we conclude:

\begin{theorem}[heredity]\label{thm:hered}
Let $\tau \preceq \sigma$. If $\sigma$ is rigid/\prig then so is $\tau$.
In other words, the class of \prig permutations and the class of rigid permutations are closed under pattern
containment.
\end{theorem}

Equivalently, in the contrapositive form in which we shall use it: if $\tau
\preceq \sigma$ and $\dfr_\ast(\mm_\tau) > 0$, then
$\dfr_\ast(\mm_\sigma) > 0$.

\subsection{\Smth permutations and rigidity}\label{ssec:smoothrigid}

Before defining the class that governs \prigity we settle the analogous
question for rigidity, where the answer coincides with the smoothness criterion of Lakshmibai--Sandhya.
This was proved by Lapid--M\'{\i}nguez in \cite{MR3866895} but we give here a shorter argument with a sharper statement.
The implication rigid $\Rightarrow$ \smth is an immediate consequence of \Cref{thm:hered},
and hence of \Cref{thm:esi}; no hypothesis on $\kk$ is needed.

Recall that a permutation is \emph{\smth} if it avoids $3412$ and $4231$.
By \cite{MR1051089}, these are exactly the permutations for which the corresponding Schubert variety is smooth.
We will not use this fact directly.

\begin{definition}[\cores]\label{def:cores}
An \emph{\SWcr \core} of $\sigma$ is a four-point set $\Cm$ with the
following properties.
\begin{enumerate}
\item\label{core:corners} The set $\Cm$ contains the \Wcr corner and the \Scr
corner of $\sigma$.
\item\label{core:pattern} The pattern of $\Cm$ is $3412$ or $4231$.
\item\label{core:ascent} The window of the \Wcr corner and the window of
the \Scr corner each contain an ascent pair of points of $\Cm$.
\end{enumerate}
An \emph{\NEcr \core} $\Cp$ is defined dually, that is, as a set whose
image under $\rc$ is an \SWcr \core of $\rc\sigma$. We write $\Cpm$
when both are meant.
\end{definition}

A permutation may have several \SWcr \cores: $34512$ has two, $\{1,2,4,5\}$ and
$\{1,3,4,5\}$. In the situation of \Cref{prop:leaves}\labelcref{lv:attach} below
they turn out to be unique.

The crucial combinatorial input is the following elementary fact, which is essentially \cite{MR2450335}*{Lemma 17}.
Since the requirement \labelcref{core:ascent} is not part of the statement there, we include the short proof.
\begin{lemma}[blocked \SWcr corners]\label{lem:swcore}
A permutation whose \Wcr corner and \Scr corner are both blocked has
an \SWcr \core. Similarly, one whose \Ncr and \Ecr corners are both blocked has an \NEcr \core.
\end{lemma}

\begin{proof}
The second statement follows from the first by applying $\rc$.
Let $W$ and $S$ be the \Wcr and \Scr corners, let $R$ be the set of points strictly between them in both position and value,
let $X$ be the set of points after $S$ and below $W$, and $Y$ the set of points before $S$ and above $W$.
Then $\win(W) = R \cup \{S\} \cup X$ and $\win(S) = R \cup \{W\} \cup Y$.
If $R$ contains an ascent pair $k < l$, then $\{W,k,l,S\}$ has pattern $4231$, and $k,l$ lie in both windows.
Otherwise $R \cup \{S\}$ and $R \cup \{W\}$ are decreasing, so $X$ and $Y$ are nonempty since neither window is decreasing.
For $x \in X$ and $y \in Y$ the set $\{W,y,S,x\}$ has pattern $3412$, and $S,x$ and $W,y$ are ascent pairs in $\win(W)$ and $\win(S)$ respectively.
\end{proof}

Since a \good corner cannot belong to an occurrence of $3412$ or of $4231$ we infer:
\begin{corollary}[\smthness recursively]\label{prop:smoothrec}
A nonempty permutation $\sigma$ is \smth if and only if it has a
\good corner $\crn$ (which can be taken to be the \Wcr or the \Scr corner) such that $\sigma - \crn$ is \smth.
\end{corollary}
This is \cite{MR2450335}*{Proposition 18}, up to applying $\rc$: the two alternatives
there single out the \Ncr and the \Ecr corner.
For the Schubert-theoretic context --- maximal singular loci, the singularities
governing them, rational smoothness, and further combinatorics of \smth
permutations --- see \cites{MR1990570,MR2422304,MR1278700,MR4290619}.

The minimal inversions are the only coordinates that matter.
By \Cref{lem:top} the set of generators of $\bmdl$ is
\begin{equation}\label{eq:Dcirc}
   \bmdl^{\gen} \;=\; \{\lambda \in \bmdl :
      \lambda_{ab} \ne 0 \text{ for all } (a,b) \in \Cov(\sigma)\} ;
\end{equation}
it is open and nonempty and $G$-stable.
By \eqref{eq:opengen}, or directly by \Cref{cor:gencov}, an open $G$-orbit in
$\bmdl$, if there is one, is contained in $\bmdl^{\gen}$. The next proposition
shows that for \smth $\sigma$ it is all of $\bmdl^{\gen}$.

\begin{proposition}[the open orbit]\label{prop:covparam}
Let $\sigma$ be \smth. Then $G$ acts transitively on $\bmdl^{\gen}$,
which is therefore the open $G$-orbit in $\bmdl$.
In particular $\Sop^\sigma_\lambda$ is surjective for every $\lambda \in
\bmdl^{\gen}$, and $\sigma$ is rigid.
\end{proposition}

\begin{proof}
We use induction on $|\sigma|$; the case $\XX_\sigma = \emptyset$ is trivial.

By \Cref{prop:smoothrec} and symmetry we may assume that the \Wcr corner
$\crn$ of $\sigma$ is \good and that $\tau = \sigma - \crn$ is \smth. Put
$a = \sigma_1$, let $j_k$ be given by $\sigma_{j_k} = k$, so that $j_a = 1$,
and put $u_k = (1,j_k)$ for $1 \le k < a$; these are the inversions of
$\sigma$ not inherited from $\tau$. Put $e = \algbss_{11}$ and $e' = 1-e$. Since
$(i,1) \in \YY_\sigma$ only for $i=1$, and $(i,1) \notin \XX_\sigma$,
\begin{equation}\label{eq:Wsplit}
   \alg e = \kk e, \qquad \bmdl e = 0 .
\end{equation}

Thus $e$ satisfies the axioms \eqref{eq:cornerax} and we may apply
\Cref{lem:cornersplit}. The quotient pair $(\alg',\bmdl')$ of
\eqref{eq:cornerpieces} is the pair attached to $\tau$, so that the
proposition for $\tau$ asserts that $G'$ is transitive on
$(\bmdl')^{\gen}$. If $a = 1$ then
$e\bmdl = 0$, so $p$ is bijective and we are done by induction; assume
therefore $a>1$. It remains to check the two hypotheses of
\Cref{lem:cornersplit}\labelcref{cs:transitive}, namely that $\dim
e\bmdl/e\rbm(\bmdl) = 1$ and that $\mathcal N\lambda = e\rbm(\bmdl)$.

Since $\sigma_1 = a$, the pairs in $\XX_\sigma$
with first entry $1$ are exactly the $u_k$, $1 \le k < a$; as $\YY_\sigma
= \XX_\sigma\sqcup\DD_\sigma$ by \eqref{eq:dictUV}, both $e\bmdl$ and $\mathcal N =
e\alg e'$ have bases $\bmdlbss_{u_k}$, resp.\ $\algbss_{u_k}$, $1 \le k <
a$. \Goodness of $\crn$ makes the points $j_a,\dots,j_1$ a decreasing
chain, so that $(j_k,j_l) \in \XX_\sigma$ precisely for $l<k$, and
\begin{equation}\label{eq:Ztri}
   \algbss_{u_k}\lambda = \sum_{l<k}\lambda_{j_kj_l}\,\bmdlbss_{u_l} .
\end{equation}
Fix $\lambda \in \bmdl^{\gen}$. The inversions $(j_k,j_{k-1})$ carry
consecutive values, hence are minimal, so that $\lambda_{j_kj_{k-1}} \ne
0$; by \eqref{eq:Ztri} the element $\algbss_{u_{i+1}}\lambda$ is
$\lambda_{j_{i+1}j_i}\bmdlbss_{u_i}$ modulo $\Span\{\bmdlbss_{u_l} :
l<i\}$, so that $\bmdlbss_{u_i} \in \mathcal N\lambda$ for every $i \le a-2$, by
induction on $i$.
Thus $\mathcal N\lambda$ contains the hyperplane $\Span\{\bmdlbss_{u_l} : l \le
a-2\}$ of $e\bmdl$. On the other hand $\mathcal N = e\mathcal N \subseteq e\rad(\alg)$ gives
$\mathcal N\lambda \subseteq e\rbm(\bmdl)$, while $e\rbm(\bmdl) \ne e\bmdl$, since
$u_{a-1}$ carries the consecutive values $a,a-1$, hence is minimal, so
that $\bmdlbss_{u_{a-1}} \notin \rbm(\bmdl)$ by \Cref{lem:top}. Therefore
$\mathcal N\lambda = e\rbm(\bmdl)$ is that hyperplane: both hypotheses hold.

Let $\lambda,\lambda' \in \bmdl^{\gen}$. By
induction $G'$ is transitive on $(\bmdl')^{\gen}$, and $G \to G'$ is
onto, so after replacing $\lambda'$ by a $G$-translate --- which stays in
$\bmdl^{\gen}$ --- we may assume $p(\lambda') = p(\lambda)$;
then $\lambda'\in H\cdot\lambda$ by
\Cref{lem:cornersplit}\labelcref{cs:transitive}.
\end{proof}

\begin{remark}[the Borel picture -- cf.\ \cite{MR3866895}*{Remark 4.11 and Example 4.12}]\label{rem:borel}
For $\sigma = \wo$ we have $\alg = \mathfrak{b}$, the upper triangular
matrices, $G = B$, $\bmdl = \mathfrak{n}$ its nilpotent radical, and
$\rbm(\bmdl) = [\mathfrak{n},\mathfrak{n}]$ is spanned by the root
vectors above the first superdiagonal, so that $\Cov(\wo) =
\{(i,i+1)\}$: the generators of the bimodule are the regular nilpotent
elements of $\mathfrak{n}$, and the proposition is the classical fact that
they form a dense $B$-orbit.
\end{remark}

\begin{theorem}[\smth $=$ rigid]\label{thm:smoothrigid}
A permutation is rigid if and only if it is \smth.
In this case, we can choose over any field a parameter $\lambda$ with coordinates in $\{0,1\}$
(or even the all-ones parameter) such that $\Sop^\sigma_\lambda$ is surjective. 
\end{theorem}

\begin{proof}
A \smth permutation is rigid by \Cref{prop:covparam}. For the second assertion,
let $\kk_0$ be any field and let $\lambda$ be the parameter all of whose coordinates equal $1$. It lies in
$\bmdl^{\gen}$ by \eqref{eq:Dcirc}, so $\Sop^\sigma_\lambda$ is surjective over
an algebraic closure of $\kk_0$ by \Cref{prop:covparam}, hence over $\kk_0$ itself.

It remains to see that a rigid permutation is \smth. A permutation that is
not \smth contains $3412$ or $4231$, so by \Cref{thm:hered} it suffices
that these two be not rigid. We have
\begin{align*}
   \Cov(3412) &= \{(1,3),(1,4),(2,3),(2,4)\},\\
   \Cov(4231) &= \{(1,2),(1,3),(2,4),(3,4)\} ,
\end{align*}
forming a single cycle in the cover graph, so that $\cyc(\cgr_\sigma) = 1$ for these $\sigma$'s
and \Cref{lem:countmm} gives $\dfr_\rsub \ge1$.
\end{proof}

\subsection{\Psm permutations}\label{ssec:PA}

\emph{Plateaus, quotients, inflations.} A \emph{plateau pair} is a pair
of adjacent positions carrying consecutive decreasing values, that is
$\sigma_{i+1} = \sigma_i - 1$; a \emph{plateau} is a maximal run of
consecutive positions carrying consecutive decreasing values, and
$\sigma$ is \emph{plateau-free} if every plateau is a single point. A
plateau occupies simultaneously an interval of positions and an interval
of values. Writing $D_\ell$ for the decreasing permutation of length
$\ell$, the (descending-interval) \emph{inflation}
$q[\ell] = q[D_{\ell_1},\dots,D_{\ell_k}]$ of $q \in S_k$ by
$\ell=(\ell_1,\dots,\ell_k)$ with $\ell_i \ge 1$ for all $i$ replaces the
$i$-th entry of $q$ by a block of $\ell_i$ consecutive values in
decreasing order, the blocks being placed left to right and their
value-intervals ordered according to $q$. Contracting each plateau of $\sigma$ to a
point gives the \emph{quotient} $q(\sigma)$, which is plateau-free,
and $\sigma$ is the inflation of $q(\sigma)$ by the plateau lengths.
This is the unique way to write $\sigma$ as an inflation of a plateau-free permutation.

The following properties are standard.
\begin{lemma} \label{lem:eleminf}
Let $\sigma=\tau[\ell]$.
\begin{enumerate}
\item\label{part:cutinf} The cuts of $\sigma$ are the images of the cuts
of $\tau$; in particular $\sigma$ is indecomposable if and only if $\tau$
is, and if $\tau = \alpha\oplus\beta$ then $\sigma$ is the direct sum of
inflations of $\alpha$ and $\beta$.
\item A corner of $\sigma$ is an endpoint of the block of the corresponding corner of $\tau$ --- the first
point of the block for the \Wcr and \Ncr corners, the last
point for the \Scr and \Ecr corners --- and it is \good if and
only if the corresponding corner of $\tau$ is.
In particular \cbness is invariant under descending-interval inflation.
\item An inflation of $\sigma$ is also an inflation of $\tau$.
\item \label{part:deleteinf} If $c$ is a point in $\sigma$, in the block $d$ of $\tau$, then $\sigma-c$ is either an inflation
of $\tau$ with $\ell_d$ decreased by one if $\ell_d>1$, or the inflation of $\tau-d$ if $\ell_d=1$.
\item\label{part:covinf} The minimal inversions of $\sigma$ are of two
kinds: the pairs of adjacent positions inside a single block, which make
each block span a path of $\cgr_\sigma$, and, for each minimal inversion
of $\tau$, the single pair joining the last point of the earlier block to
the first point of the later one. Thus $\cgr_\sigma$ is obtained from
$\cgr_\tau$ by expanding each vertex $i$ into a path on $\ell_i$
vertices, namely its block, and reattaching the edges formerly at $i$ to
the two ends of that path: those leading to a later block at its last
vertex, those leading to an earlier one at its first, the two ends
coinciding when $\ell_i = 1$. 
Conversely, $\cgr_\tau$ is obtained from $\cgr_\sigma$ by contracting each block path to a vertex.
In particular,
\begin{equation}\label{eq:covinf}
   |\Cov(\sigma)| = |\Cov(\tau)| + \bigl(|\sigma| - |\tau|\bigr) ,
   \qquad
   \cyc(\cgr_\sigma) = \cyc(\cgr_\tau) .
\end{equation}
\end{enumerate}
\end{lemma}

\begin{definition}[the eight tower quotients]\label{def:towers}
The \emph{tower quotients} are the eight involutions
\begin{equation}\label{eq:eight}
   \{3412\},\quad \{4231\},\quad \{35142,\,42513\},\quad \{45312\},\quad
   \{426153\},\quad \{463152,\,526413\},
\end{equation}
grouped into their orbits under $\rc$. A \emph{tower} is an inflation $q[\ell]$ of one of these eight.
\end{definition}

Each of the eight quotients is indecomposable and \cb, hence so is every tower, by \Cref{lem:eleminf}.
They are also plateau-free, so that the quotient $q(t)$ of a tower $t$ is one of the eight and
the presentation $t=q(t)[\ell]$ is unique.
Their cover graphs are drawn in \Cref{fig:towers}.

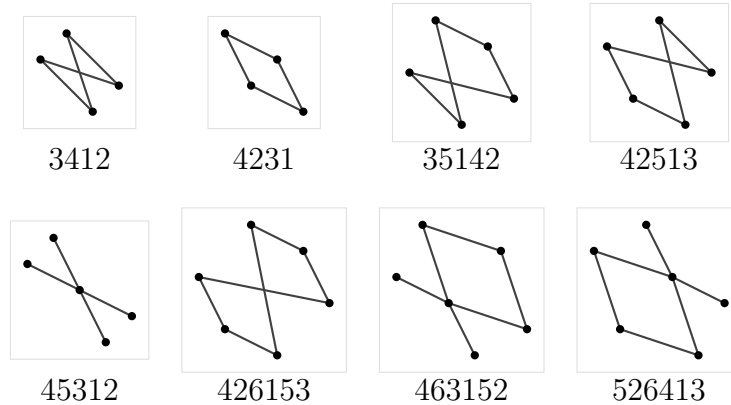
\begin{figure}[!htbp]
\centering
\begin{tabular}{@{}cccc@{}}
\begin{tikzpicture}[scale=0.345,baseline=(current bounding box.center)]
\draw[gray!25] (0.35,0.35) rectangle (4.65,4.65);
\draw[covedge] (1,3)--(3,1);
\draw[covedge] (1,3)--(4,2);
\draw[covedge] (2,4)--(3,1);
\draw[covedge] (2,4)--(4,2);
\node[permdot] at (1,3) {};
\node[permdot] at (2,4) {};
\node[permdot] at (3,1) {};
\node[permdot] at (4,2) {};
\end{tikzpicture} &
\begin{tikzpicture}[scale=0.345,baseline=(current bounding box.center)]
\draw[gray!25] (0.35,0.35) rectangle (4.65,4.65);
\draw[covedge] (1,4)--(2,2);
\draw[covedge] (1,4)--(3,3);
\draw[covedge] (2,2)--(4,1);
\draw[covedge] (3,3)--(4,1);
\node[permdot] at (1,4) {};
\node[permdot] at (2,2) {};
\node[permdot] at (3,3) {};
\node[permdot] at (4,1) {};
\end{tikzpicture} &
\begin{tikzpicture}[scale=0.345,baseline=(current bounding box.center)]
\draw[gray!25] (0.35,0.35) rectangle (5.65,5.65);
\draw[covedge] (1,3)--(3,1);
\draw[covedge] (1,3)--(5,2);
\draw[covedge] (2,5)--(3,1);
\draw[covedge] (2,5)--(4,4);
\draw[covedge] (4,4)--(5,2);
\node[permdot] at (1,3) {};
\node[permdot] at (2,5) {};
\node[permdot] at (3,1) {};
\node[permdot] at (4,4) {};
\node[permdot] at (5,2) {};
\end{tikzpicture} &
\begin{tikzpicture}[scale=0.345,baseline=(current bounding box.center)]
\draw[gray!25] (0.35,0.35) rectangle (5.65,5.65);
\draw[covedge] (1,4)--(2,2);
\draw[covedge] (1,4)--(5,3);
\draw[covedge] (2,2)--(4,1);
\draw[covedge] (3,5)--(4,1);
\draw[covedge] (3,5)--(5,3);
\node[permdot] at (1,4) {};
\node[permdot] at (2,2) {};
\node[permdot] at (3,5) {};
\node[permdot] at (4,1) {};
\node[permdot] at (5,3) {};
\end{tikzpicture}\\
$3412$ & $4231$ & $35142$ & $42513$\\[10pt]
\begin{tikzpicture}[scale=0.345,baseline=(current bounding box.center)]
\draw[gray!25] (0.35,0.35) rectangle (5.65,5.65);
\draw[covedge] (1,4)--(3,3);
\draw[covedge] (2,5)--(3,3);
\draw[covedge] (3,3)--(4,1);
\draw[covedge] (3,3)--(5,2);
\node[permdot] at (1,4) {};
\node[permdot] at (2,5) {};
\node[permdot] at (3,3) {};
\node[permdot] at (4,1) {};
\node[permdot] at (5,2) {};
\end{tikzpicture} &
\begin{tikzpicture}[scale=0.345,baseline=(current bounding box.center)]
\draw[gray!25] (0.35,0.35) rectangle (6.65,6.65);
\draw[covedge] (1,4)--(2,2);
\draw[covedge] (1,4)--(6,3);
\draw[covedge] (2,2)--(4,1);
\draw[covedge] (3,6)--(4,1);
\draw[covedge] (3,6)--(5,5);
\draw[covedge] (5,5)--(6,3);
\node[permdot] at (1,4) {};
\node[permdot] at (2,2) {};
\node[permdot] at (3,6) {};
\node[permdot] at (4,1) {};
\node[permdot] at (5,5) {};
\node[permdot] at (6,3) {};
\end{tikzpicture} &
\begin{tikzpicture}[scale=0.345,baseline=(current bounding box.center)]
\draw[gray!25] (0.35,0.35) rectangle (6.65,6.65);
\draw[covedge] (1,4)--(3,3);
\draw[covedge] (2,6)--(3,3);
\draw[covedge] (2,6)--(5,5);
\draw[covedge] (3,3)--(4,1);
\draw[covedge] (3,3)--(6,2);
\draw[covedge] (5,5)--(6,2);
\node[permdot] at (1,4) {};
\node[permdot] at (2,6) {};
\node[permdot] at (3,3) {};
\node[permdot] at (4,1) {};
\node[permdot] at (5,5) {};
\node[permdot] at (6,2) {};
\end{tikzpicture} &
\begin{tikzpicture}[scale=0.345,baseline=(current bounding box.center)]
\draw[gray!25] (0.35,0.35) rectangle (6.65,6.65);
\draw[covedge] (1,5)--(2,2);
\draw[covedge] (1,5)--(4,4);
\draw[covedge] (2,2)--(5,1);
\draw[covedge] (3,6)--(4,4);
\draw[covedge] (4,4)--(5,1);
\draw[covedge] (4,4)--(6,3);
\node[permdot] at (1,5) {};
\node[permdot] at (2,2) {};
\node[permdot] at (3,6) {};
\node[permdot] at (4,4) {};
\node[permdot] at (5,1) {};
\node[permdot] at (6,3) {};
\end{tikzpicture}\\
$45312$ & $426153$ & $463152$ & $526413$
\end{tabular}
\caption{The eight tower quotients \eqref{eq:eight}. Each is drawn as the plot
of its points --- positions to the east, values to the north --- with the
minimal inversions of \Cref{def:covmm} joined by segments, so that the picture
is the cover graph $\cgr_\tau$ laid over the plot. In the order of
\eqref{eq:eight} the edge counts are $4,4,5,5,4,6,6,6$ against $4,4,5,5,5,6,6,6$
vertices, and each graph is connected; so each $\cgr_\tau$ is a pseudotree.
Inflating an entry expands its vertex into a path
(\Cref{lem:eleminf}\labelcref{part:covinf}), changing neither the components nor
the cycle rank, whence \Cref{prop:PApf} for towers.}
\label{fig:towers}
\end{figure}

Now we define the class of \psm permutations recursively.
\begin{definition}[\psm permutations]\label{def:PA}
The identity permutation in $S_1$ is \psm.
An indecomposable $\sigma$ with $|\sigma| > 1$ is \psm if it is either a tower or has a \good corner $\crn$ with $\sigma - \crn$ \psm.
(See \Cref{lem:cornerindec}.)

In general, a permutation is \psm if all its indecomposable summands are;
in particular the empty permutation is \psm. 
\end{definition}
We write $\PA$ for the class of \psm permutations.

The terminology is motivated by \Cref{thm:smoothrigid}: the rigid permutations are exactly the \smth
ones, while by \Cref{thm:mainA} the \prig ones are exactly the members of
$\PA$; so \psm stands to \smth as \prig stands to rigid.

\begin{remark} \label{rem:psm45}
Every permutation of size $\le4$ is \psm. (The only non-smooth ones are $4231$ and $3412$.)
Of the 32 non-smooth permutations in $S_5$ the ones which are not \psm are $\{34512,\,45123\}, \{45231,\,53412\}, \{52341\}$.
In particular, any permutation contained in one of the eight tower quotients is \psm.
\end{remark}

The \Wcr and \Scr corners of $\sigma$ lie in its first indecomposable
summand and the \Ncr and \Ecr corners in its last, and in each case the
window is contained in that summand; so a \good corner of $\sigma$ is a
corner of the same type, with the same window, of the indecomposable
summand $\gamma$ of $\sigma$ containing it, and it is \good there. By
\Cref{lem:cornerindec} the indecomposable summands of $\sigma - \crn$ are
those of $\sigma$ with $\gamma$ replaced by $\gamma - \crn$ (or removed, if $|\gamma|=1$). Hence
\begin{equation}\label{eq:goodloc}
\text{if $\crn$ is a \good corner of $\sigma$ and $\sigma-\crn \in \PA$,
then $\sigma \in \PA$.}
\end{equation}
\begin{equation}\label{eq:PAhered}
\text{The class $\PA$ is closed under inflation and pattern containment.}
\end{equation}
For inflation this follows by induction using \Cref{lem:eleminf}.
For pattern containment, it suffices to show that $\sigma - d \in \PA$ for $\sigma \in \PA$ and
every point $d$; we induct on $|\sigma|$ and write $\tau = \sigma - d$.
The summands of $\tau$ are those of $\sigma$ with the one containing $d$,
say $\gamma$, replaced by the summands of $\gamma - d$; so we may assume
that $\sigma$ is indecomposable, and that $|\sigma| > 1$.

If $\sigma$ has a \good corner $\crn$ with $\sigma - \crn \in \PA$: for $d = \crn$ there
is nothing to prove. Otherwise, $\crn$ is still the corresponding corner of $\tau$, and
it is still \good since its window there is its window in $\sigma$ with $d$ removed.
Moreover $\tau - \crn = (\sigma - \crn) - d \in \PA$ by induction, so that
$\tau\in\PA$ by \eqref{eq:goodloc}.

If $\sigma$ is a tower, say $\sigma=q[\ell]$, then by \Cref{lem:eleminf}\labelcref{part:deleteinf} $\tau$ is either an
inflation of $q$ or an inflation of $q-d'$ for some point $d'$ of $q$. In the former case $\tau$ is
again a tower, hence in $\PA$. In the latter case $q-d'$ is contained in $q$ and
therefore lies in $\PA$ by \Cref{rem:psm45}, so that $\tau \in\PA$ by the closure
of $\PA$ under inflation.

Consequently the recursion of \Cref{def:PA} involves no choice, clarifying
a possible ambiguity in the definition:
\begin{equation}\label{eq:nochoice}
\text{if $\crn$ is a \good corner of $\sigma$, then $\sigma\in\PA$ if and
only if $\sigma - \crn \in \PA$.}
\end{equation}
Indeed, one direction is \eqref{eq:goodloc} and the other is
\eqref{eq:PAhered}.

\begin{remark}[deciding \psmness]\label{rem:algorithm}
This gives rise to an efficient algorithm to decide \psmness:
\eqref{eq:nochoice} turns \Cref{def:PA} into a deterministic procedure.
First we reduce to the indecomposable case by splitting $\sigma$ into its indecomposable summands.
In the indecomposable case, we either delete a \good corner $\crn$ if there is one, applying the procedure
for $\sigma-\crn$ --- which is again indecomposable, by
\Cref{lem:cornerindec}, so that no further splitting is needed --- or if
there is none test whether $\sigma$ is a tower.
It decides membership in $\PA$ in $O(n^2)$ steps, since there are at most $n$ deletions and each is
surrounded by $O(n)$ work --- finding the summands from the prefix maxima,
locating the four corners, testing that their windows are decreasing and,
in the terminal case, contracting the plateaus and comparing the quotient
with \eqref{eq:eight}.
\end{remark}

The recursion also bounds the cover graph, by a purely combinatorial argument.

\begin{proposition}[the cover graph of a \psm permutation]\label{prop:PApf}
Let $\sigma \in \PA$. Then $\cgr_\sigma$ is a pseudoforest: every
connected component of it contains at most one cycle.
\end{proposition}

\begin{proof}[Proof of \Cref{prop:PApf}]
If $\sigma = \gamma_1 \oplus \dots \oplus \gamma_k$ then the graph $\cgr_\sigma$ is the disjoint union of the
$\cgr_{\gamma_i}$. Therefore, we may assume that $\sigma$ is indecomposable.

If $\sigma$ is a tower, say $\sigma = q[\ell]$ with $q$ one of the eight
quotients of \Cref{def:towers}, then by
\Cref{lem:eleminf}\labelcref{part:covinf} the graph $\cgr_\sigma$ is a
blow-up of $\cgr_q$ that changes neither the components nor the cycle
rank; so it is enough that each of the eight $\cgr_q$ be a pseudotree,
and indeed each is connected with at most as many edges as vertices,
having, in the order of \eqref{eq:eight} (see \Cref{fig:towers}), respectively $4,4,5,5,4,6,6,6$ minimal
inversions against sizes $4,4,5,5,5,6,6,6$.

Otherwise either $|\sigma|=1$ or $\sigma$ has a \good corner $\crn$ with $\tau = \sigma - \crn
\in \PA$. By \eqref{eq:gooddegree} the graph $\cgr_\sigma$ is
$\cgr_\tau$ with a new vertex added, isolated or a leaf, which creates no
cycle. We can therefore argue by induction.
\end{proof}

By \Cref{lem:countmm} this is the combinatorial shadow of \prigity, and
\Cref{thm:mainA} below makes the two agree; but the proposition is a
statement about permutations alone, proved from \Cref{def:PA} alone.
\Cref{thm:covercount} is its converse at the minimal obstructions (see \Cref{rem:PApattern}).

\subsection{The characterization}\label{ssec:mainthm}

\begin{theorem}\label{thm:mainA}
Let $\sigma \in S_n$. Then $\sigma$ is \prig if and only if $\sigma$ is \psm.
In this case, we can choose over any field parameters $\lambda,\mu$ with coordinates in $\{0,1\}$ such that
$\Sop^\sigma_{\lambda,\mu}$ is surjective.
\end{theorem}

The two implications are proved in \Cref{sec:forward} and
\Cref{sec:converse} respectively. The converse direction runs through the
permutations that are minimal outside $\PA$, called \emph{critical}: it
is shown in \Cref{sec:converse} that each of them carries a counting
obstruction to \prigity, and \Cref{thm:hered} then propagates the
obstruction to every permutation containing one.

\section{\Psm implies \prig}\label{sec:forward}

Throughout this section $\sigma \in S_n$, and $\XX_\sigma$, $\YY_\sigma$,
$\Gc_\sigma$ and $\Sop^\sigma_{\lambda,\mu}$ are as in
\Cref{ssec:dictionary}. By \Cref{cor:psGS} it suffices to exhibit, for
every \psm $\sigma$, a single value of $(\lambda,\mu)$ at which
$\Sop^\sigma_{\lambda,\mu}$ is surjective. We do so by
a combinatorial mechanism --- peeling --- which is stable under the
clauses of \Cref{def:PA}.

\subsection{Peeling}\label{ssec:peeling}

For any subset $\pset$ of colors of $\sigma$ (\Cref{def:M}) we denote by
$\Gc^\pset_\sigma$ the subgraph of $\Gc_\sigma$ consisting of the edges whose color lies in $\pset$. 

\begin{definition}[peeling]\label{def:peel}
A \emph{peeling datum} for $\pset$ consists of a partial order $\le$ on
$\XX_\sigma$ and a map $h : \XX_\sigma \to \YY_\sigma$ such that, in
$\Gc^\pset_\sigma$,
\begin{equation}\label{eq:peelorder}
   x \to h(x) \text{ is an edge for every } x, \text{ and }
   x' \to h(x) \text{ is one only if } x' \le x .
\end{equation}
A color set admitting a peeling datum is called a \emph{\pal}, and
$\sigma$ is \emph{\spc} if it has a \pal.
\end{definition}

By \eqref{eq:pairsym} and \eqref{eq:colorsym}, inversion and rotation carry
a \pal for $\sigma$ to a \pal for $\sigma^{-1}$, resp.\ for $\rc\sigma$;
so \spcness is invariant under the Klein four-group.

We can think of $h$ as a preferred matching in the graph $\Gc^\pset_\sigma$.
It is automatically injective: if $h(x) = h(x')$ then
\eqref{eq:peelorder} gives $x \le x'$ and $x' \le x$, and $\le$ is
antisymmetric. This is one of the places where it matters that $\le$ is a
partial order and not merely a preorder.
Given $h$ (and $\pset$), the existence of a partial order satisfying
\eqref{eq:peelorder} is equivalent to the acyclicity of the \emph{strict}
dependency digraph on $\XX_\sigma$,
\[
   x' \longrightarrow_D x
   \quad\Longleftrightarrow\quad
   x' \ne x \ \text{ and }\ x' \to h(x)\text{ is an edge in }\Gc^\pset_\sigma .
\]
The clause $x'\ne x$ cannot be omitted: $x \to h(x)$ is an edge for every
$x$, so without it the relation would have a loop at every vertex and
would never be acyclic.
Any refinement of $\le$ also satisfies \eqref{eq:peelorder}, so we may take
$\le$ to be a total order.

In the totally ordered case, writing $\XX_\sigma =
\{x_1,\dots,x_N\}$ with $x_1 < \dots < x_N$ and $y_i = h(x_i)$, the
condition reads: $x_i \to y_i$ is an edge and $x_j \to y_i$ is not for
$j > i$. These are exactly the successful runs of the process that gives
the name: as long as some $y \in \YY_\sigma$ has exactly one surviving
neighbor in $\Gc^\pset_\sigma$, delete that neighbor. Indeed, in a run that
deletes all of $\XX_\sigma$, take for $x_i$ the $i$-th inversion deleted
and for $y_i$ a vertex witnessing its deletion; conversely, along such an
order the $x_i$'s are deleted one by one.

\begin{proposition}[soundness of peeling]\label{prop:peel}
Let $\pset$ be a \pal for $\sigma$. Then $\Sop^\sigma_{\lambda,\mu}$
is surjective at every $(\lambda,\mu)$ giving a nonzero value to each color
in $\pset$ and the value zero to every other color. Every \spc permutation is
\prig, over every field.
\end{proposition}

\begin{proof}
Fix a peeling datum $(\le,h)$ for $\pset$, refine $\le$ to a total order
$x_1 < \dots < x_N$ of $\XX_\sigma$ and put $y_i = h(x_i)$; these are
distinct. For $(\lambda,\mu)$ as in the statement, the square submatrix of $\Sop^\sigma_{\lambda,\mu}$
with rows $x_1,\dots,x_N$ and columns $y_1,\dots,y_N$ is upper triangular with nonzero
diagonal entries, hence invertible. Hence, $\Sop^\sigma_{\lambda,\mu}$ is surjective and by \Cref{cor:psGS} $\sigma$ is \prig.
\end{proof}

Thus a \pal furnishes a pair $(\lambda,\mu)$ of parameters
--- $\lambda$ supported on the colors of $\pset$ of type $\blam$ and $\mu$ on
those of type $\bmu$ --- at which the rank problem of \Cref{def:pair}
attains its maximum. Note that, by \Cref{cor:gencov} every \pal contains,
for each minimal inversion $(i,k)$, at least one of the two colors $\blam_{ik}$,
$\bmu_{ik}$. The remainder of the section proves that every \psm
permutation admits a \pal.

\subsection{The frame}\label{ssec:frame}

A peeling datum by itself does not propagate along the recursion of
\Cref{def:PA}: the inductive step at a \good corner needs to know that
certain colors are available in the smaller permutation. Fortunately, it is
possible to amend the situation by requiring extra assumptions.

For a value $v$ let $p_v = \sigma^{-1}(v)$ be its position. Let $r$ be
the last index of the maximal decreasing prefix, $\sigma_1 > \sigma_2 >
\dots > \sigma_r$, and $s$ the first index of the maximal decreasing
suffix, $\sigma_s > \dots > \sigma_n$. Let $a$ be the largest element
of $\{1,\dots,n\}$ with $p_a < p_{a-1} < \dots < p_1$, and $b$ the
smallest element of $\{1,\dots,n\}$ with $p_n < p_{n-1} < \dots <
p_b$; both conditions are vacuous for $a = 1$, resp.\ $b = n$.

\begin{definition}[the frame]\label{def:frame}
The \emph{frame} of $\sigma$ is the (possibly empty) color set
\begin{equation}\label{eq:frame}
\begin{aligned}
   \Fr(\sigma) = {}& \{\bmu_{i,i+1} : 1 \le i < r\}
   \;\cup\; \{\blam_{i,i+1} : s \le i < n\}\\
   &\cup\; \{\blam_{p_vp_{v-1}} : 2 \le v \le a\}
   \;\cup\; \{\bmu_{p_vp_{v-1}} : b < v \le n\} ,
\end{aligned}
\end{equation}
the four families corresponding to the decreasing prefix, the decreasing
suffix, the bottom-value chain and the top-value chain. A \pal containing
$\Fr(\sigma)$ is a \emph{\fpal}, and the permutation $\sigma$ is
\emph{frame-\spc} if it has one.
\end{definition}

The frame is compatible with inversion and rotation, whose effect on pairs
is \eqref{eq:pairsym} and on colors \eqref{eq:colorsym}.
Since a \pal is carried to a \pal (\Cref{ssec:peeling}),
\begin{equation}\label{eq:framesym}
   \sigma \text{ frame-\spc} \iff \sigma^{-1} \text{ frame-\spc}
   \iff \rc\sigma \text{ frame-\spc} .
\end{equation}

Every frame-\spc permutation is \spc. We prove that every \psm permutation is frame-\spc.

\subsection{Towers}\label{ssec:towers}

Throughout this subsection $\tau \in S_k$ is a base permutation, and
$\XX_\tau$, $\YY_\tau$, $\DD_\tau$, $\Gc_\tau$, the types and the
convention on neighbors are those of \Cref{def:M} and \eqref{eq:dictUV}.

\begin{definition}[tower datum]\label{def:towerdatum}
Let $\pset_\tau$ be a set of colors of $\tau$. A \emph{tower datum} for
$\pset_\tau$ consists of a partial order $\le$ on $\YY_\tau$ and a map $h_\tau :
\XX_\tau \to \YY_\tau$ such that $x \to h_\tau(x)$ is an edge of
$\Gc^{\pset_\tau}_\tau$ for every $x$, and such that the following hold.
\begin{enumerate}
\renewcommand{\theenumi}{T\arabic{enumi}}%
\renewcommand{\labelenumi}{(\theenumi)}%
\item\label{T:cover} Every $x \in \XX_\tau$ is the unique element of $\YY_\tau$ covered by $h_\tau(x)$.
Thus, $\YY_\tau$ being finite, $z < h_\tau(x)$ implies $z \le x$ for every
$z \in \YY_\tau$.
\item\label{T:onto} Every element of $\DD_\tau$ lies in the image of $h_\tau$.
\item\label{T:type} Whenever $h_\tau(x) \in \XX_\tau$, the edges $x \to h_\tau(x)$
and $h_\tau(x) \to h_\tau(h_\tau(x))$ have the same type.
\item\label{T:mono} In $\Gc^{\pset_\tau}_\tau$, every edge $w \to v$ satisfies $w < v$.
\end{enumerate}
\end{definition}

Condition \labelcref{T:cover} has the following consequences.
First, the map $h_\tau$ is injective, and for $x \in \XX_\tau$
the iterates $x, h_\tau(x), h_\tau^2(x),\dots$ --- the \emph{orbit} of $x$ --- form
a chain in $\le$, which must end at a diagonal vertex. The maximal orbits are pairwise disjoint, and
together with \labelcref{T:onto}, they partition $\YY_\tau$ into precisely $k$ chains, one for
each element of $\DD_\tau$ (as a terminal vertex of the chain).
Let $\orbs$ be the set of maximal orbits, let $\orb(y) \in \orbs$ be the maximal
orbit containing $y \in \YY_\tau$ and let $\pstn(y) \in \NN$ be the position of $y$ along that orbit.
Define a partial order on $\orbs$ by comparing the initial vertices of the orbits. Then, again by \labelcref{T:cover}, the map
\begin{equation}\label{eq:orbcoord}
   \YY_\tau \longrightarrow \orbs\times\NN,
   \qquad
   y \longmapsto \bigl(\orb(y),\,\pstn(y)\bigr) ,
\end{equation}
is injective and order preserving with respect to the lexicographic order
on $\orbs\times\NN$. Indeed, let $y' < y$. If the two lie in one orbit
the assertion is clear; otherwise, applying \labelcref{T:cover} repeatedly
along the orbit of $y$ gives $y' < u$, where $u$ is the initial vertex of
$\orb(y)$, while the initial vertex of $\orb(y')$ is $\le y'$.

Finally, by \labelcref{T:type}, all edges along an orbit
are of the same type, and we assign that type to every vertex of the
orbit, the terminal diagonal vertex included. We write
\[
   \typ : \YY_\tau \longrightarrow \{\blam,\bmu\}
\]
for the resulting function.

A tower datum is in particular a peeling datum for $\pset_\tau$ in the sense
of \Cref{def:peel}, with $\le$ restricted to $\XX_\tau$: if $w \to h_\tau(x)$
is an edge then $w < h_\tau(x)$ by \labelcref{T:mono}, hence $w \le x$ by
\labelcref{T:cover}, which is \eqref{eq:peelorder}. Thus $\pset_\tau$ is a
\pal; what the four conditions add is the shape of the peeling, and
that is what lifts to the inflations.

Now let $\sigma = \tau[\ell_1,\dots,\ell_k] \in S_n$, with $n = \ell_1 +
\dots + \ell_k$, be an inflation of $\tau$. The block replacing the $i$-th
entry of $\tau$ occupies an interval of $\ell_i$ consecutive positions,
and we write
\[
   \prj : \{1,\dots,n\} \longrightarrow \{1,\dots,k\}
\]
for the resulting projection, so that $B_i = \prj^{-1}(i)$ is the $i$-th
block. Let
\[
   \lmin,\ \lmax : \{1,\dots,k\} \longrightarrow \{1,\dots,n\}
\]
be the two sections of $\prj$ sending $i$ to the first, resp.\ the
last, position of $B_i$; thus $B_i = [\lmin(i),\lmax(i)]$ and $\lmax(i) -
\lmin(i) + 1 = \ell_i$.

Each block carries consecutive values in decreasing order, and the value
intervals are ordered according to $\tau$. Hence, for $p < q$, the pair
$(p,q)$ is an inversion of $\sigma$ if and only if $\prj(p) = \prj(q)$ or
$(\prj(p),\prj(q)) \in \XX_\tau$. Consequently $\prj\times\prj$ restricts
to a surjection
\[
   \prjY : \YY_\sigma \longrightarrow \YY_\tau,
   \qquad
   \prjY(i,j) = \bigl(\prj(i),\prj(j)\bigr) ,
\]
whose fibers are
\begin{equation}\label{eq:fiber}
\begin{aligned}
   \prjY^{-1}(i,j) &= B_i\times B_j \subseteq \XX_\sigma,
   && (i,j) \in \XX_\tau,\\
   \prjY^{-1}(c,c) &= \{(p,q) : p,q \in B_c,\ p \le q\},
   && (c,c) \in \DD_\tau ,
\end{aligned}
\end{equation}
of cardinalities $\ell_i\ell_j$ and $\binom{\ell_c+1}{2}$; the second
fiber meets $\DD_\sigma$ in the $\ell_c$ diagonal vertices of the block
and $\XX_\sigma$ in the remaining $\binom{\ell_c}{2}$ pairs. We call $x
\in \XX_\sigma$ \emph{external} if $\prjY(x) \in \XX_\tau$, and
\emph{internal} to the block $B_c$ if $\prjY(x) = (c,c)$; in both cases we
set $\typ(x) = \typ(\prjY(x))$, the \emph{type} of $x$.

The \emph{corner lift} is the (partial) section
\begin{equation}\label{eq:cornerlift}
   \cl : \XX_\tau \longrightarrow \XX_\sigma,
   \qquad
   \cl(i,j) = \bigl(\lmax(i),\,\lmin(j)\bigr)
\end{equation}
of $\prjY$: it selects, out of the fiber $B_i\times B_j$, the inversion
joining the last position of $B_i$ to the first position of $B_j$. It
lifts colors as well,
\begin{equation}\label{eq:colorlift}
   \cl(\blam_{ij}) = \blam_{\cl(i,j)},
   \qquad
   \cl(\bmu_{ij}) = \bmu_{\cl(i,j)},
   \qquad (i,j) \in \XX_\tau ,
\end{equation}
a color of $\tau$ lifting to a color of $\sigma$ of the same type. Lift
the color set by
\begin{equation}\label{eq:liftedA}
   \pset_\sigma \;=\; \cl(\pset_\tau) \;\cup\;
   \bigl\{\blam_{p,p+1},\ \bmu_{p,p+1} \;:\; \prj(p) = \prj(p+1)\bigr\} :
\end{equation}
each selected base color is taken on its corner lift, and both colors are
taken on every adjacent inversion inside a block.

We now write down a peeling datum for $\pset_\sigma$. Define $h_\sigma :
\XX_\sigma \to \YY_\sigma$ on an inversion $x = (p,q)$ of $\sigma$ by
\begin{equation}\label{eq:hmu}
   h_\sigma(x) =
   \begin{cases}
      (p+1,q), & p \ne \lmax(\prj(p)),\\
      (\lmin(a'),q), & p = \lmax(\prj(p)),\ h_\tau(\prjY(x)) = (a',c),
   \end{cases}
   \qquad \typ(x) = \bmu ,
\end{equation}
and by the mirror image
\begin{equation}\label{eq:hlambda}
   h_\sigma(x) =
   \begin{cases}
      (p,q-1), & q \ne \lmin(\prj(q)),\\
      (p,\lmax(b')), & q = \lmin(\prj(q)),\ h_\tau(\prjY(x)) = (c,b'),
   \end{cases}
   \qquad \typ(x) = \blam .
\end{equation}
The second alternative makes sense. It occurs only for external $x$: an
internal $x = (p,q)$ has $p < q$ inside one block, so that $p \ne
\lmax(\prj(p))$ and $q \ne \lmin(\prj(q))$. For external $x$ the vertex
$\prjY(x)$ lies in the domain of $h_\tau$, and the edge $\prjY(x) \to
h_\tau(\prjY(x))$ is of type $\typ(x)$, so by \eqref{eq:type12} its target has
the displayed shape --- the same second coordinate for the type $\bmu$,
the same first coordinate for the type $\blam$. In both alternatives $x
\to h_\sigma(x)$ is an edge of $\Gc^{\pset_\sigma}_\sigma$: in the first one,
which covers all internal inversions, the edge carries an adjacent
internal color, and in the second one a lifted color.

Finally we order the inversions of $\sigma$. Put
\[
   \delta(x) =
   \begin{cases}
      -q, & \typ(x) = \bmu,\\
      p, & \typ(x) = \blam,
   \end{cases}
   \qquad x = (p,q) \in \XX_\sigma ,
\]
and consider the \emph{key}
\begin{equation}\label{eq:towerorder}
   \kappa : \XX_\sigma \longrightarrow \orbs\times\ZZ\times\NN,
   \qquad
   \kappa(x) =
   \bigl(\orb(\prjY(x)),\ \delta(x),\ \pstn(\prjY(x))\bigr) ,
\end{equation}
with the target ordered lexicographically:
base orbits in their order; inside a base orbit, columns from right to
left, resp.\ rows from top to bottom; and only then the base vertices
along the orbit. This is the order \eqref{eq:orbcoord} of $\prjY(x)$, with
$\delta(x)$ inserted between its two coordinates. Note that $\delta$ is compared only inside a
base orbit, where all vertices have one and the same type.

The key $\kappa$ need not be injective. Its first and last coordinates
depend on $x$ only through $\prjY(x)$, while over a base inversion of type
$\blam$ the middle one is $\delta(p,q) = p$, which ignores $q$; so any two
inversions in one fiber $B_i\times B_j$ with the same first coordinate have
the same key. For instance, for the inflation of $3412$ by descending
blocks of length two, namely $\sigma = 65872143$, the distinct inversions
$(4,5)$ and $(4,6)$ lie over the base inversion $(2,3)$, which is of type
$\blam$ by \Cref{tab:certs}, and $\kappa(4,5) = \kappa(4,6)$. Pulling the
lexicographic order back along $\kappa$ would therefore give only a
preorder, with no antisymmetry. We order $\XX_\sigma$ instead by
\begin{equation}\label{eq:towerpo}
   x' \le x \quad\Longleftrightarrow\quad
   x' = x \quad\text{or}\quad \kappa(x') < \kappa(x) ,
\end{equation}
declaring distinct inversions with equal keys incomparable. This is a
partial order: reflexivity is built in, antisymmetry holds because
$\kappa(x') < \kappa(x)$ and $\kappa(x) < \kappa(x')$ cannot both occur,
and transitivity follows from that of the lexicographic order, the case of
an equality being trivial.

\begin{lemma}[block lifting]\label{lem:blocklift}
Suppose that $\pset_\tau$ admits a tower datum. Then the set $\pset_\sigma$ of
\eqref{eq:liftedA} is a \pal. Moreover, if $\pset_\tau$ is a \fpal for
$\tau$ then $\pset_\sigma$ is a \fpal for $\sigma$.
\end{lemma}

\begin{proof}
We check \eqref{eq:peelorder} for $h_\sigma$ and the order
\eqref{eq:towerpo}.
Take $x = (p,q)$ of type $\bmu$, the type $\blam$ being the mirror image,
and write $h_\sigma(x) = (P,q)$: both alternatives of \eqref{eq:hmu} keep
the second coordinate $q$ of $x$, and
\begin{equation}\label{eq:baseofh}
   \prjY(h_\sigma(x)) =
   \begin{cases}
      \prjY(x) & \text{in the first alternative},\\
      h_\tau(\prjY(x)) & \text{in the second}.
   \end{cases}
\end{equation}
Suppose that $x'\to h_\sigma(x)$ is an edge in $\Gc_\sigma$ with color
$\clr\in \pset_\sigma$. By \eqref{eq:liftedA} the color $\clr$ is either
adjacent internal or lifted, and by \eqref{eq:type12} it is of type
$\bmu$ with second index $P$, or of type $\blam$ with first index $q$.
There are thus three cases to consider.

If $\clr$ is an adjacent internal color of type $\bmu$, then $\clr=\bmu_{P-1,P}$, so that
$\prj(P-1) = \prj(P)$ and $x' = (P-1,q)$. In the second alternative of
\eqref{eq:hmu} we have $P = \lmin(\prj(P))$ and no such color exists; in
the first, $P = p+1$ and $x' = x$.

If $\clr$ is a lifted color $\cl(\bmu_{ia'})$ of type $\bmu$, with $\bmu_{ia'} \in
\pset_\tau$, then $P = \lmin(a')$, which forces the second alternative of \eqref{eq:hmu}.
In this case $\prjY(h_\sigma(x)) = h_\tau(\prjY(x)) = (a',c)$ with $c
= \prj(q)$, and $x' = (\lmax(i),q)$ has $\prjY(x') = (i,c)$, a neighbor of
$(a',c)$. By \labelcref{T:mono}, $\prjY(x') < (a',c)$, and
\labelcref{T:cover} turns this into $\prjY(x') \le \prjY(x)$. If equality holds
then $x'=x$ since $x$ is the only vertex of the fiber of $\prjY(x)$ joined to
$h_\sigma(x)$. Otherwise, $\prjY(x') < \prjY(x)$, whence either
$\orb(\prjY(x')) < \orb(\prjY(x))$, or the two base orbits coincide and
then $\pstn(\prjY(x')) < \pstn(\prjY(x))$ while $\delta(x') = \delta(x) =
-q$. In both cases $\kappa(x') < \kappa(x)$.

The last case is where $\clr$ is of type $\blam$, so that $x' = (P,b)$ with $b
> q$. If $\prjY(x')$ lies in the orbit of $\prjY(x)$ --- as it does when $\clr$
is an adjacent internal color, for which $\prjY(x') = \prjY(h_\sigma(x))$ (using
\eqref{eq:baseofh}) --- then $x'$ is of type $\bmu$ and $\delta(x') = -b <
-q = \delta(x)$, so $\kappa(x') < \kappa(x)$. Otherwise $\clr$ is a lifted color and
$\prjY(x')$ is a neighbor of $\prjY(h_\sigma(x))$; by \labelcref{T:mono},
and by \labelcref{T:cover} in the second alternative of
\eqref{eq:baseofh}, this gives $\prjY(x') \le \prjY(x)$, so
$\orb(\prjY(x')) \le \orb(\prjY(x))$, with strict inequality since the two
orbits differ. Again $\kappa(x') < \kappa(x)$.

Thus in every case either $x' = x$ or $\kappa(x') < \kappa(x)$; that is,
$x' \le x$ in the order \eqref{eq:towerpo}. Note that the argument never
compares two inversions with equal keys, which is what allows
\eqref{eq:towerpo} to leave them incomparable. Hence $\pset_\sigma$ is a
\pal.

Suppose now that $\Fr(\tau) \subseteq \pset_\tau$. A frame color of $\sigma$
whose two positions lie in one block is an adjacent internal color, and
those are selected in both types. A decreasing prefix or suffix of
$\sigma$ can cross a block boundary only if the corresponding adjacent
entries of $\tau$ decrease, and a bottom- or top-value chain can pass from
one value block to the next only if the corresponding consecutive values
of $\tau$ form the analogous chain; in either case the crossing frame
color is the lift \eqref{eq:colorlift} of the corresponding color of
$\Fr(\tau)$, which lies in $\pset_\tau$ by assumption. Hence $\Fr(\sigma)
\subseteq \pset_\sigma$.
\end{proof}

It remains to produce a tower datum for each of the eight tower
quotients; this is what \Cref{tab:certs,tab:frames} do, and \Cref{fig:peel}
illustrates the case $\tau = 4231$. For a base
permutation $\tau$ we write $ij$ for the inversion $(i,j)$, and
$\pset_{\bmu}(\tau)$, $\pset_{\blam}(\tau)$ for the sets of inversions
carrying a selected color of type $\bmu$, resp.\ $\blam$; the \pal
$\pset_\tau$ is determined by the two of them. \Cref{tab:certs} gives the maximal orbits
of $h_\tau$, the symbol in brackets recording the type of the orbit:
$h_\tau$ sends a vertex to its successor in its orbit, and $\le$ is the
total order in which the vertices are listed --- orbits in the order
given, each read from its start. \Cref{tab:frames} gives the two sets of
inversions and the frame \eqref{eq:frame}; in each of its rows
$\pset_{\bmu}(\tau)$ and $\pset_{\blam}(\tau)$ consist of the inversions used by
the orbits together with those occurring in the frame, and of nothing
else.

\begin{table}[!htbp]
\centering
\renewcommand{\arraystretch}{1.25}
\begin{tabular}{@{}l@{\qquad}>{\raggedright\arraybackslash}p{0.76\textwidth}@{}}
\hline
$\tau$ & maximal orbits of $h_\tau$, in increasing order\\
\hline
$3412$ &
$[\bmu]\,13\to(3,3)$;\quad $[\blam]\,14\to(1,1)$;\quad
$[\blam]\,23\to(2,2)$;\quad $[\bmu]\,24\to(4,4)$\\
$4231$ &
$[\blam]\,14\to12\to(1,1)$;\quad $[\blam]\,24\to(2,2)$;\quad
$[\bmu]\,34\to(4,4)$;\quad $[\bmu]\,13\to(3,3)$\\
$35142$ &
$[\bmu]\,13\to(3,3)$;\quad $[\blam]\,15\to(1,1)$;\quad
$[\blam]\,23\to(2,2)$;\quad $[\bmu]\,25\to45\to(5,5)$;\quad
$[\bmu]\,24\to(4,4)$\\
$42513$ &
$[\bmu]\,14\to24\to(4,4)$;\quad $[\bmu]\,12\to(2,2)$;\quad
$[\blam]\,15\to(1,1)$;\quad $[\blam]\,34\to(3,3)$;\quad
$[\bmu]\,35\to(5,5)$\\
$45312$ &
$[\blam]\,15\to(1,1)$;\quad $[\blam]\,24\to23\to(2,2)$;\quad
$[\bmu]\,25\to35\to(5,5)$;\quad $[\bmu]\,14\to34\to(4,4)$;\quad
$[\bmu]\,13\to(3,3)$\\
$426153$ &
$[\bmu]\,14\to24\to(4,4)$;\quad $[\bmu]\,12\to(2,2)$;\quad
$[\blam]\,16\to(1,1)$;\quad $[\blam]\,34\to(3,3)$;\quad
$[\bmu]\,36\to56\to(6,6)$;\quad $[\bmu]\,35\to(5,5)$\\
$463152$ &
$[\bmu]\,14\to34\to(4,4)$;\quad $[\bmu]\,26\to56\to(6,6)$;\quad
$[\bmu]\,25\to(5,5)$;\quad $[\blam]\,16\to13\to(1,1)$;\quad
$[\blam]\,24\to23\to(2,2)$;\quad $[\blam]\,36\to(3,3)$\\
$526413$ &
$[\blam]\,15\to12\to(1,1)$;\quad $[\blam]\,25\to(2,2)$;\quad
$[\blam]\,36\to34\to(3,3)$;\quad $[\bmu]\,35\to45\to(5,5)$;\quad
$[\bmu]\,16\to46\to(6,6)$;\quad $[\bmu]\,14\to(4,4)$\\
\hline
\end{tabular}
\caption{Tower data for the eight tower quotients \eqref{eq:eight}: the
maximal orbits of $h_\tau$, listed in the order that defines $\le$.}
\label{tab:certs}
\end{table}

\begin{table}[!htbp]
\centering
\renewcommand{\arraystretch}{1.25}
\begin{tabular}{@{}l@{\qquad}l@{\qquad}l@{\qquad}l@{}}
\hline
$\tau$ & $\pset_{\bmu}(\tau)$ & $\pset_{\blam}(\tau)$ & $\Fr(\tau)$\\
\hline
$3412$   & $\{13,24\}$          & $\{14,23\}$             & $\varnothing$\\
$4231$   & $\{12,13,34\}$       & $\{12,24,34\}$          &
   $\{\bmu_{12},\bmu_{13},\blam_{24},\blam_{34}\}$\\
$35142$  & $\{13,24,45\}$       & $\{15,23,45\}$          &
   $\{\bmu_{24},\blam_{45}\}$\\
$42513$  & $\{12,24,35\}$       & $\{15,24,34\}$          &
   $\{\bmu_{12},\blam_{24}\}$\\
$45312$  & $\{13,23,34,35\}$    & $\{15,23,34\}$          & $\varnothing$\\
$426153$ & $\{12,24,35,56\}$    & $\{16,24,34,56\}$       &
   $\{\bmu_{12},\bmu_{35},\blam_{24},\blam_{56}\}$\\
$463152$ & $\{13,25,34,56\}$    & $\{13,23,34,36,56\}$    &
   $\{\bmu_{25},\blam_{56}\}$\\
$526413$ & $\{12,14,34,45,46\}$ & $\{12,25,34,46\}$       &
   $\{\bmu_{12},\blam_{25}\}$\\
\hline
\end{tabular}
\caption{The selected colors and the frame.}
\label{tab:frames}
\end{table}

\begin{lemma}[the eight tower data]\label{lem:eightdata}
Let $\tau$ be one of the eight tower quotients \eqref{eq:eight}, let
$\pset_\tau$ and $\Fr(\tau)$ be as in \Cref{tab:frames}, and let $h_\tau$ and
$\le$ be read off \Cref{tab:certs} as above. Then $(\le,h_\tau)$ is a
tower datum for $\pset_\tau$ in the sense of \Cref{def:towerdatum}, and
$\Fr(\tau) \subseteq \pset_\tau$.
\end{lemma}

\begin{proof}
Conditions \labelcref{T:cover,T:onto,T:type} are visible from
\Cref{tab:certs}: the orbits partition $\YY_\tau$, each of them ends at a
diagonal vertex, and each carries a single type. That every displayed
arrow is an edge of the stated color, and that \labelcref{T:mono} holds,
is a direct application of \eqref{eq:type12}: a finite check on
permutations of length at most six. The last assertion is read off
\Cref{tab:frames}.
\end{proof}

\begin{figure}[!htbp]
\centering
\begin{tikzpicture}[scale=0.42,baseline=(current bounding box.center)]
\draw[gray!25] (0.35,0.35) rectangle (4.65,4.65);
\draw[covedge] (1,4)--(2,2);
\draw[covedge] (1,4)--(3,3);
\draw[covedge] (2,2)--(4,1);
\draw[covedge] (3,3)--(4,1);
\node[permdot] at (1,4) {};
\node[permdot] at (2,2) {};
\node[permdot] at (3,3) {};
\node[permdot] at (4,1) {};
\end{tikzpicture}
\hspace{1.1em}
\begin{tikzpicture}[baseline=(current bounding box.center),
   every node/.style={font=\footnotesize,inner sep=1.5pt},
   yscale=1.30]
\node (a1) at (0,3) {$14$};   \node (a2) at (2.1,3) {$12$};  \node (a3) at (4.2,3) {$(1,1)$};
\node (b1) at (0,2) {$24$};   \node (b2) at (2.1,2) {$(2,2)$};
\node (c1) at (0,1) {$34$};   \node (c2) at (2.1,1) {$(4,4)$};
\node (d1) at (0,0) {$13$};   \node (d2) at (2.1,0) {$(3,3)$};
\draw[peelarr] (a1)--node[above,font=\scriptsize]{$\blam_{24}$}(a2);
\draw[peelarr] (a2)--node[above,font=\scriptsize]{$\blam_{12}$}(a3);
\draw[peelarr] (b1)--node[above,font=\scriptsize]{$\blam_{24}$}(b2);
\draw[peelarr] (c1)--node[above,font=\scriptsize]{$\bmu_{34}$}(c2);
\draw[peelarr] (d1)--node[above,font=\scriptsize]{$\bmu_{13}$}(d2);
\draw[auxarr] (a2)--node[right,font=\scriptsize]{$\bmu_{12}$}(b2);
\draw[auxarr] (c1)--node[sloped,above,font=\scriptsize,pos=0.45]{$\blam_{34}$}(d2);
\draw[auxarr] (a1) to[bend right=25]
   node[right=1pt,font=\scriptsize,pos=0.6]{$\bmu_{12}$} (b1);
\draw[auxarr] (a1) to[bend right=45]
   node[left=1pt,font=\scriptsize,pos=0.5]{$\bmu_{13}$} (c1);
\draw[auxarr] (a1) to[bend right=60]
   node[left=1pt,font=\scriptsize,pos=0.5]{$\blam_{34}$} (d1);
\end{tikzpicture}
\hspace{1.1em}
$\displaystyle
\begin{pmatrix}
\blam_{24} & 0 & 0 & 0 & 0\\
0 & \blam_{12} & -\bmu_{12} & 0 & 0\\
0 & 0 & \blam_{24} & 0 & 0\\
0 & 0 & 0 & -\bmu_{34} & \blam_{34}\\
0 & 0 & 0 & 0 & -\bmu_{13}
\end{pmatrix}$
\caption{Peeling for $\tau = 4231$, with the \fpal $\pset_\tau$ of
\Cref{tab:frames}: the colors $\bmu_{12},\bmu_{13},\bmu_{34}$ and
$\blam_{12},\blam_{24},\blam_{34}$. Left: the plot of $\tau$ with its cover
graph. Middle: the graph $\Gc^{\pset_\tau}_\tau$, its rows being the four
maximal orbits of $h_\tau$ listed in \Cref{tab:certs}; solid arrows are
$x\to h_\tau(x)$, dashed arrows the remaining edges. Right: the square
submatrix of $\Sop^\tau_{\lambda,\mu}$ on the rows $14,12,24,34,13$ --- the
order $\le$ of \Cref{tab:certs} restricted to $\XX_\tau$ --- and the columns
$h_\tau$ of those rows. It is upper triangular with nonzero diagonal, which is
the proof of \Cref{prop:peel}; a color of type $\bmu$ enters with a minus sign.}
\label{fig:peel}
\end{figure}

\begin{proposition}\label{prop:C0}
Every tower is frame-\spc.
\end{proposition}

\begin{proof}
Let $\sigma = \tau[\ell]$ with $\tau$ one of the eight quotients
\eqref{eq:eight}, and let $\pset_\sigma$ be the lift \eqref{eq:liftedA} of the
color set $\pset_\tau$ of \Cref{lem:eightdata}. By that lemma and
\Cref{lem:blocklift}, $\pset_\sigma$ is a \fpal for $\sigma$.
\end{proof}

\subsection{Direct sums and \good corners}\label{ssec:forwardsteps}

\begin{proposition}\label{prop:directsum}
If $\alpha$ and $\beta$ are frame-\spc, so is $\alpha\oplus\beta$.
\end{proposition}

\begin{proof}
There are no inversions between the two summands, so after shifting the
indices of the second, $\Gc_\sigma = \Gc_\alpha \sqcup \Gc_\beta$ for
$\sigma = \alpha\oplus\beta$, and the union $\pset_\sigma = \pset_\alpha \sqcup
\pset_\beta$ of two \pals is again a \pal: take the two maps $h$ together and
the disjoint union of the two partial orders, no target of one summand
being joined to an inversion of the other. No
frame chain crosses the sum boundary --- the boundary is an ascent in
position order, the bottom values lie in $\alpha$ and the top values in
$\beta$ --- so $\Fr(\sigma) \subseteq \Fr(\alpha)\sqcup\Fr(\beta)
\subseteq \pset_\sigma$.
\end{proof}

\begin{proposition}[\good-corner extension]\label{prop:corner}
Let $\sigma$ be indecomposable with $|\sigma|>1$ and let $\crn$ be a \good corner of $\sigma$.
If $\sigma - \crn$ is frame-\spc, then so is $\sigma$.
\end{proposition}

\begin{proof}
Apply the unique symmetry of \eqref{eq:cornersym} carrying $\crn$ to the \Wcr
corner; \goodness is preserved by \Cref{ssec:permconv} and
frame-\spcness by
\eqref{eq:framesym}, and the color set and the peeling datum produced
below are transported back by the same symmetry. We may therefore assume
that $\crn$ is the \Wcr corner. Write $\tau = \sigma - \crn$, which is
frame-\spc by hypothesis.

Put $a = \sigma_1$; then $a > 1$, since $a = 1$ would give $\sigma = 1
\oplus \tau$. The window of the first entry consists of the positions of
the values $a-1,a-2,\dots,1$, and \goodness means that these occur in that
order: writing $j_k$ for the position of the value $k$, this says that
$1 = j_a < j_{a-1} < \dots < j_1$. The inversions of $\sigma$ not
inherited from $\tau$ are exactly $u_k = (1,j_k)$, $1 \le k < a$.

Choose a \fpal $\pset_\tau$ for $\tau$, and identify the
positions and colors of $\tau$ with the corresponding ones of $\sigma$
(this only shifts the labels). Since the values $a-1,\dots,1$ occur in
decreasing order in $\tau$, its bottom-value frame contains the colors
$\blam_{j_{k+1}j_k}$ for $1 \le k \le a-2$. Put
\begin{equation}\label{eq:cornerA}
   \pset_\sigma = \pset_\tau \cup \{\blam_{1j_{a-1}}\} \cup E,
   \qquad
   E = \begin{cases}
      \{\bmu_{12}\}, & \sigma_1 > \sigma_2,\\
      \varnothing, & \sigma_1 < \sigma_2;
   \end{cases}
\end{equation}
note that if $\sigma_1 > \sigma_2$ then the first member of the window
sits at position $2$, that is $j_{a-1} = 2$.

Order the inversions of $\sigma$ by putting the new ones first, among
themselves $u_1 < u_2 < \dots < u_{a-1}$, and above them the
inversions of $\tau$ with a peeling datum of $\pset_\tau$; take for $h$ on
the latter the map of that datum. For $k \le a-2$ put $h(u_k) =
(1,j_{k+1})$: the color $\blam_{j_{k+1}j_k}$ supplies the edge $(1,j_k)
\to (1,j_{k+1})$ by \eqref{eq:type12}. The target of an edge of type
$\bmu$ cannot have first coordinate $1$, so every inversion adjacent to
$(1,j_{k+1})$ is of the
form $(1,m)$ with $m > j_{k+1}$ and $\sigma_m < \sigma_{j_{k+1}}$, hence
$m \in \{j_k,\dots,j_1\}$; those are $u_k$ itself and
$u_{k-1},\dots,u_1$, which are smaller. Put $h(u_{a-1}) = (1,1) =
(1,j_a)$, using
$\blam_{1j_{a-1}}$; no old color has first coordinate $1$, so $(1,1)$ has
no other neighbor at all.

It remains to check \eqref{eq:peelorder} at the targets inherited from
$\tau$. By \eqref{eq:type12} the added color $\blam_{1j_{a-1}}$ colors only
the edge $u_{a-1} \to (1,1)$, while $\bmu_{12}$, when present, colors the
edges $(1,c) \to (2,c)$ for $(2,c) \in \YY_\sigma$; in both cases the source has
first coordinate $1$ and is therefore one of the new inversions. Hence
every edge of $\Gc^{\pset_\sigma}_\sigma$ whose source is an old inversion
already has its color in $\pset_\tau$, so a target inherited from $\tau$
keeps exactly the neighbors it had in $\Gc^{\pset_\tau}_\tau$, apart from
new inversions that $\bmu_{12}$ may attach to it. The latter precede all
old inversions in the order, and so \eqref{eq:peelorder} holds at those
targets as well.

For the frame: the only frame colors created by the new entry are the
bottom-value color $\blam_{1j_{a-1}}$, joining the value $a$ to the value
$a-1$, and, when the first two entries form a descent, the prefix color
$\bmu_{12}$; all others are inherited from $\tau$. Hence $\Fr(\sigma)
\subseteq \Fr(\tau)\cup\{\blam_{1j_{a-1}}\}\cup E \subseteq \pset_\sigma$.
\end{proof}

\subsection{The forward theorem}\label{ssec:forwardthm}

\begin{theorem}\label{thm:forward}
Every \psm permutation is frame-\spc; in particular every \psm
permutation is \spc, hence \prig. Moreover, for \psm $\sigma$ and any field
$\kk_0$, the map $\Sop^\sigma_{\lambda,\mu}$ is surjective at a parameter
with coordinates $0$ and $1$.
\end{theorem}

\begin{proof}
Induction along \Cref{def:PA}. The base cases, the empty permutation and $|\sigma|=1$, are trivial.
If $\sigma$ is decomposable, its indecomposable summands are frame-\spc by induction and
\Cref{prop:directsum} applies.
If $\sigma$ is indecomposable with $|\sigma|>1$ and has a \good corner $\crn$, then $\sigma -
\crn$ is frame-\spc by induction and \Cref{prop:corner} applies. If
$\sigma$ is a tower, \Cref{prop:C0} applies.
\Prigity then follows by \Cref{prop:peel}.

For the last assertion, note that nothing in \Cref{def:peel} or in the
constructions of \Crefrange{ssec:frame}{ssec:forwardsteps} involves $\kk_0$
beyond the requirement that the colors in the \pal be given nonzero values,
so we may give them all the value $1$; by \Cref{prop:peel} the resulting
square submatrix of $\Sop^\sigma_{\lambda,\mu}$ is then triangular with
diagonal entries $\pm1$, hence invertible over every field.
\end{proof}

The proof is algorithmic: it constructs a \pal and a peeling
datum, by recursing on the first applicable clause of \Cref{def:PA} ---
splitting at the first sum cut, deleting the first \good corner in the
order \Wcr, \Ncr, \Scr, \Ecr,
or, at a tower, writing it as the inflation of its quotient and lifting the tower datum of \Cref{tab:certs}. By \eqref{eq:nochoice} these choices are
immaterial: the recursion terminates on a \psm permutation whichever
\good corner is deleted at each step, and it has depth at most $n$ (\Cref{rem:algorithm}).

\section{\Prig implies \psm}\label{sec:converse}

In this section we prove the converse part of \Cref{thm:mainA}: a \prig permutation is \psm
(\Cref{cor:converse}).
We say that a permutation is \emph{critical} if it is not \psm but minimally so ---
any point deletion of it is \psm (\Cref{def:crit}).
By \Cref{thm:hered} it is enough to show that a critical permutation $\sigma$ is not \prig.
To that end we will use the criterion of \Cref{lem:countmm}, namely we will show that
the cover graph $\cgr_\sigma$ has at least
$|\sigma|+1$ edges (\Cref{thm:covercount}). In fact, this lower bound is an
equality and for $|\sigma|\ge9$ the graph $\cgr_\sigma$ is a dumbbell:
two four-cycles joined by a path (\Cref{cor:dumbbell}). This will be done by examining the \emph{inversion graph} of a critical permutation and showing that its two leaf blocks
carry the two \cores of \Cref{def:cores} which in turn carry the four-cycles.

Recall that the cover graph $\cgr_\sigma = \cgr_{\mm_\sigma}$ of
\Cref{def:covmm} has vertex set $\{1,\dots,n\}$, the positions of
$\sigma$, and edge set $\Cov(\sigma)=\Cov(\mm_\sigma)$, the minimal inversions of $\sigma$.
The latter are exactly the transpositions realizing the covers
$\sigma \gtrdot \sigma\cdot(i\,k)$ of the Bruhat order.
By \Cref{lem:countmm} a necessary condition for $\sigma$ to be \prig is that $\cgr_\sigma$ is a pseudoforest.

In a follow-up paper we will classify the critical permutations completely and provide more combinatorial information about them;
in particular it will turn out that, unlike in the smooth case, there are infinitely many of them.
None of this is used here.
We will need however to consider the case $|\sigma|\le8$ by an ad hoc argument, either
by brute force enumeration or by a more selective argument (which still involves a substantial number of cases);
see \Cref{ssec:proofcover}.

\subsection{Critical permutations}\label{ssec:crit}

\begin{definition}\label{def:crit}
A permutation is \emph{critical} if it is not \psm but every proper
pattern of it is; equivalently, the criticals are the minimal elements,
in the pattern order, of the complement of $\PA$.
\end{definition}

\noindent
Thus, by \eqref{eq:PAhered} a permutation is \psm if and only if it does not contain a critical pattern.

\Cref{def:PA} , \eqref{eq:goodloc} and \eqref{eq:PAhered} give obvious constraints on critical permutations, which we
will use repeatedly:
\begin{equation}\label{eq:critbasic}
\text{a critical permutation is indecomposable, \cb, plateau-free and not a tower.}
\end{equation}
For the remainder of this section we will prove the following statement.

\begin{theorem}[the cover graph of a critical permutation]\label{thm:covercount}
Let $\sigma$ be a critical permutation of size $n$. Then $\cgr_\sigma$ is
connected but not a pseudotree, i.e.,
\begin{equation}\label{eq:covercount}
   |\Cov(\sigma)| \;\ge\; n+1.
\end{equation}
Consequently $\sigma$ is not \prig, by \Cref{lem:countmm}.
\end{theorem}

The proof has two halves. The sizes
$n \ge 9$ are handled by the block structure of the inversion graph,
developed in \Cref{ssec:invgraph,ssec:critblocks} and summarized in
\Cref{prop:leaves}, which produces two distinct cycles in $\cgr_\sigma$;
the sizes $n \le 8$ are the computation recorded in \Cref{prop:critsmall}.
The two halves are put together in \Cref{ssec:proofcover}.

\begin{remark}[the count is exact, and the possible shapes]\label{rem:covexact}
It is easy to see that for every critical permutation $\sigma$ of size $n$ we have $|\Cov(\sigma)|\le n+1$.
Indeed, $\sigma - p$ lies in $\PA$ for every point
$p$, so that $\cgr_{\sigma-p}$ is a pseudoforest by \Cref{prop:PApf} and hence
\[
   |\Cov(\sigma - p)| \;\le\; n-1 \qquad\text{for every point } p .
\]
Summing \eqref{eq:covdelete-count} over $p$ we get
\[
(n-2) |\Cov(\sigma)| \le\sum_p |\Cov(\sigma - p)| \le n(n-1)
\]
so that
\[
   |\Cov(\sigma)| \;\le\; \frac{n(n-1)}{n-2} \;=\; n+1+\frac{2}{n-2} .
\]
Since $n \ge 5$ by \Cref{rem:psm45} we get the required inequality.

Thus, once we know that the connected graph $\cgr_\sigma$ is not a pseudotree, the inequality \eqref{eq:covercount} is in fact an
equality, i.e., $\cyc(\cgr_\sigma) = 2$. Moreover $\cgr_\sigma$ has minimum degree at least two, since by
\eqref{eq:covdelete-count}
\[
   \deg_{\cgr_\sigma}(p) \;\ge\; |\Cov(\sigma)| - |\Cov(\sigma - p)|
   \;\ge\; (n+1) - (n-1) \;=\; 2
   \qquad\text{for every point } p .
\]
A connected graph with $n$ vertices, $n+1$ edges and minimum degree at
least two is a theta graph (two vertices joined by three internally
disjoint paths), or two cycles sharing a single vertex, or a
\emph{dumbbell} (two vertex-disjoint cycles joined by a path). Hence
$\cgr_\sigma$ is necessarily one of these three graphs. All three occur
(\Cref{rem:smallshape}), but only the dumbbell for $n \ge 9$ (\Cref{cor:dumbbell}).
\end{remark}

\subsection{The inversion graph}\label{ssec:invgraph}

\begin{definition}\label{def:invgraph}
The \emph{inversion graph} $\igr_\sigma$ of $\sigma \in S_n$ has as vertices
the $n$ points of $\sigma$, with an edge between $i < j$ if and only if
$(i,j)$ is an inversion; by \eqref{eq:dictUV} its edge set is
$\XX_\sigma$.
\end{definition}

Non-adjacency of two points means that they form an ascent pair; a set of
points is a clique exactly when it is a decreasing subsequence, and an
independent set exactly when it is an increasing one. The induced
subgraph on a point set $\ptst$ is the inversion graph of $\std(\sigma|_{\ptst})$;
in particular $\igr_{\sigma - e} = \igr_\sigma - e$. The graph $\cgr_\sigma$
of \Cref{def:covmm} is a spanning subgraph of $\igr_\sigma$; being a Hasse
diagram, it is triangle-free \cite{MR2376109}*{\S4.1.1}. A witness
against an inversion is a common neighbor of its two endpoints, so
\begin{equation}\label{eq:notriangle}
\text{an edge of $\igr_\sigma$ lying in no triangle is an edge of
$\cgr_\sigma$;}
\end{equation}
this applies in particular to an edge lying in no cycle at all.

The following is standard.\footnote{For \labelcref{dict:comp} and \labelcref{dict:cover} see also
\cite{MR1317469}*{Lemma 1.10}.}

\begin{lemma}[the graph dictionary]\label{lem:dict}
Let $\sigma \in S_n$.
\begin{enumerate}
\item\label{dict:comp} If $\sigma = \gamma_1 \oplus \cdots \oplus
\gamma_k$ is the decomposition into indecomposable summands, the
connected components of
$\igr_\sigma$ are exactly the point sets underlying $\gamma_1,\dots,
\gamma_k$, and they are linearly ordered, each earlier one lying entirely
southwest of each later one. In particular $\igr_\sigma$ is connected if and
only if $\sigma$ is indecomposable.
\item\label{dict:cover} The spanning subgraph $\cgr_\sigma$ has the same
connected components as $\igr_\sigma$. In particular $\cgr_\sigma$ is
connected if and only if $\sigma$ is indecomposable, in which case
$|\Cov(\sigma)| \ge n-1$.
\item\label{dict:corner} For each of the four corners $\crn$, the window of
$\crn$ is the neighborhood of $\crn$ in $\igr_\sigma$. Hence $\crn$ is a \good corner
if and only
if $\crn$ is a \emph{simplicial} vertex, that is, its neighborhood is a
clique.
\item\label{dict:sumorder} Suppose that $\sigma_1 \ne 1$, that is, that
the \Wcr and \Scr corners are distinct, as is the case when $\sigma$ is
indecomposable and $n \ge 2$. Then for every point $v$ at least one of
these two corners survives in $\sigma - v$, and the indecomposable
summand of $\sigma - v$ containing it is the first one.
\end{enumerate}
Moreover the Klein four-group acts by graph isomorphisms, permuting the
four corners.
\end{lemma}

\begin{proof}[Proof of part \labelcref{dict:sumorder}]
The two corners being distinct, deleting the single point $v$ leaves at
least one of them. By \labelcref{dict:comp} the first summand of $\sigma
- v$ occupies an initial interval of positions carrying an initial
interval of values. If the \Wcr corner survives, its position is smaller
than all others; if only the \Scr corner does, its value is. Either way
its summand comes first.
\end{proof}

\paragraph{Blocks.\protect\footnote{In this section ``block'' always refers to the
inversion graph $\igr_\sigma$ in the sense defined here, rather than the blocks
of an inflation (\Cref{lem:eleminf}, \Cref{prop:C0}), which play no role here.}}
We use the standard vocabulary. A \emph{cut vertex} of a connected graph
is a vertex whose removal disconnects it; a \emph{block} is a maximal
connected subgraph without a cut vertex of its own, that is, a maximal
$2$-connected subgraph, or an edge whose removal disconnects (a
\emph{bridge}), or an isolated vertex. Two blocks meet in at most one
vertex, which is then a cut vertex; a non-cut vertex lies in exactly one
block; the edge sets of the blocks partition the edge set; every cycle
lies in a single block. The blocks and the cut vertices form the
\emph{block tree}, a block being joined to each cut vertex it contains. A
\emph{leaf block} contains exactly one cut vertex, and a connected graph
with more than one block has at least two leaf blocks.

Since the vertices of $\igr_\sigma$ are the points of $\sigma$, we transfer
the terminology to an indecomposable permutation $\sigma$: a \emph{cut point}
of $\sigma$ is a cut vertex of $\igr_\sigma$, that is, by
\Cref{lem:dict}\labelcref{dict:comp}, a point $e$ with $\sigma - e$
decomposable. Blocks are \emph{induced} subgraphs, hence determined by
their vertex sets, and we identify them with these: a \emph{block} of
$\sigma$ is a set $B$ of points of $\sigma$ which is the vertex set of a
block of $\igr_\sigma$. When the graph itself is needed we write
$\igr_\sigma[B]$ for the subgraph of $\igr_\sigma$ induced on $B$, so that
$\igr_\sigma[B]$ is $2$-connected, or a single edge, or a single point. Note
that $\igr_\sigma[B] = \igr_{\std(\sigma|_B)}$ is the inversion graph of the
pattern of $B$. The cover graph, by contrast, is not inherited by patterns.
Nonetheless, the following lemma says that it is inherited by blocks, and it is what makes the argument of this
section local.

\begin{lemma}[block-locality of minimal inversions]\label{lem:blockloc}
Let $B$ be a block of $\sigma$ and let $\{u,v\}$ be an edge of
$\igr_\sigma$ with $u,v \in B$. Then
\[
   \{u,v\} \in \Cov(\sigma)
   \iff
   \{u,v\} \in \Cov\bigl(\std(\sigma|_B)\bigr) .
\]
\end{lemma}

\begin{proof}
A witness lying in $B$ is a witness for $\sigma$ as well, which gives the
implication $\Rightarrow$. Conversely, let $w$ be a witness against
$\{u,v\}$ in $\sigma$. As in \eqref{eq:notriangle}, $w$ is a common
neighbor of $u$ and $v$, so $u,v,w$ span a triangle of $\igr_\sigma$.
Every cycle lies in a single block, and the edge sets of the blocks
partition the edge set, so that block is the one containing $\{u,v\}$,
namely $B$. Hence $w \in B$ and $\{u,v\} \notin
\Cov(\std(\sigma|_B))$.
\end{proof}

\subsection{The block structure of a critical
permutation}\label{ssec:critblocks}

Throughout this subsection $\sigma$ is critical of size $n$; in particular
$\sigma$ is indecomposable by \eqref{eq:critbasic}, so that $\igr_\sigma$
is connected and the vocabulary of blocks applies. The goal is
\Cref{prop:leaves}: as long as no one-point deletion of $\sigma$ is a tower and
$\igr_\sigma$ is not $2$-connected --- both hold for
$n \ge 9$ --- $\sigma$ has exactly two leaf blocks, of at most five
points each, and they carry the two \cores.
Nothing is asserted about the blocks in between. The main tools are the
\cores of \Cref{def:cores} and \Cref{lem:swcore}.

\begin{lemma}[the \core{} lies in a block]\label{lem:core}
Let $\sigma$ be indecomposable with an \SWcr \core{} $\Cm$. Then $\Cm$ is
contained in a unique block of $\sigma$; dually for $\Cp$. If $\sigma$ is critical, \Cref{lem:swcore} (with \eqref{eq:critbasic}) provides an \SWcr
\core{} and an \NEcr \core; we fix a choice $\Cm$, $\Cp$ of them once and for
all (they need not be unique in general, but see \Cref{prop:leaves}\labelcref{lv:attach})
and write $\Rm$, $\Rp$ for the blocks containing them.
\end{lemma}

\begin{proof}
The inversion graph of $3412$ is a four-cycle and that of $4231$ is a
four-clique less one edge, both $2$-connected; a $2$-connected subgraph
lies in a single block.
\end{proof}

From \Cref{def:PA} and \Cref{lem:dict} we immediately get:
\begin{lemma}[exit trichotomy]\label{lem:exit3}
Let $\sigma$ be critical and $d$ any point. Then at least one of the
following holds.
\begin{enumerate}
\renewcommand{\theenumi}{M\arabic{enumi}}%
\renewcommand{\labelenumi}{(\theenumi)}%
\item\label{M:cut} The point $d$ is a cut point of $\sigma$.
\item\label{M:corner} The permutation $\sigma - d$ has a \good corner.
\item\label{M:tower} The permutation $\sigma - d$ is a tower.
\end{enumerate}
\end{lemma}

The following is elementary.
\begin{lemma}[plateau loss]\label{lem:plateaubound}
Let $\sigma$ be plateau-free and $d$ a point of $\sigma$. Then every
plateau pair of $\sigma - d$ is of exactly one of two kinds: its points
straddle $d$ in position, the deletion having made them adjacent, or they
straddle $d$ in value, the deletion having made their values consecutive.
A pair of the first kind determines the position of $d$ and a pair of the
second its value; so there is at most one of each, and in particular
\begin{equation}\label{eq:plateaubound}
   \sigma - d \text{ has at most two plateau pairs, and }
   |q(\sigma - d)| \ge |\sigma| - 3 .
\end{equation}
\end{lemma}

\begin{proof}
Let $\{x,y\}$ be a plateau pair of $\sigma - d$, with $x$ the earlier
point. Since $\sigma$ is plateau-free it is not a plateau pair of
$\sigma$, so $d$ separates $x$ and $y$ either in position or in value ---
these are the two kinds --- and in either case $d$ is the only point of
$\sigma$ that does so. Thus a pair of the first kind consists of the two
neighbors of $d$ in position and a pair of the second of the two
neighbors of $d$ in value, which gives the uniqueness of each kind. The
two kinds are exclusive: were $d$ to separate $x$ and $y$ both in
position and in value, the points $x,d,y$ would occupy three consecutive
positions of $\sigma$ carrying three consecutive decreasing values,
a plateau of length three. Finally, contracting the plateaus of
$\sigma-d$ removes one point per plateau pair, so $|q(\sigma - d)| \ge
(|\sigma|-1) - 2$.
\end{proof}

\begin{corollary}[no deletion is a tower]\label{cor:notower}
Let $\sigma$ be critical with $n \ge 9$. Then no one-point deletion of
$\sigma$ is a tower; equivalently, mode \labelcref{M:tower} of
\Cref{lem:exit3} does not occur.\footnote{As it turns out, the corollary holds for $n=8$ as well.}
\end{corollary}

\begin{proof}
Suppose that $\tau = \sigma - d$ is a tower. Being critical, $\sigma$ is
plateau-free by \eqref{eq:critbasic}, so the quotient of $\tau$ has at
least $n-3$ points (\eqref{eq:plateaubound}), while the quotient of a
tower is one of the eight permutations of \Cref{def:towers}, of size at
most $6$. For $n \ge 10$ this is already a contradiction, so let $n = 9$.

Then the quotient of $\tau$ has exactly six points, so it is one of the
three tower quotients of that size. By symmetry, we may assume that the
quotient is $426153$ or $463152$, as the third one $526413$ is obtained
from the second one by $\rc$. Thus $\tau = q[\ell]$ for a composition
$\ell$ of $8$ into six positive parts, which leaves $2 \cdot 21 = 42$
possibilities for $\tau$. Moreover $\tau$ has exactly two plateau pairs,
one of each kind, and by \Cref{lem:plateaubound} the first fixes the
position of $d$ and the second its value. The two ways of matching the
pairs with the two kinds give exactly two candidates for each $\tau$,
hence $84$ in all, $72$ of them distinct.

A direct verification shows that $59$ of these $72$ permutations contain
as a proper pattern one of the five permutations of size $5$ listed in
\Cref{rem:psm45}, and that the remaining $13$ contain one of
\[
   361452,\quad 523614,\quad 524613 .
\]
None of these eight permutations is \psm. This refutes the criticality of $\sigma$.
\end{proof}

\begin{lemma}[core confinement]\label{lem:coreconf}
Let $\sigma$ be critical, and suppose that no one-point deletion of $\sigma$ is a tower.
(This holds for $n\ge9$ by \Cref{cor:notower}.)
Then the following hold.
\begin{enumerate}
\item\label{cc:conf} Every point of $\sigma$ which is not a cut point lies in
$\Cm \cup \Cp$.
\item\label{cc:not2} If $\igr_\sigma$ is $2$-connected then every point of
$\sigma$ lies in $\Cm \cup \Cp$, and in particular $n \le 8$.\footnote{As it turns out, at $n = 8$ every critical permutation has a
unique \SWcr \core{} and a unique \NEcr \core, and the two are disjoint, so that every point lies in exactly one of them;
the ones whose inversion graph is $2$-connected are $42863175$ and its
inverse $62518473$.}
\end{enumerate}
\end{lemma}

\begin{proof}
Let $w \notin \Cm \cup \Cp$ be a point which is not a cut point. Then $\sigma - w$ is
indecomposable (\Cref{lem:dict}\labelcref{dict:comp}). The four corners
lie in $\Cm \cup
\Cp$, so they survive and are the corners of $\sigma - w$; the window of
each in $\sigma - w$ is its window in $\sigma$ with $w$ removed, hence still
contains the ascent pair provided by
\Cref{def:cores}\labelcref{core:ascent}. So all four
corners of $\sigma - w$ are blocked, and modes \labelcref{M:cut} and
\labelcref{M:corner} of \Cref{lem:exit3} both fail at $w$; hence $\sigma - w$
is a tower, contrary to the hypothesis. This proves \labelcref{cc:conf}.

If $\igr_\sigma$ is $2$-connected then $\sigma$ has no cut point, so
every point lies in $\Cm \cup \Cp$ by \labelcref{cc:conf}, whence
$n \le |\Cm| + |\Cp| = 8$.
\end{proof}

We can now state the main structural statement about the blocks of $\sigma$.

\begin{proposition}[the \core{} blocks are the leaves]\label{prop:leaves}
Let $\sigma$ be critical, and suppose that no one-point deletion of $\sigma$ is
a tower and that $\igr_\sigma$ is not $2$-connected. Then $\Rm$ and $\Rp$ are distinct,
and they are exactly the leaf blocks of $\sigma$. Moreover:
\begin{enumerate}
\item\label{lv:attach} Either $\Rm = \Cm$ or $\Rm = \Cm \cup \{a_-\}$, where
$a_-$ is the unique cut point of $\sigma$ in $\Rm$, its
\emph{attachment}; in particular $4 \le |\Rm| \le 5$. Since $\Rm$ is the
leaf block containing the \Wcr and \Scr corners, this determines $\Cm$:
the \SWcr \core{} is unique, and so is the \NEcr \core.
\item\label{lv:sep} Every point outside $\Rm$ lies northeast of every
point of $\Rm \setminus \{a_-\}$.
\item\label{lv:notwitness} Suppose $|\Rm| = 5$, so that $\Rm = \Cm \cup
\{a_-\}$ with $a_- \notin \Cm$. Then either $\sigma_{a_-} > \sigma_c$ for every
$c \in \Cm$, or $\pos(a_-) > \pos(c)$ for every $c \in \Cm$. In
particular $a_-$ is not a witness against any inversion with both
endpoints in $\Cm$.
\item\label{lv:cycle} The four minimal inversions of the pattern of
$\Cm$ are edges of $\cgr_\sigma$, and they form a four-cycle $\cyclem$
of $\cgr_\sigma$ with vertex set $\Cm$.
Similarly for $\Cp$. The four-cycles $\cyclem$ and $\cyclep$ are distinct.
\end{enumerate}
\end{proposition}

\begin{proof}
\Cref{lem:coreconf} applies, by the first hypothesis. By the second,
$\sigma$ has at least two blocks and its block tree at least two leaf
blocks. Let $L$ be a leaf block. It contains a point
other than its unique cut point; that point is not a cut point, so it lies in no
other block, and it belongs to $\Cm \cup \Cp$ by \labelcref{cc:conf}; since $\Cpm \subseteq \Rpm$, its unique block is
$\Rm$ or $\Rp$, so $L \in \{\Rm,\Rp\}$. If $\Rm = \Rp$ there would be at
most one leaf block; so $\Rm \ne \Rp$, and these are exactly the two
leaves.

Being a leaf block, $\Rm$ contains exactly one cut point $a_-$. Its
other points are not cut points, hence lie in $\Cm \cup \Cp$; a point of $\Cp$
lies in $\Rp \ne \Rm$ and, not being a cut point, in no other block, so they lie
in $\Cm$. With $\Cm \subseteq \Rm$ this gives $\Rm \setminus \{a_-\}
\subseteq \Cm \subseteq \Rm$. For the uniqueness, note that any \SWcr \core{}
lies in a single block (\Cref{lem:core}) containing the \Wcr and \Scr corners,
which are distinct, and two blocks share at most one vertex; so that block is
$\Rm$, and the inclusions just proved, applied to this \core, leave no choice:
it is $\Rm$ if $|\Rm| = 4$ and $\Rm \setminus \{a_-\}$ if $|\Rm| = 5$.

\labelcref{lv:sep} Delete $a_-$. The set $\Rm \setminus \{a_-\}$ is
contained in a single component of $\igr_\sigma - a_-$, since
$\igr_\sigma[\Rm]$ is $2$-connected and $|\Rm| \ge 4$. That component
contains the surviving one of
the \Wcr and \Scr corners, hence is the first summand by
\Cref{lem:dict}\labelcref{dict:sumorder}; and since $\Rm$ is a leaf block with attachment
$a_-$, every point outside $\Rm$ lies in a later component, northeast of
all of $\Rm \setminus \{a_-\}$.

\labelcref{lv:notwitness} Suppose $|\Rm| = 5$, so that $\Rm \setminus
\{a_-\} = \Cm$. Being a cut point, $a_-$ lies in a block other than
$\Rm$, which meets
$\Rm$ in $a_-$ alone; let $x \ne a_-$ be a point of that block adjacent to
$a_-$ in $\igr_\sigma$, so that $x \notin \Rm$ and $\{a_-,x\}$ is an inversion. By
\labelcref{lv:sep}, $x$ lies northeast of every point
of $\Cm$:
\begin{equation}\label{eq:xNE}
   \pos(x) > \pos(c) \quad\text{and}\quad \sigma_x > \sigma_c
   \qquad \text{for every } c \in \Cm .
\end{equation}
If $\pos(a_-) < \pos(x)$, then the inversion $\{a_-,x\}$ gives
$\sigma_{a_-} > \sigma_x$, and \cref{eq:xNE} gives $\sigma_x > \sigma_c$ for every $c
\in \Cm$; so $\sigma_{a_-} > \sigma_c$ for every $c \in \Cm$. If $\pos(a_-) >
\pos(x)$, then \cref{eq:xNE} gives $\pos(a_-) > \pos(x) > \pos(c)$ for
every $c \in \Cm$.

These are the two alternatives. In the first, $a_-$ has a value above
every value of $\Cm$, and in the second, a position after every position
of $\Cm$; in neither case can $a_-$ lie strictly between two points of
$\Cm$ in both position and value.

\labelcref{lv:cycle} The pattern of $\Cm$ is $3412$ or $4231$
(\Cref{def:cores}\labelcref{core:pattern}). A direct
inspection gives
\begin{align*}
   \Cov(3412) &= \bigl\{(1,3),\ (1,4),\ (2,3),\ (2,4)\bigr\} ,\\
   \Cov(4231) &= \bigl\{(1,2),\ (1,3),\ (2,4),\ (3,4)\bigr\} ,
\end{align*}
a four-cycle in either case: for $3412$ all four inversions are minimal,
and for $4231$ the five inversions are these four together with $(1,4)$,
which has the witness of value $2$.

Fix one of these four pairs $\{u,v\}$, with $u,v \in \Cm$. No point of
$\Cm$ witnesses against it, by the displayed lists. By
\labelcref{lv:attach} the only other possible point of
$\Rm$ is the attachment $a_-$, and if it is there it does not witness
against $\{u,v\}$ either, by \labelcref{lv:notwitness}. Hence $\{u,v\}$ is a
minimal inversion of $\std(\sigma|_{\Rm})$, and therefore an edge of
$\cgr_\sigma$ by \Cref{lem:blockloc}, the two endpoints lying in the block
$\Rm$. The four pairs thus give a four-cycle of $\cgr_\sigma$ with vertex set $\Cm$.

Finally, the statement for $\Cp$ follows by applying $\rc$. The two
cycles have vertex sets $\Cm$ and $\Cp$, so they are distinct as soon as
$\Cm \ne \Cp$; and if $\Cm = \Cp$ then $\Rm$ and $\Rp$ would both be the
unique block containing that set (\Cref{lem:core}), against $\Rm \ne \Rp$.
\end{proof}

\begin{figure}[!htbp]
\centering
\begin{tikzpicture}[scale=0.52,baseline=(current bounding box.center)]
\draw[gray!25] (0.35,0.35) rectangle (9.65,9.65);
\draw[covedge] (1,3)--(3,1);
\draw[covedge] (1,3)--(5,2);
\draw[covedge] (2,4)--(3,1);
\draw[covedge] (2,4)--(5,2);
\draw[covedge] (4,6)--(5,2);
\draw[covedge] (4,6)--(8,5);
\draw[covedge] (6,8)--(8,5);
\draw[covedge] (6,8)--(9,7);
\draw[covedge] (7,9)--(8,5);
\draw[covedge] (7,9)--(9,7);
\node[permdot] at (1,3) {};
\node[permdot] at (2,4) {};
\node[permdot] at (3,1) {};
\node[permdot] at (4,6) {};
\node[permdot] at (5,2) {};
\node[permdot] at (6,8) {};
\node[permdot] at (7,9) {};
\node[permdot] at (8,5) {};
\node[permdot] at (9,7) {};
\draw[corebox] (0.55,0.55) rectangle (5.45,4.45);
\draw[corebox] (5.6,4.55) rectangle (9.45,9.45);
\node[cutdot] at (5,2) {};
\node[cutdot] at (4,6) {};
\node[cutdot] at (8,5) {};
\node[font=\scriptsize,below right=1pt and 2pt] at (5,2) {$a_-$};
\node[font=\scriptsize,below right=1pt and 2pt] at (8,5) {$a_+$};
\node[font=\scriptsize,above left=1pt and 2pt] at (4,6) {$4$};
\node[font=\footnotesize,anchor=east] at (0.15,2.5) {$\Rm=\Cm$};
\node[font=\footnotesize,anchor=west] at (9.85,7.0) {$\Rp=\Cp$};
\foreach \x in {1,...,9}
  {\node[gray!75,font=\tiny,anchor=north] at (\x,0.3) {\x};}
\end{tikzpicture}
\caption{The critical permutation $\sigma = 341628957$ of \Cref{tab:crit},
drawn with its minimal inversions. Here all ten inversions are minimal, so
$\igr_\sigma = \cgr_\sigma$ and one picture carries both graphs. The blocks of
$\sigma$ are the point sets $\{1,2,3,5\}$ and $\{6,7,8,9\}$ (indexed by position),
on each of which $\igr_\sigma$ induces a four-cycle,
together with $\{4,5\}$ and $\{4,8\}$, which induce bridges; the cut
points are the circled points $5,4,8$. The first two blocks are the leaf
blocks $\Rm$, $\Rp$ of \Cref{prop:leaves}, here equal to the \cores $\Cm$,
$\Cp$, with attachments $a_-=5$ and $a_+=8$; each has pattern $3412$. Every
point outside $\Rm$ lies northeast of every point of $\Rm\setminus\{a_-\}$, as
\Cref{prop:leaves}\labelcref{lv:sep} asserts. The graph
is a dumbbell: two four-cycles joined by the path
$5,4,8$, with $|\Cov(\sigma)| = 10 = n+1$.}
\label{fig:critical}
\end{figure}
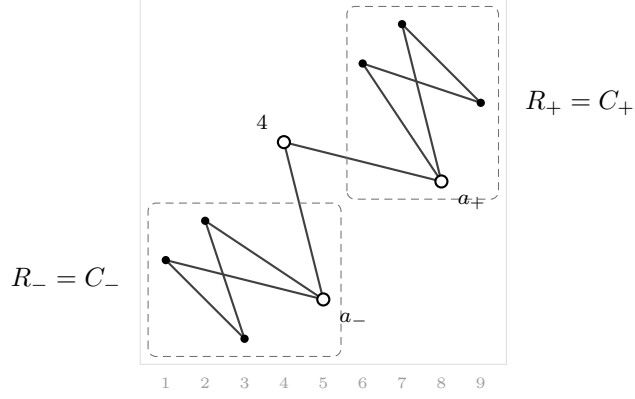

\subsection{Proof of \texorpdfstring{\Cref{thm:covercount}}{the theorem}}\label{ssec:proofcover}

The sizes $n \le 8$ are handled by an exhaustive computation.

\begin{proposition}[the small sizes]\label{prop:critsmall}
The critical permutations $\sigma$ with $|\sigma| \le 8$ are the
fifty-seven listed in \Cref{tab:crit}, falling into twenty-two orbits
of the Klein four-group; the smallest are of size $5$, and each of the fifty-seven
satisfies $|\Cov(\sigma)| = |\sigma| + 1$. The table also lists the
eighteen critical permutations of size $9$, which are not needed here.
\end{proposition}

\begin{proof}
By computer. The criticals are obtained by sieving: a permutation is
critical when it fails the recursion of \Cref{def:PA} while each of its
one-point deletions passes it, using \Cref{rem:algorithm};
\Cref{tab:census} records how few of the indecomposable \cb permutations
survive. For each of them $|\Cov(\sigma)|$ is computed directly.
\end{proof}

\begin{remark}[a smaller search]\label{rem:smallsearch}
For $n = 7,8$ the computation can be cut down substantially,
since we only have to consider permutations $\sigma$ that fail at least one of the
two hypotheses of \Cref{prop:leaves}. If $\sigma$ is plateau-free and some $\tau=\sigma - d$ is a tower, then
\Cref{lem:plateaubound} restricts the possibilities for $\tau$ and $\sigma$: for $n = 8$ there are $63$
towers of size $7$ whose quotient has at least five points, out of the $103$ of
\Cref{tab:census}, and they yield $243$ candidates; and for
$n = 7$ all $38$ towers of size $6$ qualify and they yield $194$. If
no $\sigma - d$ is a tower but $\igr_\sigma$ is $2$-connected, then every point of $\sigma$ lies in $\Cm \cup \Cp$ by
\Cref{lem:coreconf}\labelcref{cc:not2}. Such decompositions are
enumerated directly: choose the pattern of each \core and which of its
points is shared, then interleave the two position orders around the
shared point and, independently, the two value orders, subject to the
\Wcr and \Scr corners falling in $\Cm$ and the \Ncr and \Ecr corners in
$\Cp$. At $n = 8$ the two \cores are disjoint and there is nothing to
share: $1600$ decompositions, yielding $1273$ permutations. At $n = 7$
they meet in one point: $729$ decompositions, yielding $434$
permutations.
So we only need to verify \Cref{thm:covercount} for the critical ones among the explicitly generated
$243+1273$ (resp., $194+434$) candidates for $n=8$ (resp., $n=7$).
\end{remark}

\begin{proof}[Proof of \Cref{thm:covercount}]
For $n \le 8$ this is \Cref{prop:critsmall}. Let $n \ge 9$. The hypotheses of
\Cref{prop:leaves} hold, by \Cref{cor:notower} and
\Cref{lem:coreconf}\labelcref{cc:not2}, so $\cgr_\sigma$ --- connected by
\Cref{lem:dict}\labelcref{dict:cover} --- carries two distinct four-cycles by
\Cref{prop:leaves}\labelcref{lv:cycle}, and in particular is not a pseudotree.
\end{proof}

\begin{corollary}[the shape for $n \ge 9$]\label{cor:dumbbell}
Let $\sigma$ be critical of size $n \ge 9$. Then $\cgr_\sigma$ is a dumbbell:
the four-cycles $\cyclem$ and $\cyclep$ of \Cref{prop:leaves}, necessarily
vertex-disjoint, joined by a path through the remaining $n-8$ points.
\end{corollary}

\begin{proof}
By \Cref{rem:covexact}, $\cgr_\sigma$ is a theta graph, two cycles sharing a
vertex, or a dumbbell, and it contains the two distinct four-cycles $\cyclem$,
$\cyclep$. A theta graph with two distinct four-cycles has at most six
vertices, and two four-cycles sharing a vertex have seven.
\end{proof}

\begin{remark}[the shape of the cover graph for small sizes]\label{rem:smallshape}
A direct computation shows that all three shapes allowed by
\Cref{rem:covexact} occur among the critical permutations of
size at most $8$. For $n = 5$ all five are the theta graph $K_{2,3}$. For
$n = 6$ all thirteen are theta graphs, with paths of lengths $1,3,3$ (the
orbits of $351624$, $364152$, $426351$, $635241$) or $2,2,3$ (the orbits of
$361452$ and $462513$). For $n = 7$ the cover graph is two four-cycles
sharing a vertex for all but one orbit; for the orbit of $3517462$ it is a theta graph with paths of
lengths $1,3,4$, which contains a single four-cycle. For $n = 8$ all
eighteen are dumbbells, two four-cycles joined by a single edge; for $n = 9$
all eighteen are dumbbells with a path of length two, as in
\Cref{fig:critical}.
\end{remark}

\begin{table}[htbp]
\centering
\renewcommand{\arraystretch}{1.25}
\begin{tabular}{@{}l@{\quad}l@{\quad}l@{\quad}>{\raggedright\arraybackslash}p{0.66\textwidth}@{}}
\hline
$n$ & \# & orbits & the critical permutations of size $n$\\
\hline
$5$ & $5$ & $3$ &
   $\{34512,\,45123\}$, $\{45231,\,53412\}$, $\{52341\}$\\
$6$ & $13$ & $6$ &
   $\{351624\}$, $\{361452,\,523614\}$,
   $\{364152,\,461352,\,524613,\,526314\}$, $\{426351,\,624153\}$,
   $\{462513,\,536142\}$, $\{635241,\,642531\}$\\
$7$ & $21$ & $7$ &
   $\{3416725,\,3612745\}$, $\{3417562,\,3712564,\,4236715,\,6231745\}$,
   $\{3517462,\,3715264,\,4263715,\,6241735\}$,
   $\{3751264,\,4267315,\,4517362,\,6251734\}$, $\{4237561,\,7231564\}$,
   $\{4275631,\,5347261,\,7261453,\,7523164\}$, $\{5274163\}$\\
$8$ & $18$ & $6$ &
   $\{34172856,\,35127846\}$,
   $\{34182675,\,35128674,\,42371856,\,52317846\}$,
   $\{34186275,\,36128574,\,42731856,\,52417836\}$,
   $\{42381675,\,52318674\}$,
   $\{42386175,\,42831675,\,52418673,\,62318574\}$,
   $\{42863175,\,62518473\}$\\
$9$ & $18$ & $6$ &
   $\{341628957,\,351284967\}$,
   $\{341629785,\,351294786,\,423618957,\,523184967\}$,
   $\{341729685,\,351297486,\,426318957,\,524183967\}$,
   $\{423619785,\,523194786\}$,
   $\{423719685,\,426319785,\,523197486,\,524193786\}$,
   $\{427319685,\,524197386\}$\\
\hline
\end{tabular}
\caption{The critical permutations of size at most nine, braced into
orbits of the Klein four-group; \Cref{prop:critsmall} records the rows up
to $n = 8$. Each permutation listed has exactly $n+1$ minimal
inversions.}
\label{tab:crit}
\end{table}

\begin{table}[htbp]
\centering
\renewcommand{\arraystretch}{1.2}
\begin{tabular}{@{}r@{\quad}rr@{\qquad}rr@{\qquad}rr@{\qquad}rr@{}}
\hline
 & \multicolumn{2}{c}{indecomposable,} & \multicolumn{2}{c}{towers}
 & \multicolumn{2}{c}{not \psm} & \multicolumn{2}{c}{critical}\\
 & \multicolumn{2}{c}{\cb} & & & & & &\\
$n$ & \# & orbits & \# & orbits & \# & orbits & \# & orbits\\
\hline
$4$ & $2$        & $2$       & $2$   & $2$   & $0$        & $0$       & $0$  & $0$\\
$5$ & $16$       & $8$       & $11$  & $5$   & $5$        & $3$       & $5$  & $3$\\
$6$ & $124$      & $49$      & $38$  & $16$  & $86$       & $33$      & $13$ & $6$\\
$7$ & $1076$     & $320$     & $103$ & $34$  & $973$      & $286$     & $21$ & $7$\\
$8$ & $10\,384$  & $2\,812$  & $238$ & $76$  & $10\,146$  & $2\,736$  & $18$ & $6$\\
$9$ & $109\,486$ & $28\,043$ & $490$ & $142$ & $108\,996$ & $27\,901$ & $18$ & $6$\\
\hline
\end{tabular}
\caption{Indecomposable \cb permutations of size $n$: each
\psm one is a tower. Every column is stable under the Klein four-group,
whose orbits are counted alongside.}
\label{tab:census}
\end{table}

By \Cref{thm:covercount} and \Cref{thm:hered} we conclude:
\begin{corollary}\label{cor:converse}
Every \prig permutation is \psm.
\end{corollary}

Together with \Cref{thm:forward} this concludes the proof of \Cref{thm:mainA}, our main result.

\subsection{Complements}\label{ssec:complements}

\begin{remark}[forest-like permutations and the coefficient of $q$]\label{rem:forest}
The permutations for which the graph $\cgr_\sigma$ is acyclic were
characterized in \cite{MR2376109}*{Theorem 1.1} in terms of (a modified form of) pattern avoidance.\footnote{The
conventions in \cite{MR2376109} are opposite to ours: their permutation is $\wo\sigma$.}
Combined with the criterion of \cite{MR2264071}*{Proposition 2}, this condition
also characterizes the Schubert varieties that are locally factorial.

Also, the $q$ coefficient of the Kazhdan--Lusztig polynomials $P_{e,w}(q)$
is $\cyc(\cgr_w)$ (see \cite{MR1635681}, \cite{MR1317469}*{Proposition 1.7},
\cite{MR1279583}*{Corollary 1.8}). In particular it vanishes exactly when
$\cgr_w$ is a forest.

Neither class coincides with the \psm permutations: $3412$ is \psm but $\cgr_{3412}$ is a
four-cycle, while $\cgr_{456312}$ is a tree but $456312$ is not \psm (\Cref{rem:PApattern}).
Likewise, the permutations with $P_{e,w}(1) \le 2$, i.e., with $P_{e,w}(q) = 1$ or $1+q^h$ for some $h\ge1$,
were characterized in \cite{MR2515773}*{Theorem 1.1}, and purely in terms of pattern avoidance
in \cite{MR2515773}*{Theorem A.1} (in the appendix by Billey and Weed);
the tower quotients $463152$ and $526413$ are among the patterns excluded there, so the \psm permutations do
not form that class either.
\end{remark}

\begin{remark}[a pattern-avoidance reformulation]\label{rem:PApattern}
Combining \Cref{prop:PApf} with \eqref{eq:PAhered} and \Cref{thm:covercount} we get
\[
   \sigma \in \PA \iff \cgr_\tau \text{ is a pseudoforest for every pattern $\tau$ of $\sigma$.}
\]
Indeed, if $\sigma \notin \PA$, let $\tau$ be a pattern of $\sigma$ of minimal size outside $\PA$;
then $\tau$ is critical, so that $\cgr_\tau$ is
connected with $|\tau|+1$ edges by \Cref{thm:covercount}.
The condition on $\cgr_\sigma$ alone is not sufficient: for $\sigma = 456312$ the graph $\cgr_\sigma$ is a tree,
while deleting the entry $3$ gives $34512 \notin \PA$ (cf.\ \Cref{rem:psm45}), whose cover graph is $K_{2,3}$.
\end{remark}

\subsection*{Acknowledgements}
The first named author is indebted to Alberto M\'inguez for discussions in the early stages of this project
and to the University of Vienna for its hospitality.
The second named author is supported by the European Union (ERC, Function Fields, 101161909)
and is the Dr.~A.~Edward Friedmann Career Development Chair in Mathematics.

AI tools (Anthropic's Claude and OpenAI's ChatGPT) were used extensively in the preparation of this paper,
for exploratory computations, for drafting and checking arguments, and for adversarial reviews of successive versions.
All results were verified by the authors, who take full responsibility for the content.

%\bibliographystyle{amsalpha}
%\bibliography{../../Bibfiles/all}
%\end{document}
\def\cprime{$'$}

\end{document}